\documentclass{gtpart}     
\usepackage{pinlabel}  
\usepackage{graphicx}  
\usepackage[all]{xy}
\usepackage{amscd}
\usepackage[usenames]{color}
\usepackage{enumitem}
\usepackage{breqn}
\usepackage{mathtools}
\author{\'{I}lker S. Y\"{u}ce}
\givenname{\.{I}lker Sava\c{s}}
\surname{Y\"{u}ce}
\address{}
\email{ilkersyuce@gmail.com} 
\urladdr{}
  \title[Symmetric decompositions of free Kleinian groups]{Symmetric decompositions of free Kleinian groups \\ and hyperbolic displacements}
\volumenumber{} \issuenumber{} \publicationyear{} \papernumber{} \startpage{}
\endpage{}
\doi{} \MR{} \Zbl{} \received{} \revised{} \accepted{} \published{} \publishedonline{} \proposed{} \seconded{} \corresponding{} \editor{} \version{}
\newtheorem{thm}{Theorem}[section]    
\newtheorem{lemma}{Lemma}[section]
\newtheorem{prop}{Proposition}[section]
\newtheorem{conjecture}{Conjecture}[section]
\newtheorem*{log3}{Log 3 Theorem}
\theoremstyle{definition}
\newtheorem{defn}{Definition}[section]    
\newcommand{\tb}{\textbf}
\newcommand{\tr}{\textrm}

\newcommand{\ra}{\rightarrow}

\newcommand{\dis}{\displaystyle}
\newcommand{\dgamma}{\textrm{dist}(z_0,\ \gamma\cdot z_0)}                     
\newcommand{\dgammainv}{\textrm{dist}(z_0,\ \gamma^{-1}\cdot z_0)}             
\newcommand{\dpsi}{\textrm{dist}(z_0,\ \psi\cdot z_0)}                             
\newcommand{\dxiinv}{\textrm{dist}(z_0,\ \xi^{-1}\cdot z_0)}                   
\newcommand{\deta}{\textrm{dist}(z_0,\ \eta\cdot z_0)}                         
\newcommand{\dxieta}{\textrm{dist}(z_0,\ \xi\eta\cdot z_0)}                    
\newcommand{\detaxi}{\textrm{dist}(z_0,\ \eta^{-1}\xi^{-1}\cdot z_0)}          
\newcommand{\dxitwo}{\textrm{dist}(z_0,\ \xi^{2}\cdot z_0)}                    
\newcommand{\dxietanone}{\textrm{dist}(z_0,\ \xi\eta^{-1}\cdot z_0)}           
\newcommand{\hyp}{\mathbb{H}^3}

\newcommand{\deltad}{\Delta^{d-1}}
\newcommand{\tnr}{\textnormal}
\newcommand{\xs}{\textbf{x}^*}
\newcommand{\ys}{\textbf{y}^*}

\def\co{\colon\thinspace}
\newcommand{\hype}{\overline{\mathbb{H}}^3}

\begin{abstract}    
In this paper, it is shown that every point in the hyperbolic 3-space is moved at a distance at least $0.5\log\left(12\cdot 3^{k-1}-3\right)$ by one of the isometries of length at most $k\geq 2$ in a 2-generator Klenian group $\Gamma$ which is torsion-free, not co-compact and contains no parabolic. Also some lower bounds for the maximum of hyperbolic displacements given by  symmetric subsets of isometries in purely loxodromic finitely generated free Kleinian groups are conjectured. 
\end{abstract}

\begin{document}

\maketitle


\section{Introduction}

The following is sequel to Y\"{u}ce  \cite{Y} in which the machinery developed by Culler and Shalen \cite{CSParadox} that gives a lower bound for the maximum of the displacements under the generators of $\Gamma$ is extended to calculate a lower bound for the maximum of the displacements under any finite set of isometries in $\Gamma$ in connection with the solutions of certain minimax problems with a constraint. Here $\Gamma$ is a Kleinian group generated by two non--commuting isometries $\xi$ and $\eta$ of $\hyp$ that satisfies the hypothesis of the $\log 3$ Theorem which can be stated as follows:

\begin{log3} Suppose that $\Gamma$ is torsion--free, not co--compact and contains no parabolic. Let $\Gamma_1$ be the set $\{\xi,\eta\}$. Then we have $\max_{\gamma\in\Gamma_1}\{\textnormal{dist}(z_0,\gamma\cdot z_0)\}\geq 0.5\log 9$ for any $z_0\in\hyp$.\end{log3}

The use of this extension for the set of isometries $\Gamma_{\dagger}=\{\xi,\eta,\xi\eta\}\subset\Gamma$ implies, for instance, the fact that  $\max_{\gamma\in\Gamma_{\dagger}}\{\tnr{dist}(z_0,\gamma\cdot z_0)\}\geq 0.5\log (5+3\sqrt{2})$ for any $z_0\in\hyp$ \cite[Theorem 5.1]{Y}.

Since it has implications on Margulis numbers and volume estimates for a large class of closed hyperbolic $3$--manifolds, the $\log 3$ theorem is the main tool or motivation behind many deep results that connect the topology of hyperbolic 3--manifolds to their geometry (see Agol--Culler--Shalen \cite{ACS}, Culler--Hersonski--Shalen \cite{CHS}, Culler--Shalen \cite{CSParadox, CSBetti, CSMargulis}). For example, if
$M$ is a closed hyperbolic 3-manifold whose first Betti number $b_1(M)$ is at least $4$ and  the fundamental group $\pi_1(M)$ of $M$ has no subgroup isomorphic to the fundamental group of a genus two surface, then a generalization of the $\log 3$ theorem due to Anderson--Canary--Culler--Shalen \cite{ACCS} implies that $0.5\log 5$ is a Margulis number for $M$ and, $3.08$ is a lower bound for the volume of $M$. In \cite{CSParadox}, as well as proving the $\log 3$ Theorem, Culler and Shalen show that $0.5\log 3$ is a Margulis number and, $0.92$ is a lower bound for the volume of $M$ if $b_1(M)\geq 3$ and $\pi_1(M)$ has no subgroup of finite index.  In \cite{CHS},  Culler, Hersonsky and Shalen increase the previous lower bound for $M$ to $0.94$. It must be noted that the lower volume estimates computed in \cite{ACCS} and \cite{CSParadox} are improved by the works of Calegari--Meyerhoff--Milley \cite{DRP} and Milley \cite{Mil} in which a newer method called Mom technology was introduced. 

Aiming to set the ground work to investigate the further applications of the methods developed in \cite{ACS,CHS,ACCS,CSParadox,CSBetti,CSMargulis} to improve on the Margulis numbers and volume estimates for the classes of closed hyperbolic $3$-manifolds mentioned in the previous paragraph, in this paper we shall prove the following:

\medskip
{\bf Theorem \ref{thm4.1}}\qua
{\sl If $\Gamma_k$ is the set of isometries of length at most $k\geq 2$ in $\Gamma=\langle\xi,\eta\rangle$, then we have $\max_{\gamma\in\Gamma_k}\{\textnormal{dist}(z_0,\gamma\cdot z_0)\}\geq 0.5\log (12\cdot 3^{k-1}-3)$ for any $z_0\in\hyp$,} 

which is given as \fullref{thm4.1} in  \fullref{sec4}. This theorem can be considered as a generalization of the $\log 3$ theorem for symmetric subsets of isometries, which will be made clear in \fullref{S2}, in $\Gamma=\langle\xi,\eta\rangle$. 

In the rest of this manuscript, we shall assume, unless otherwise is stated, that $\Gamma=\langle\xi,\eta\rangle$ has the properties given in the $\log 3$ theorem. The expression $S_{\infty}$ will denote the boundary of the canonical compactification $\hype$ of $\hyp$. Note that $S_{\infty}\cong S^2$. The notation $\Lambda_{\Gamma\cdot z}$ will denote the limit set of $\Gamma$--orbit of $z\in\hyp$ on $S_{\infty}$. We will express the hyperbolic displacement of $z\in\hyp$ under the action of the isometry $\gamma\co\hyp\to\hyp$ by dist($z,\gamma\cdot z$). 

The proof of \fullref{thm4.1} requires the use of the strategy carried out by Culler and Shalen in the proof of the $\log 3$ theorem together with the solution method explained in \cite{Y} to certain minimax problems which produce the lower bounds given in the theorem. In particular, the proof entails the examination of two cases:
\begin{enumerate}[label=\roman*]
    \renewcommand{\labelenumi}{\roman{enumi}}
       \item\hspace{-.3cm}.\label{I}  when $\Gamma$ is geometrically infinite; that is, $\Lambda_{\Gamma\cdot z}=S_{\infty}$ for every $z\in\hyp$,
       \item\hspace{-.3cm}.\label{II} when $\Gamma$ is geometrically finite.
\end{enumerate}

Before we summarize the proof of \fullref{thm4.1} in each case, we introduce some notation. Let $z_0$ be a given point in $\hyp$. By  \cite[Proposition 9.2]{CSParadox}, the group $\Gamma=\langle\xi,\eta\rangle$ is free on the generators $\xi$ and $\eta$. As a consequence, $\Gamma=\langle\xi,\eta\rangle$ can be decomposed as
\begin{equation}\label{symm.d.}
\{1\}\cup\Psi_r^k\cup\bigcup_{\psi\in\Psi^k}J_{\psi}
\end{equation}
for each $k\geq 2$. In (\ref{symm.d.}) $\Psi^k_r$ is the set of all words of length less than $k$ and $\Psi^k$ is the set of all words of length exactly $k$ in $\Gamma$. The expression $J_{\psi}$ is the set of words in $\Gamma$ which starts with the word $\psi\in\Psi^k$.

The set $\Psi^k$, which can be considered as $\Psi_{\xi}^k\cup\Psi_{\eta^{-1}}^k\cup\Psi_{\eta}^k\cup\Psi_{\xi^{-1}}^k$, will be given an ordering. Above $\Psi_{\gamma}^k$ denotes the set of words in $\Psi^k$ starting with $\gamma\in\{\xi,\eta^{-1},\eta,\xi^{-1}\}$. 
From left to right, elements of $\Psi^k$ will be listed so that reduced words starting with $\xi$ are in the first group, words starting with $\eta^{-1}$ are in the second, words starting with $\eta$ are in the next and finally words starting with $\xi^{-1}$ are in the last group. In each group, from left to right, each letter of each reduced word will keep the same order, eg for $k=2$, we have $\Psi^2=\{\xi^2,\xi\eta^{-1},\xi\eta,\eta^{-1}\xi^{-1},\eta^{-1}\xi,\eta^{-2},\eta^2,\eta\xi^{-1},\eta\xi,\xi^{-1}\eta^{-1},\xi^{-1}\eta,\xi^{-2}\}$.

We enumerate the elements of $\Psi^k$ as follows: Assign $1$ to the first word of $\Psi^k$ which ends with $\xi$. Every other word which ends with $\xi$ in $\Psi^k$ will be assigned positive integers which are equivalent to $1$ in modulo $4$ in increasing order. Assign $2$ to the second word of $\Psi^k$ which ends with $\eta^{-1}$. For the other words which end with $\eta^{-1}$, assign positive integers in increasing order equivalent to $2$ in modulo $4$. Repeat this process with $3$ and $4$ for $\eta$ and $\xi^{-1}$, respectively. We shall abuse the notation and for each $k\geq 2$ we shall denote these enumerations with the mapping
\begin{equation}\label{Enum}
p\co\Psi^k\to I^k=\{1,2,\dots,4\cdot 3^{k-1}\}.
\end{equation}
For $\Psi^2$, for instance, we get $p\co\xi^2\mapsto 1$, $\xi\eta^{-1}\mapsto 2$, $\xi\eta\mapsto 3$, $\eta^{-1}\xi^{-1}\mapsto 4$, $\eta^{-1}\xi\mapsto 5$, $\eta^{-2}\mapsto 6$,   $\eta^{2}\mapsto 7$, $\eta\xi^{-1}\mapsto 8$, $\eta\xi\mapsto 9$, $\xi^{-1}\eta^{-1}\mapsto 10$, $\xi^{-1}\eta\mapsto 11$, and $\xi^{-2}\mapsto 12$. We shall also need the enumeration $p\co\Psi^3\to\{1,2,\dots,36\}$ given below for $k=3$:
\begin{equation*}
\scalebox{.8}{$
\begin{array}{llllllll}
\xi\xi\xi & \mapsto  1,                    & \eta^{-1}\xi^{-1}\eta^{-1} & \mapsto  10, & \eta\eta\eta  & \mapsto  19,             & \xi^{-1}\eta^{-1}\xi^{-1}  &\mapsto 28,\\
\xi\xi\eta^{-1} & \mapsto 2,           &  \eta^{-1}\xi^{-1}\eta& \mapsto  11,         & \eta\eta\xi^{-1}  & \mapsto  20,       & \xi^{-1}\eta^{-1}\xi &\mapsto  29,\\
\xi\xi\eta  & \mapsto  3,                 & \eta^{-1}\xi^{-1}\xi^{-1} & \mapsto 12,    & \eta\eta\xi  & \mapsto  21,               & \xi^{-1}\eta^{-1}\eta^{-1}  &\mapsto  30,\\
\xi\eta^{-1}\xi^{-1} & \mapsto  4,  & \eta^{-1}\xi\xi & \mapsto 13,                    & \eta\xi^{-1}\eta^{-1}& \mapsto  22, & \xi^{-1}\eta\eta  &\mapsto  31,\\
\xi\eta^{-1}\xi  &\mapsto  5,          & \eta^{-1}\xi\eta^{-1}& \mapsto 14,            & \eta\xi^{-1}\eta & \mapsto 23,        & \xi^{-1}\eta\xi^{-1}  &\mapsto  32,\\
\xi\eta^{-1}\eta^{-1} &\mapsto  6, & \eta^{-1}\xi\eta  & \mapsto 15,                  & \eta\xi^{-1}\xi^{-1} & \mapsto  24, & \xi^{-1}\eta\xi &\mapsto 33,\\
\xi\eta\eta & \mapsto  7,                & \eta^{-1}\eta^{-1}\xi^{-1} & \mapsto  16,  & \eta\xi\xi & \mapsto  25,                 & \xi^{-1}\xi^{-1}\eta^{-1}&  \mapsto  34,\\
\xi\eta\xi^{-1} & \mapsto  8,          & \eta^{-1}\eta^{-1}\xi & \mapsto  17,          & \eta\xi\eta^{-1} & \mapsto  26,       & \xi^{-1}\xi^{-1}\eta  &\mapsto  35,\\
\xi\eta\xi & \mapsto  9,                  & \eta^{-1}\eta^{-1}\eta^{-1} & \mapsto  18, & \eta\xi\eta & \mapsto  27,              & \xi^{-1}\xi^{-1}\xi^{-1}  &\mapsto  36.
\end{array}$}
\end{equation*}
For $i=1,2,3,4$ we have $p(\Psi_{\gamma})=I_i$ for $\gamma\in\{\xi,\eta,\eta^{-1},\xi^{-1}\}$, where, by abusing the notation, we let $I_i=\{(i-1)\cdot 3^{k-1}+1,\dots, i\cdot 3^{k-1}\}$.

Let us say $J_{S(\gamma)}=\bigcup_{\psi\in S(\gamma)}J_{\psi}$. Each decomposition, denoted by $\Gamma_{\mathcal{D}^k}$, in (\ref{symm.d.}) has certain group-theoretical relations $\gamma J_{s(\gamma)}=\Gamma-J_{S(\gamma)}$ for isometries $\gamma$ and $s(\gamma)$ in $\Psi_r^k\cup\Psi_k$ and $\Psi_k$, respectively, and a set $S(\gamma)$ of isometries  in $\Psi^k$.  
For example, for $\Gamma_{\mathcal{D}^2}$, one of the group-theoretical relations is
\begin{equation}\label{ex1}
\xi^2 J_{\xi^{-2}}=\Gamma-J_{\{\xi^{2},\xi\eta,\xi\eta^{-1}\}}.
\end{equation}
We shall use the notation $(\gamma, s(\gamma), S(\gamma))$ to denote a group--theoretical relations of $\Gamma_{\mathcal{D}^k}$ for any $k\geq 2$. So the relation in (\ref{ex1}) will be also denoted by $(\xi^2,\xi^{-2},\{\xi^{2},\xi\eta,\xi\eta^{-1}\})$. Another example for a group--theoretical relation for $\Gamma_{\mathcal{D}^2}$ is
\begin{equation}\label{ex3}
\xi^2 J_{\xi^{-1}\eta}=\Gamma-J_{\{\xi^2,\xi\eta^{-1},\eta\xi,\eta^2,\eta\xi^{-1},\xi^{-1}\eta,\xi^{-2},\xi^{-1}\eta^{-1},\eta^{-1}\xi,\eta^{-1}\xi^{-1},\eta^{-2}\}}.
\end{equation}
All of the group-theoretical properties of the decompositions $\Gamma_{\mathcal{D}^k}$ for $k\geq 2$ are given in  \fullref{lem22} in \fullref{S2}. Note that $s(\gamma)$ and $S(\gamma)$ denote different isometries and sets of isometries in (\ref{ex1}) and (\ref{ex3}) for the same isometry $\gamma$. A summary for the proof of   \fullref{thm4.1} goes as follows:

In the case (\ref{I}) $\Gamma=\langle\xi,\eta\rangle$ is geometrically infinite, we first prove the statement below:

\medskip
{\bf Theorem \ref{thm3.4}}\qua
{\sl Let $\Gamma=\langle\xi,\eta\rangle$ be a free, geometrically infinite Kleinian group without parabolics and, $\Gamma_{\mathcal{D}^k}$ be the decomposition of $\Gamma$ in (\ref{symm.d.}) for $k\geq 2$. If $z_0$ denotes a point in $\hyp$, then there is a family of Borel measures $\{\nu_{\psi}\}_{\psi\in\Psi^{k}}$ defined on $S_{\infty}$ for every integer $k\geq 2$ such that $(i)\ \ A_{z_0}=\sum_{\psi\in\Psi^{k}}\nu_{\psi}$;  $(ii)\ \ A_{z_0}(S_{\infty})=1$; and
$$(iii)\quad\dis{\int_{S_{\infty}}\left(\lambda_{\gamma,z_0}\right)^2d\nu_{s(\gamma)}=1-\sum_{\psi\in S(\gamma)}\int_{S_{\infty}} d\nu_{\psi}}$$ for each group--theoretical relation $(\gamma, s(\gamma), S(\gamma))$ of $\Gamma_{\mathcal{D}^k}$, where $A_{z_0}$ is the area measure based at $z_0$.}
\medskip

This theorem is given as \fullref{thm3.4} in \fullref{S2}. In the theorem, $(\lambda_{\psi,z_0})^2$ is the conformal expansion factor of $\psi_{\infty}$ measured in the round metric centered at $z_0$. 

Decompositions of $\Gamma=\langle\xi,\eta\rangle$ in (\ref{symm.d.}) will be used in part ($i$) of  \fullref{thm3.4} to decompose the area measure $A_{z_0}$  as a sum of Borel measures $\nu_{\psi}$ indexed by $\psi\in\Psi^k$. Each group--theoretical relation of $\Gamma_{\mathcal{D}^k}$ translates into a measure--theoretical relation among the Borel measures $\{\nu_{\psi}\}_{\psi\in\Psi^{k}}$ as described in part ($iii$) of \fullref{thm3.4}. In particular, each measure $\nu_{\psi}$ is transformed to the complement of certain measures in the set $\{\nu_{\gamma}\co\gamma\in\Psi^k-\{\psi\}\}$. 

For instance, the theorem above implies that $A_{z_0}(S_{\infty})=\sum_{\psi\in\Psi^2}\nu_{\psi}(S_{\infty})$ for $\Gamma_{\mathcal{D}^2}$ so that the Borel measure $\nu_{\xi^{-2}}$ is transformed to the complement of the sum of the measures $\nu_{\xi\eta}$, $\nu_{\xi^2}$ and $\nu_{\xi\eta^{-1}}$ by the group--theoretical property in (\ref{ex1}), which can also be expressed as 
\begin{equation}\label{ex1.1}
\int_{S_{\infty}}\lambda_{\xi^2,z_0}^2\ d\nu_{\xi^{-2}}=1-\nu_{\xi^2}(S_{\infty})-\nu_{\xi\eta}(S_{\infty})-\nu_{\xi\eta^{-1}}(S_{\infty}).
\end{equation}

Each displacement $\dgamma$ for $\gamma\in\Psi^k$ has a lower bound determined by a formula, proved originally in \cite{CSParadox} by Culler and Shalen and improved slightly in \cite{CSMargulis}, which involves the Borel measures in $\{\nu_{\psi}\}_{\psi\in\Psi^{k}}$. This formula is given as follows:
\begin{lemma}\label{lem1.2}(\cite[Lemma 5.5]{CSParadox}; \cite[Lemma 2.1]{CSMargulis}) 
Let $a$ and $b$ be numbers in $[0,1]$ which are not both equal to $0$ and are not both equal to $1$. Let $\gamma$ be a loxodromic isometry of $\hyp$ and let $z_0$
be a point in $\hyp$. Suppose that $\nu$ is a measure on $S_{\infty}$ such that
(i) $\nu\leq A_{z_0}$,  (ii) $\nu\left(S_{\infty}\right)\leq a$, (iii) $\int_{S_{\infty}}(\lambda_{\gamma,z_0})^2d\nu\geq b$.
Then we have $a>0$, $b<1$, and \begin{equation*} \textnormal{dist}(z_0,\gamma\cdot z_0)\geq 0.5\log\frac{b(1-a)}{a(1-b)}.\end{equation*}
\end{lemma}

For a given decomposition $\Gamma_{\mathcal{D}^k}$, assuming $0<\nu_{s(\gamma)}(S_{\infty})<1$ for every group--theoretical relation ($\gamma$, $s(\gamma)$, $S(\gamma)$), in \fullref{lem1.2}  if we let $\nu=\nu_{s(\gamma)}$, $a=\nu_{s(\gamma)}(S_{\infty})$ and $b=\int_{S_{\infty}}(\lambda_{\gamma,z_0})^2d\nu_{s(\gamma)}$, we obtain the lower bounds
\begin{equation}
\tr{dist}(z_0,\gamma\cdot z_0)\geq0.5\log\frac{\left(1-\dis{\sum_{\psi\in S(\gamma)}\int_{S_{\infty}} d\nu_{\psi}}\right)\left(1-\dis{\int_{S_{\infty}}
d\nu_{s(\gamma)}}\right)}{\left(\dis{\sum_{\psi\in S(\gamma)}\int_{S_{\infty}} d\nu_{\psi}}\right)\left(\dis{\int_{S_{\infty}} d\nu_{s(\gamma)}}\right)}\label{eqn2}
\end{equation}
by  \fullref{thm3.4}. The constant values inside the logarithms on the righthand side of the inequality in (\ref{eqn2}) can be considered as the values of certain functions, referred to as displacement functions for $\Gamma_{\mathcal{D}^k}$, defined on the set $\Delta^{d-1}$ of all points  in $\mathbb{R}^d$ whose entries add to $1$. Here $d=4\cdot 3^{k-1}$ is the cardinality of $\Psi^k$. 

As an example, assuming $0<\nu_{\psi}(S_{\infty})<1$ for $\psi\in\{\xi^2,\xi\eta,\xi\eta^{-1},\xi^{-2}\}$, by 
 \fullref{thm3.4} for $k=2$,  \fullref{lem1.2}, (\ref{ex1.1}) and the definition of $p$ for $k=2$, we have the displacement function
\[
f^{2}_{12}(\tb{x})=\frac{1-x_{1}-x_{2}-x_{3}}{x_{1}+x_{2}+x_{3}}\cdot\frac{1-x_{12}}{x_{12}}
\]
for the decomposition $\Gamma_{\mathcal{D}^2}$ such that $\dxitwo\geq 0.5\log f^{2}_{12}(\tb{m})$ for the point $\tb{m}=\left(\nu_{p(\psi)}(S_{\infty})\right)_{\psi\in\Psi^2}\in\Delta^{11}$.
More generally, $\tb{m}$ will denote in the rest of this paper the point in $\mathbb{R}^d$ whose entries formed by the total masses of the measures in $\{\nu_{\psi}\co\psi\in\Psi^k\}$ keeping the same ordering of $\Psi^k$. Note that for each decomposition $\Gamma_{\mathcal{D}^k}$,  \fullref{thm3.4} and  \fullref{lem1.2} produce as many displacement functions as the number of group--theoretical relations which are counted in  \fullref{lem22} in \fullref{S2}. 

For $k=2$, for instance, there are $48$ group--theoretical relations, and consequently, there is a set $\mathcal{G}^{2}$ of  $48$ displacement functions.  One of which is $f_{12}^2$ given above (See (\ref{deffs}), (\ref{defgs}) and (\ref{dispothergs}) for some others). These functions provide a lower bound for the maximum of hyperbolic displacements by the inequality 
$$
\max_{\gamma\in\Gamma_{2}}\left\{\dgamma\right\} \geq  \frac{1}{2}\log G^{2}(\tb{m})
                         \geq  \frac{1}{2}\log\left(\inf_{\tb{x}\in\Delta^{11}} G^{2}(\tb{x})\right)
$$
for $\Gamma_2=\Psi_r^2\cup\{\xi^2,\xi\eta^{-1},\xi\eta,\eta^{-1}\xi^{-1},\eta^{-1}\xi,\eta^{-2},\eta^2,\eta\xi^{-1},\eta\xi,\xi^{-1}\eta^{-1},\xi^{-1}\eta,\xi^{-2}\}$, where $G^2(\tb{x})=\max_{\tb{x}\in\Delta^{11}}\{f(\tb{x})\co f\in\mathcal{G}^2\}$.

Let $\mathcal{G}^k$ denote the set of all displacement functions for the decomposition $\Gamma_{\mathcal{D}^k}$ of $\Gamma=\langle\xi,\eta\rangle$. Explicit formulas of the functions in $\mathcal{G}^k$ are given in  \fullref{dispfunc} in \fullref{S2}. In general we shall prove the following statement.

\medskip
{\bf Theorem \ref{minG}}\qua 
{\sl If $G^{k}\co\Delta^{d-1}\to\mathbb{R}$ is the function defined by $\tb{x}\mapsto\max\{f(\tb{x}): f\in\mathcal{G}^k\}$, then we have $\inf_{\tb{x}\in\Delta^{d-1}}G^{k}(\tb{x})=12\cdot 3^{k-1}-3$ for $k\geq 2$,}
\medskip

which provides the lower bounds in  \fullref{thm4.1}. This is  \fullref{minG} in \fullref{S3}. 

To prove \fullref{minG}, we first introduce a subset $\mathcal{F}^k=\{f_1^k,\dots,f_d^k\}$ of displacement functions in $\mathcal{G}^k$.  A list of explicit formulas of the  functions in $\mathcal{F}^k=\{f_1^k,\dots,f_d^k\}$ are again given in  \fullref{dispfunc} in \fullref{S2}. For $\tb{x}\in\Delta^{d-1}$ let us say
\[
\begin{array}{lcl}
F^k(\tb{x})=\max\left(f_1^k(\tb{x}),\dots,f_{d}^k(\tb{x})\right) & \tnr{ and } &   \alpha_*=\inf_{\tb{x}\in\Delta^{d-1}}F^k(\tb{x}).
\end{array}
\]
We will prove in \S\ref{S3} that $  \alpha_*=\inf_{\tb{x}\in\Delta^{d-1}}G^{k}(\tb{x})$.  This is because  $  \alpha_*\leq\inf_{\tb{x}\in\Delta^{d-1}}G^{k}(\tb{x})$ by $\mathcal{F}^k\subset\mathcal{G}^k$. The reverse inequality follows from the fact that the functions in $\mathcal{F}^k$ take bigger values at the points in $\Delta^{d-1}$ that are significant to compute $\inf_{\tb{x}\in\Delta^{d-1}}G^{k}(\tb{x})$. 

The computation of   $\alpha_*$ follows from the following two properties of the function $F^k$:
\[
\begin{array}{ll}
(A)\ \    \alpha_*=\min_{\tb{x}\in\Delta^{d-1}}F^k(\tb{x}), & (B)\ \  F^k(\xs)=  \alpha_* \tnr{ for a  unique point $\tb{x}^*\in\Delta^{d-1}$}.
\end{array}
\]
The equality in (A) is proved in \fullref{lemtwo} in \fullref{S3} which uses the observation that each displacement function $f_i^k$ approaches to infinity on any sequence $\{\tb{x}_n\}\subset\Delta^{d-1}$ which limits on $\partial\Delta^{d-1}$. 

Proving the property in (B) takes most of the technical work in this paper. Using  \fullref{convex2}, \ref{convex1} and  \ref{convex3} we first show that each displacement function $f_i^k$ is strictly convex on a strictly convex subset $C_{f_i}$ of $\Delta^{d-1}$. These subsets are defined in (\ref{C1}) and (\ref{C2}). Next by  \fullref{unique1}, \ref{unique2}, \ref{unique3}, \ref{unique4}, \ref{unique5}, \ref{unique6}, \ref{unique7} and \ref{unique8} we establish in  \fullref{uniquef} that $\xs$ is in the intersection $C$ of all of these sets $C_{f_i}$ which is itself strictly convex. Then using a number of facts  \fullref{convex4}, \ref{convex5} and \fullref{convex6} from convex analysis we deduce that $F^k$ is strictly convex function on $C$ which implies  the uniqueness of $\tb{x}^*$. This is given in \fullref{unique9}.

Since $\xs$ is unique, it is fixed by every bijection of $\Delta^{d-1}$ preserving the set $\mathcal{F}^k$. This leads to the relations $x_i^*=x_j^*$ among the coordinates of $\xs$ for every distinct $i,j\in\{1,2,\dots,4\cdot 3^{k-1}\}$. A list of bijections and the details of the computations of the coordinates of $\xs$ and $  \alpha_*$ are given in \fullref{minF}. This completes the proof of  \fullref{minG} and consequently the proof of  \fullref{thm4.1} in the case (\ref{I}).

Let $\mathfrak{X}$ denote the character variety $PSL_2(\C)\times PSL_2(\C)$. In the case (\ref{II}) $\Gamma=\langle\xi,\eta\rangle$ is geometrically finite, we define the function $f_{z_0}^k\co\mathfrak{X}\to\R$ such that
\begin{displaymath}
f_{z_0}^k(\xi,\eta)=\max_{\psi\in\Psi^k}\{\dpsi\}
\end{displaymath}
for a fixed $z_0\in\hyp$. This function is continuous and proper.  We shall show that $f_{z_0}^k$ has no local minimum in $\mathfrak{GF}$ the set of pairs of isometries $(\xi,\eta)\in\mathfrak{X}$ such that $\langle\xi,\eta\rangle$ is free, geometrically finite and without parabolics. Since the set of $(\xi,\eta)$ such that $\langle\xi,\eta\rangle$ is free, geometricly infinite and without parabolics is dense in $\overline{\mathfrak{GF}}-\mathfrak{GF}$ and, every $(\xi,\eta)\in\mathfrak{X}$ with  $\langle\xi,\eta\rangle$ is free and without parabolic is in $\overline{\mathfrak{GF}}$ \cite[Propositions 9.3 and 8.2]{CSParadox},  geometrically finite case reduces to geometrically infinite case completing the proof of \fullref{thm4.1}.



\section{Symmetric Decompositions of Free Groups}\label{S2}

Let $\Gamma$ be a group which is free on a finite generating set $\Xi=\{\xi_1,\xi_2,\dots,\xi_n\}$. Let $\Xi^{-1}=\{\gamma^{-1}\co\gamma\in\Xi\}$. Every element $\gamma$ of $\Gamma$ can be written uniquely as a reduced word
$\psi_{1}\cdots\psi_{m}$, where $m\geq 0$, each $\psi_i$  is an element of $\Xi\cup\Xi^{-1}$, and $\psi_{i+1}\neq\psi_{i}^{-1}$ for $i=1,\dots ,m-1$. If
$n\leq m$ is a positive integer and $\gamma\neq 1$, we shall call $\psi_1\dots\psi_n$ the \textit{initial word of length $n$} of $\gamma$.

Let $\Psi^*$ be a finite set of words in $\Gamma$.  For each word $\psi\in\Psi^*$, let $J_{\psi}$ denote the set of non-trivial elements of $\Gamma$ that have the
initial word $\psi$. Depending on the number of elements in $\Xi$ and lengths of words in $\Psi^*$ there may be a set of words which are not contained in
any of $J_{\psi}$. This set will be called the \textit{residue set} of $\Psi^*$ and denoted by $\Psi_r^*$. For a given pair $(\Psi^*, \Psi_r^*)$ of finite sets of
words $\Psi^*$ and $\Psi_r^*$ in $\Gamma$, if we have $\Gamma=\{1\}\cup \Psi_r^*\cup\bigcup_{\psi\in\Psi^*}J_{\psi}$, then $\Gamma_{\mathcal{D}^*}$ with
$\mathcal{D}^*=(\Psi^*,\Psi_r^*)$ is a decomposition of $\Gamma$. In particular we shall be interested in the following decompositions:
\begin{defn}\label{fgsymd}
A decomposition $\Gamma_{\mathcal{D}^*}$ with $\mathcal{D}^*=(\Psi^*,\Psi_r^*)$ is symmetric if $\Psi^*$ and $\Psi_r^*$ are preserved by every bijection of $\Xi\cup\Xi$, ie if $\phi\co\Xi\cup\Xi^{-1}\to\Xi\cup\Xi^{-1}$ is a bijection, then $\phi(\Psi^*)=\Psi^*$ and $\phi(\Psi_r^*)=\Psi_r^*$.
\end{defn}

Let $\Gamma_k$ be the set of all isometries of length at most $k\geq 2$ in $\Gamma=\langle\xi_1,\dots,\xi_n\rangle$. Let $\Psi^k$ be the set of all isometries of length $k$ and $\Psi^l_r$ be the set of all non--identity isometries of length less than $k$. It is straightforward to see that
\begin{eqnarray*}
\Gamma   & = & \{1\}\cup \Psi_r^k\cup\bigcup\nolimits_{\psi\in\Psi^k} J_{\psi}
\end{eqnarray*}
for every $k\geq 2$. Therefore, $\Gamma_{\mathcal{D}^{k,n}}$ is a decomposition of $\Gamma=\langle\xi_1,\dots,\xi_n\rangle$ with $\mathcal{D}^{k,n}=(\Psi^k,\Psi_r^k)$, where
$\Gamma_k=\Psi^k\cup\Psi^l_r$. Note that $\Gamma_{\mathcal{D}^{k,n}}$ is symmetric for each $n,k\geq 2$. In the case $n=2$, we have the lemma below for the number of group--theoretical relations:
\begin{lemma}\label{lem22}
Let $\Gamma$ be a $2$--generator free group and, $\Gamma_{\mathcal{D}^k}$ be a symmetric decomposition of $\Gamma$ for $k\geq 2$. Then there are 
$R_k=4\cdot r_k\cdot 3^{k-1}$
many group--theoretical relations, where
\begin{equation}
\begin{array}{lcl}
r_k=1+\sum_{i=1}^{k-1}\left(1+2\sum_{j=1}^{\min\{i,k-i\}}3^{j-1}\right) & \tnr{or}, & r_k=\sum_{j=0}^ka_j,\label{l22p1}
\end{array}
\end{equation}
for $a_j= 1$ if $j=0,1$, $a_j=1+2\sum_{i=1}^{\lfloor j/2\rfloor}3^{i-1}$ if $2\leq j\leq k-1$, $a_j=2\sum_{i=1}^{\lfloor k/2\rfloor}3^{i-1}$ if $j=k$.
Above $\lfloor\cdot\rfloor$ denotes the floor function.
\end{lemma}
\begin{proof}
Let $\psi=\psi_1\psi_2\dots\psi_k$ be a reduced initial word in $\Psi^k$. Since we know that the isometries $\psi_1^{-1}$, $(\psi_1\psi_2)^{-1}$,\dots,$(\psi_1\psi_2\dots\psi_{k-1})^{-1}$ are all in $\Psi^l_r$ and $\psi^{-1}\in\Psi^k$, we count the group--theoretical relations $(\gamma, s(\gamma), S(\gamma))$  according to the number $i$ of cancellations in the product $\gamma s(\gamma)$ for $i=1,2,\dots,k-1$, where $s(\gamma)=\psi$ for $\gamma\in\Psi_r^k\cup\Psi^k$.

Note that the product $\psi_{i}^{-1}\cdots\psi_2^{-1}\psi_1^{-1}\psi$ gives a group--theoretical relation with $i$--cancellation. Assume that the product $\gamma\psi$ also gives a relation with $i$--cancellation. Then we have $\gamma=w\psi_{i}^{-1}\cdots\psi_2^{-1}\psi_1^{-1}$ for some $w\in\Psi^*_r$. Since we have to have $1\leq length(w\psi_{i}^{-1}\cdots\psi_2^{-1}\psi_1^{-1}\psi)\leq k$, we derive that $1\leq length(w)\leq\min\{i,k-i\}$ where $k\geq 2$. We have $2$ choices for the last letter of $w$ and $3$ choices for the rest of the letters of $w$. Therefore, there are $1+2\sum_{j=1}^{\min\{i,k-i\}}3^{j-1}$ group--theoretical relations with $i$--cancellation. Finally, the product $(\psi_1\dots\psi_{k-1}\psi_k)^{-1}\psi$ provides the group--theoretical relation with $k$--cancellation. There is only $1$ such relation. There are $4\cdot 3^{k-1}$ many choices for the isometry $\psi\in\Psi^k$. Thus, the first part of (\ref{l22p1})  follows.

For the second part of (\ref{l22p1}), let $j$ denote the length of the product $\gamma\psi$, for $0\leq j\leq k$. If $j$ is $0$ or $1$, then we derive that $\gamma=(\psi_1\psi_2\cdots\psi_k)^{-1}$ or $\gamma=(\psi_1\psi_2\cdots\psi_{k-1})^{-1}$, respectively. There is only $1$ group--theoretical relation for each case. Let $a_0^k=1$ and $a_1^k=1$. Assume that $j=k$. Let $i$ denote the number of cancellations in the product $\gamma\psi$. Since $j=length(\gamma)+k-2i$, we get $0<i\leq\lfloor k/2\rfloor$. Then we have $\gamma=w(\psi_1\psi_2\cdots\psi_{i})^{-1}$ for some $w\in\Psi_r^k$ such that $length(w)=i$. There are $2$ choices for the first letter of $w$ and $3$ choices for the rest. Consequently, there are $2\sum_{i=1}^{\lfloor k/2\rfloor}3^{i-1}$ many products $\gamma\psi$ whose length is $k$.

An argument analogous to the one above can be repeated for each $j\in\{2,\dots,k-1\}$ to count the number of products $\gamma\psi$ so that $length(\gamma\psi)=j$ with the exception that $w=1$. In those cases, we get $1$ additional product $\gamma\psi$, where $\gamma$ is $(\psi_1\psi_2\dots\psi_{k-j+1})^{-1}$ for each $j\in\{2,\dots,k-1\}$. Hence, we obtain the sum $1+2\sum_{i=1}^{\lfloor j/2\rfloor}3^{i-1}$ for $2\leq j\leq k-1$, which concludes the proof.
\end{proof}

As an example, we will list all of the group--theoretical relations for the symmetric decomposition $\Gamma_{\mathcal{D}^2}$. There are $R_2=48$ relations by \fullref{lem22}. First we list the relations $(\gamma,s(\gamma),S(\gamma))$ in  \fullref{table1} so that $\gamma s(\gamma)$ has length $0$. There are $12$ such relations. Note that those are the relations with $s(\gamma)=\gamma^{-1}$.
\begin{table}[h]
\begin{center}
\scalebox{.65}{
\begin{tabular}{|c|c|c|c|c|c|c|c|}
  \hline
&  $\gamma$ & $s(\gamma)$ & $S(\gamma)$ & & $\gamma$ & $s(\gamma)$ & $S(\gamma)$\\
  \hline
$1$ &  $\xi^{-2}$  &$\xi^{2}$ & $\{\xi^{-1}\eta,\xi^{-2},\xi^{-1}\eta^{-1}\}$  & $7$ &  $\eta^{-2}$  & $\eta^{2}$ & $\{\eta^{-1}\xi,\eta^{-1}\xi^{-1},\eta^{-2}\}$ \\ 
  \hline 
$2$ &  $\eta\xi^{-1}$ & $\xi\eta^{-1}$ & $\{\eta\xi,\eta^{2},\eta\xi^{-1}\}$   & $8$  &  $\xi\eta^{-1}$  & $\eta\xi^{-1}$ & $\{\xi^{2},\xi\eta,\xi\eta^{-1}\}$        \\ 
  \hline
$3$ &  $\eta^{-1}\xi^{-1}$  & $\xi\eta$ & $\{\eta^{-1}\xi,\eta^{-1}\xi^{-1},\eta^{-2}\}$  & $9$ & $\xi^{-1}\eta^{-1}$  & $\eta\xi$ & $\{\xi^{-1}\eta,\xi^{-2},\xi^{-1}\eta^{-1}\}$   \\ 
  \hline
$4$ &   $\xi\eta$  &$\eta^{-1}\xi^{-1}$ & $\{\xi^{2},\xi\eta,\xi\eta^{-1}\}$  & $10$ & $\eta\xi$  & $\xi^{-1}\eta^{-1}$ & $\{\eta\xi,\eta^{2},\eta\xi^{-1}\}$  \\ 
  \hline
$5$ &   $\xi^{-1}\eta $  & $\eta^{-1}\xi $ & $\{\xi^{-1}\eta,\xi^{-2},\xi^{-1}\eta^{-1}\}$ &  $11$ & $\eta^{-1}\xi$  & $\xi^{-1}\eta$ & $\{\eta^{-1}\xi,\eta^{-1}\xi^{-1},\eta^{-2}\}$  \\ 
  \hline
$6$ &   $\eta^{2}$  & $\eta^{-2}$ & $\{\eta\xi,\eta^{2},\eta\xi^{-1}\}$ & $12$ &  $\xi^2$  &$\xi^{-2}$ & $\{\xi^{2},\xi\eta,\xi\eta^{-1}\}$ \\ 
  \hline
\end{tabular}}
\caption{Group--theoretical properties of  $\Gamma_{\mathcal{D}^{2}}$ with $s(\gamma)=\gamma^{-1}$ or $length(\gamma s(\gamma))=0$.}\label{table1}
\end{center}
\end{table}

Next we give the group--theoretical relations $(\gamma,s(\gamma),S(\gamma))$ in  \fullref{table2} and \fullref{table3}  such that $\gamma s(\gamma)$ has length $1$ or $2$. There are $12$ and $24$ such relations, respectively.
\begin{table}[h!]
\begin{center}
\scalebox{.6}{
\begin{tabular}{|c|c|c|c|}
  \hline
 &  $\gamma$ & $s(\gamma)$ & $S(\gamma)$ \\
  \hline
1 &   $\xi^{-1}$  & $\xi^{2}$ & $\{\eta^{-1}\xi^{-1},\eta^{-1}\xi,\eta^{-2},\eta\xi,\eta^2,\eta\xi^{-1},\xi^{-1}\eta,\xi^{-2},\xi^{-1}\eta^{-1}\}$ \\
  \hline 
2 &   $\xi^{-1}$  & $\xi\eta^{-1}$ & $\{\xi^{2},\xi\eta^{-1},\xi\eta,\eta^2,\eta\xi^{-1},\eta\xi,\xi^{-1}\eta^{-1},\xi^{-1}\eta,\xi^{-2}\}$\\
  \hline
3 &   $\xi^{-1}$  & $\xi\eta$ & $\{\xi^{2},\xi\eta^{-1},\xi\eta,\eta^{-1}\xi^{-1},\eta^{-1}\xi,\eta^{-2},\xi^{-1}\eta^{-1},\xi^{-1}\eta,\xi^{-2}\}$ \\ 
  \hline
 4 &  $\eta$  & $\eta^{-1}\xi^{-1}$ & $\{\xi^{2},\xi\eta^{-1},\xi\eta,\eta^{-1}\xi^{-1},\eta^{-1}\xi,\eta^{-2},\eta^2,\eta\xi^{-1},\eta\xi\}$  \\ 
  \hline
5 &  $\eta$  & $\eta^{-1}\xi$ & $\{\eta^{-1}\xi^{-1},\eta^{-1}\xi,\eta^{-2},\eta^2,\eta\xi^{-1},\eta\xi,\xi^{-1}\eta^{-1},\xi^{-1}\eta,\xi^{-2}\}$ \\ 
 \hline
 6 &  $\eta$  &$\eta^{-2}$ & $\{\xi^{2},\xi\eta^{-1},\xi\eta,\eta^2,\eta\xi^{-1},\eta\xi,\xi^{-1}\eta^{-1},\xi^{-1}\eta,\xi^{-2}\}$  \\
  \hline
7 & $\eta^{-1}$  & $\eta^{2}$ & $\{\xi^{2},\xi\eta^{-1},\xi\eta,\eta^{-1}\xi^{-1},\eta^{-1}\xi,\eta^{-2},\xi^{-1}\eta^{-1},\xi^{-1}\eta,\xi^{-2}\}$ \\ 
  \hline
 8 &  $\eta^{-1}$ & $\eta\xi^{-1}$ & $\{\xi^{2},\xi\eta^{-1},\xi\eta,\eta^{-1}\xi^{-1},\eta^{-1}\xi,\eta^{-2},\eta^2,\eta\xi^{-1},\eta\xi\}$\\ 
  \hline
9 &  $\eta^{-1}$  & $\eta\xi$ & $\{\eta^{-1}\xi^{-1},\eta^{-1}\xi,\eta^{-2},\eta^2,\eta\xi^{-1},\eta\xi,\xi^{-1}\eta^{-1},\xi^{-1}\eta,\xi^{-2}\}$ \\
  \hline
10 & $\xi$  & $\xi^{-1}\eta^{-1}$ & $\{\xi^{2},\xi\eta^{-1},\xi\eta,\eta^2,\eta\xi^{-1},\eta\xi,\xi^{-1}\eta^{-1},\xi^{-1}\eta,\xi^{-2}\}$ \\
   \hline
11 & $\xi$  & $\xi^{-1}\eta$ & $\{\xi^{2},\xi\eta^{-1},\xi\eta,\eta^{-1}\xi^{-1},\eta^{-1}\xi,\eta^{-2},\xi^{-1}\eta^{-1},\xi^{-1}\eta,\xi^{-2}\}$\\
\hline
 12 &   $\xi$  & $\xi^{-2}$ & $\{\xi^{2},\xi\eta^{-1},\xi\eta,\eta^{-1}\xi^{-1},\eta^{-1}\xi,\eta^{-2},\eta^2,\eta\xi^{-1},\eta\xi\}$ \\ 
  \hline
\end{tabular}}
\caption{Group--theoretical properties of  $\Gamma_{\mathcal{D}^{2}}$ with $length(\gamma s(\gamma))=1$.}\label{table2}
\end{center}
\end{table}

\newpage

\begin{table}[h!]
\begin{center}
\scalebox{.7}{
\begin{tabular}{|c|c|c|c|c|c|c|c|}
  \hline
   & $\gamma$ & $s(\gamma)$ & $S(\gamma)$ & & $\gamma$  & $s(\gamma)$ & $S(\gamma)$ \\
  \hline
 1 &  $\eta\xi^{-1}$  &$\xi^2$ & $\Psi^2-\{\eta\xi\}$                         & 13 &  $\eta^{-1}\xi^{-1}$  & $\xi^2$ & $\Psi^2-\{\eta^{-1}\xi\}$\\
  \hline
2 &  $\eta^{-1}\xi^{-1}$  &$\xi\eta^{-1}$ & $\Psi^2-\{\eta^{-2}\}$     & 14 &  $\xi^{-2}$  & $\xi\eta^{-1}$ & $\Psi^2-\{\xi^{-1}\eta^{-1}\}$ \\ 
  \hline
 3 & $\xi^{-2}$  &$\xi\eta$ & $\Psi^2-\{\xi^{-1}\eta\}$                       & 15 &  $\eta\xi^{-1}$  & $\xi\eta$ & $\Psi^2-\{\eta^2\}$ \\
  \hline
4 &  $\eta^{2}$  &$\eta^{-1}\xi$ & $\Psi^2-\{\eta\xi\}$                          & 16 &    $\xi\eta$  &$\eta^{-1}\xi$ & $\Psi^2-\{\xi^2\}$  \\
  \hline
5 &  $\xi\eta$  &$\eta^{-2}$ & $\Psi^2-\{\xi\eta^{-1}\}$                     & 17 &  $\xi^{-1}\eta$  &$\eta^{-2}$ & $\Psi^2-\{\xi^{-1}\eta^{-1}\}$\\
  \hline 
6 &  $\xi\eta^{-1}$  &$\eta^{2}$ & $\Psi^2-\{\xi\eta\}$                       & 18 &   $\xi^{-1}\eta^{-1}$  &$\eta^{2}$ & $\Psi^2-\{\xi^{-1}\eta\}$ \\
  \hline
7 &  $\xi^{-1}\eta$  & $\eta^{-1}\xi^{-1}$ & $\Psi^2-\{\xi^{-2}\}$       & 19 &  $\eta^{2}$  &$\eta^{-1}\xi^{-1}$ & $\Psi^2-\{\eta\xi^{-1}\}$  \\
  \hline 
8 &  $\eta^{-2}$  & $\eta\xi$ & $\Psi^2-\{\eta^{-1}\xi\}$                   & 20 &  $\xi\eta^{-1}$  &$\eta\xi$ & $\Psi^2-\{\xi^2\}$ \\
  \hline
9 &  $\xi^{-1}\eta^{-1}$  &$\eta\xi^{-1}$ & $\Psi^2-\{\xi^{-2}\}$       & 21 & $\eta^{-2}$  &$\eta\xi^{-1}$ & $\Psi^2-\{\eta^{-1}\xi^{-1}\}$  \\
  \hline
10 &  $\eta^{-1}\xi$  &$\xi^{-1}\eta^{-1}$ & $\Psi^2-\{\eta^{-2}\}$    & 22 & $\xi^2$  &$\xi^{-1}\eta^{-1}$ & $\Psi^2-\{\xi\eta^{-1}\}$ \\
  \hline
11 &   $\xi^2$  &$\xi^{-1}\eta$ & $\Psi^2-\{\xi\eta\}$                     & 23 & $\eta\xi$  &$\xi^{-1}\eta$ & $\Psi^2-\{\eta^2\}$ \\
  \hline
12 & $\eta\xi$  &$\xi^{-2}$ & $\Psi^2-\{\eta\xi^{-1}\}$                     & 24 & $\eta^{-1}\xi$  &$\xi^{-2}$ & $\Psi^2-\{\eta^{-1}\xi^{-1}\}$\\
  \hline
\end{tabular}}
\caption{Group--theoretical properties of  $\Gamma_{\mathcal{D}^{2}}$ with $length(\gamma s(\gamma))=2$.}\label{table3}
\end{center}
\end{table}


Similarly, we also list some of the group--theoretical relations for the symmetric decomposition $\Gamma_{\mathcal{D}^3}$ as we shall need them in this section. By  \fullref{lem22} there are  in total $252$ group--theoretical relations. We list in the table below the ones that $\gamma s(\gamma)=1$.
\begin{table}[ht!]
\begin{center}
\scalebox{0.8}{
\begin{tabular}{|c|c|c|c|c|c|}
  \hline
 $\gamma$ & $s(\gamma)$ & $S(\gamma)$ & $\gamma$ & $s(\gamma)$ & $S(\gamma)$\\
  \hline
  $\xi^{-3}$  &$\xi^{3}$ & $\Psi_{\xi^{-1}}$                                           &  $\eta^{-3}$  & $\eta^{3}$ & $\Psi_{\eta^{-1}}$\\ 
  \hline
   $\eta\xi^{-2}$  &$\xi^{2}\eta^{-1}$ & $\Psi_{\eta}$                           & $\xi\eta^{-2}$ & $\eta^{2}\xi^{-1}$ & $\Psi_{\xi}$\\ 
  \hline
  $\eta^{-1}\xi^{-2}$  &$\xi^{2}\eta$ & $\Psi_{\eta^{-1}}$                    &  $\xi^{-1}\eta^{-2} $  & $\eta^{2}\xi $ & $\Psi_{\xi^{-1}}$  \\ 
  \hline
  $\xi\eta\xi^{-1}$  & $\xi\eta^{-1}\xi^{-1}$ & $\Psi_{\xi}$                  &  $\eta\xi\eta^{-1}$ & $\eta\xi^{-1}\eta^{-1}$ & $\Psi_{\eta}$\\ 
  \hline
  $\xi^{-1}\eta\xi^{-1}$  & $\xi\eta^{-1}\xi$ & $\Psi_{\xi^{-1}}$          & $\eta^{-1}\xi\eta^{-1}$  & $\eta\xi^{-1}\eta$ & $\Psi_{\eta^{-1}}$\\ 
  \hline
  $\eta^2\xi^{-1}$ & $\xi\eta^{-2}$ & $\Psi_{\eta}$                              &  $\xi^2\eta^{-1}$  &$\eta\xi^{-2}$ & $\Psi_{\xi}$ \\ 
  \hline
  $\eta^{-2}\xi^{-1}$  & $\xi\eta^{2}$ & $\Psi_{\eta^{-1}}$                   & $\xi^{-2}\eta^{-1}$  & $\eta\xi^{2}$ & $\Psi_{\xi^{-1}}$\\ 
  \hline
   $\xi\eta^{-1}\xi^{-1}$ & $\xi\eta\xi^{-1}$ & $\Psi_{\xi}$                  & $\eta\xi^{-1}\eta^{-1}$  &$\eta\xi\eta^{-1}$ & $\Psi_{\eta}$ \\ 
  \hline
  $\xi^{-1}\eta^{-1}\xi^{-1}$  & $\xi\eta\xi$ & $\Psi_{\xi^{-1}}$           & $\eta^{-1}\xi^{-1}\eta^{-1}$  & $\eta\xi\eta$ & $\Psi_{\eta^{-1}}$ \\ 
  \hline
  $\eta\xi\eta$  & $\eta^{-1}\xi^{-1}\eta^{-1}$ & $\Psi_{\eta}$             &  $\xi\eta\xi$  & $\xi^{-1}\eta^{-1}\xi^{-1}$ & $\Psi_{\xi}$\\ 
  \hline
  $\eta^{-1}\xi\eta$  & $\eta^{-1}\xi^{-1}\eta$ & $\Psi_{\eta^{-1}}$     & $\xi^{-1}\eta\xi$  & $\xi^{-1}\eta^{-1}\xi$ & $\Psi_{\xi^{-1}}$ \\ 
  \hline
   $\xi^2\eta$  &$\eta^{-1}\xi^{-2}$ & $\Psi_{\xi}$                                &  $\eta^2\xi$ & $\xi^{-1}\eta^{-2}$ & $\Psi_{\eta}$  \\ 
  \hline
   $\xi^{-2}\eta$  & $\eta^{-1}\xi^{2}$ & $\Psi_{\xi^{-1}}$                     &  $\eta^{-2}\xi$  & $\xi^{-1}\eta^2$ & $\Psi_{\eta^{-1}}$ \\ 
  \hline
   $\eta\xi^{-1}\eta$  &$\eta^{-1}\xi\eta^{-1}$ & $\Psi_{\eta}$              & $\xi\eta^{-1}\xi$ & $\xi^{-1}\eta\xi^{-1}$ & $\Psi_{\xi}$\\ 
  \hline 
  $\eta^{-1}\xi^{-1}\eta $  & $\eta^{-1}\xi\eta $ & $\Psi_{\eta^{-1}}$    & $\xi^{-1}\eta^{-1}\xi$  & $\xi^{-1}\eta\xi$ & $\Psi_{\xi^{-1}}$\\ 
  \hline
  $\xi\eta^{2}$ & $\eta^{-2}\xi^{-1}$ & $\Psi_{\xi}$                                & $\eta\xi^{2}$  & $\xi^{-2}\eta^{-1}$ & $\Psi_{\eta}$ \\ 
  \hline
  $\xi^{-1}\eta^{2}$  & $\eta^{-2}\xi$ & $\Psi_{\xi^{-1}}$                       & $\eta^{-1}\xi^{2}$  & $\xi^{-2}\eta$ & $\Psi_{\eta^{-1}}$\\ 
  \hline
$\eta^3$ & $\eta^{-3}$ & $\Psi_{\eta}$                                                   & $\xi^3$  &$\xi^{-3}$ & $\Psi_{\xi}$\\ 
  \hline
\end{tabular}}
\caption{Group--theoretical properties of  $\Gamma_{\mathcal{D}^{3}}$ with $s(\gamma)=\gamma^{-1}$ or $length(\gamma s(\gamma))=0$.}\label{Table4}
\end{center}
\end{table}

Under the hypothesis of the $\log 3$ theorem, we know that $\Gamma=\langle\xi,\eta\rangle$ is a free group on the generators $\xi$ and $\eta$ \cite[Proposition 9.2]{CSParadox}. For the symmetric decompositions of $\Gamma=\langle\xi,\eta\rangle$ we have the following statement:
\begin{thm}\label{thm3.4}
Let $\Gamma=\langle\xi,\eta\rangle$ be a free, geometrically infinite Kleinian group without parabolics and, $\Gamma_{\mathcal{D}^k}$ be a symmetric decomposition of $\Gamma$ for $k\geq 2$. If $z_0$ denotes a point in $\hyp$, then there is a family of Borel measures $\{\nu_{\psi}\}_{\psi\in\Psi^{k}}$ defined on $S_{\infty}$ such that $(i)\ \ A_{z_0}=\sum_{\psi\in\Psi^{k}}\nu_{\psi}$;  $(ii)\ \ A_{z_0}(S_{\infty})=1$; and
$$(iii)\quad\dis{\int_{S_{\infty}}\left(\lambda_{\gamma,z_0}\right)^2d\nu_{s(\gamma)}=1-\sum_{\psi\in S(\gamma)}\int_{S_{\infty}} d\nu_{\psi}}$$ for each group--theoretical relation $(\gamma, s(\gamma), S(\gamma))$ of $\Gamma_{\mathcal{D}^k}$, where $A_{z_0}$ is the area measure based at $z_0$.
\end{thm}
\begin{proof}
As in the proof of \cite[Lemma 3.3]{Y}, we follow the same scheme given in the proof of  \cite[Lemma 5.3]{CSParadox}. Therefore we shall provide a proof sketch. In particular this proof involves $\Gamma$--invariant $D$--conformal densities, first constructed by Patterson \cite{SJP} and extensively studied by Sullivan \cite{Su1, Su2}. Interested readers may refer to \cite{SJP,Su1,Su2,N,CSParadox} for details on $\Gamma$--invariant $D$--conformal densities and their use in the context of this paper.

The group $\Gamma$ acts freely on $\hyp$. The symmetric decomposition $\Gamma_{\mathcal{D}^k}$ of $\Gamma$ implies that the orbit $W^k=\Gamma\cdot z_0$, where 
\begin{equation*}\label{uniformPRD}
W^{k}=\{z_0\}\cup\{\gamma\cdot z_0\co\gamma\in \Psi^{k}_r\}\cup\bigcup_{\psi\in\Psi^{k}}\{\gamma\cdot z_0\co\gamma\in J_{\psi}\},
\end{equation*}
is an infinite disjoint union for $k\geq 2$. Let $\mathcal{V}^k$ be the finite collection of all sets of the form $\bigcup_{\psi\in\Psi}V_{\psi}^k$, or $V_0^k\cup\bigcup_{\psi\in\Psi}V_{\psi}^k$, or $\{z_0\}\cup\bigcup_{\psi\in\Psi}V_{\psi}^k$, or $\{z_0\}\cup V_0^k\cup\bigcup_{\psi\in\Psi}V_{\psi}^k$ for $\Psi\subset\Psi^k$, where $V_0^k=\{\gamma\cdot z_0\co\gamma\in\Psi_r^k\}$ and $V_{\psi}^k=\{\gamma\cdot z_0\co\gamma\in J_{\psi}\}$. The application of \cite[Proposition 4.2]{CSParadox} to $W^k$ and $\mathcal{V}^k$ implies that there exists a number $D\in[0,2]$, a $\Gamma$--invariant $D$--conformal density $\mathcal{M}=(\mu_z)$ for $\hyp$ and a family of Borel measures $\{\nu_{\psi}\}_{\psi\in\Psi^k}$ such that (a) $\mu_{z_0}=\sum_{\psi\in\Psi^k}\nu_{\psi}$, (b) $\mu_{z_0}(S_{\infty})=1$ and 
$$\tnr{(c)}\quad\dis{\int_{S_{\infty}}\left(\lambda_{\gamma,z_0}\right)^Dd\nu_{s(\gamma)}=1-\sum_{\psi\in S(\gamma)}\int_{S_{\infty}} d\nu_{\psi}}$$
for every group--theoretical relation $(\gamma, s(\gamma), S(\gamma))$ of the decomposition $\Gamma_{\mathcal{D}^k}$. 

Since $\Gamma$ is finitely generated, it is tame \cite{Agol,CG}. Then \cite[Propositions 6.9]{CSParadox} and \cite[Proposition 3.9]{CSParadox} imply that every $\Gamma$--invariant $D$--conformal density $\mathcal{M}$ is a constant multiple of the area density $\mathcal{A}$ or $D=2$. From (b), we get $\mathcal{M}=\mathcal{A}$. Finally ($iii$) follows from (c).
\end{proof}

The number of displacement functions for the decomposition $\Gamma_{\mathcal{D}^k}$ is determined by the number of group--theoretical relations counted in \fullref{lem22}. We aim to apply  \fullref{thm3.4} and \fullref{lem1.2} to each group-theoretical relation $(\gamma,s(\gamma),S(\gamma))$ for the decomposition $\Gamma_{\mathcal{D}^k}$ to determine these displacement functions for each $k\geq 2$. 

Let $I_1=\{1,2,\dots,3^{k-1}\}$, $I_2=\{3^{k-1}+1,\dots,2\cdot 3^{k-1}\}$, $I_3=\{2\cdot 3^{k-1}+1,\dots,3\cdot 3^{k-1}\}$ and, $I_4=\{3\cdot 3^{k-1}+1,\dots, 4\cdot 3^{k-1}\}$. For $d=4\cdot 3^{k-1}$ let us define the set $$\deltad=\left\{\left(x_1,x_2,\dots,x_{d}\right)\in\mathbb{R}^{d}_+:\sum_{i=1}^{d}x_i=1\right\}.$$ Points of $\Delta^{d-1}$ will be written in bold fonts, eg  $\tb{x}=(x_1,x_2,\dots,x_d)$. We shall use the functions $\sigma\co (0,1)\to (0,\infty)$ and $\Sigma_j\co\Delta^{d-1}\to (0,1)$ with formulas 
\begin{equation}\label{sigma}
\sigma(x)=\frac{1-x}{x}\quad\textnormal{and}\quad\Sigma_j(\tb{x})=\sum_{i\in I_j}x_i
\end{equation}
for $j=1,2,3,4$, respectively, to express the displacement functions compactly. In particular we prove the following;
\begin{prop}\label{dispfunc}
Let $\Gamma=\langle\xi,\eta\rangle$ be a free , geometrically infinite Kleinian group without parabolics and, $\Gamma_{\mathcal{D}^k}$ be a symmetric decomposition of $\Gamma$ for $k\geq 2$. Let $a_1, a_2,\dots,a_k$ be the integers given by \fullref{lem22}. Then there exists a set of functions 
\begin{equation}
\mathcal{G}^k=\bigcup_{i\in I^k}\{f_i^{k},g_{i}^{k,1},g_{i,1}^{k,2},\dots,g_{i,a_2}^{k,2},g_{i,1}^{k,3},\dots,g_{i,a_3}^{k,3},\dots, g_{i,1}^{k,k},\dots,g_{i,a_k}^{k,k}\}
\end{equation}
such that for any $z_0\in\hyp$ and for each $\gamma\in\Psi^k$, the expression $e^{2\dgamma}$ is bounded below by $f(\tb{x})$ for $\tb{x}\in\deltad$ for at least one of $f\in\mathcal{G}^k$, where
\begin{equation}
f_i^k(\tb{x})=\left\{\begin{array}{ll}
              \sigma(\Sigma_1(\tb{x}))\sigma(x_i)   & \tnr{if $i \bmod 4\equiv 0$,}\\
              \sigma(\Sigma_4(\tb{x}))\sigma(x_i)   & \tnr{if $i \bmod 4\equiv 1$,}\\
              \sigma(\Sigma_3(\tb{x}))\sigma(x_i)   & \tnr{if $i \bmod 4\equiv 2$,}\\
              \sigma(\Sigma_2(\tb{x}))\sigma(x_i)   & \tnr{if $i \bmod 4\equiv 3$.}
             \end{array}\right.
\end{equation}
\end{prop}
\begin{proof}
Let $\{\nu_{\psi}\}_{\psi\in\Psi^k}$ be the family of Borel measures on $S_{\infty}$ given by  \fullref{thm3.4} for $\Gamma=\langle\xi,\eta\rangle$. Then we claim that $0<\nu_{\psi}(S_{\infty})<1$ for every $\psi\in\Psi^k$ for every $k\geq 2$. To prove the claim it is enough to show that $\nu_{\psi_0}(S_{\infty})\neq 0$ for all $\psi_0\in\Psi^k$. 

Assume that $\nu_{\psi_0}(S_{\infty})=0$ for a given $\psi_0\in\Psi^k$. Note that $(\psi_0,\psi_0^{-1},S(\psi_0))$ is a group-theoretical property for $\Gamma_{\mathcal{D}^k}$ when $S(\psi_0)$ is the set of words in $\Psi^k$ which doesn't start with the first letter of $\psi_0$. Since we have $\psi_0^{-1}=s(\psi_0)$, we get $\sum_{\psi\in
S(\psi_0)}\nu_{\psi}=1$ by  \fullref{thm3.4} ($iii$). Then we see that $\nu_{\psi_1}(S_{\infty})\neq 0$ for some $\psi_1\in S(\psi_0)$. Let $\psi_2\in\Psi^k-S(\psi_0)$. If $S(\psi_2)$ denotes the set of all words in $\Psi_k$ which doesn't start with the first letter of $\psi_2$, then $(\psi_2,\psi_2^{-1},S(\psi_2))$ is a group-theoretical relation for $\Gamma_{\mathcal{D}^k}$. By the equalities $\sum_{\psi\in\Psi^k}\nu_{\psi}=1$ and $\sum_{\psi\in S(\psi_0)}\nu_{\psi}=1$ we derive that
$\nu_{\psi_2}(S_{\infty})=0$. By  \fullref{thm3.4} ($iii$), we obtain that $\sum_{\psi\in S(\psi_2)}\nu_{\psi}=1$. Using the facts that
$\sum_{\psi\in\Psi^k}\nu_{\psi}=1$ and $S(\psi_0)\cap S(\psi_2)=\emptyset$, we find that $\nu_{\psi_1}(S_{\infty})=0$, a contradiction.

 \fullref{thm3.4} ($iii$) and ($ii$) show that $\nu_{s(\gamma)}(S_{\infty})$ and $\int_{S_{\infty}}\lambda^{2}_{\gamma,z_0}d\mu_{V_{s(\gamma)}}$ satisfy the hypothesis of \fullref{lem1.2} for each group-theoretical relation $(\gamma,s(\gamma),S(\gamma))$ of $\Gamma_{\mathcal{D}^k}$. Hence by letting $\nu=\nu_{s(\gamma)}$, $a=\nu_{s(\gamma)}(S_{\infty})$ and $b=\int_{S_{\infty}}\lambda^{2}_{\gamma,z_0}d\mu_{V_{s(\gamma)}}$ in \fullref{lem1.2} we obtain the lower bounds
\begin{equation}\label{eqn4}
\begin{array}{lll}
e^{\dis 2\dgamma } & \geq &  \dis\frac{\sigma\left(\dis{\int_{S_{\infty}} d\nu_{s(\gamma)}}\right)}{\sigma\left(\dis{\int_{S_{\infty}}\lambda^{2}_{\gamma,z_0}d\mu_{V_{s(\gamma)}}}\right)}\\& = & \dis\sigma\left(\sum_{\psi\in S(\gamma)}m_{p(\psi)}\right)\sigma\left(m_{p(s(\gamma))}\right)
\end{array}
\end{equation}
for each relation $(\gamma,s(\gamma), S(\gamma))$ of $\Gamma_{\mathcal{D}^k}$, where $m_{p(\psi)}=\int_{S_{\infty}} d\nu_{\psi}$ for the bijection $p:\Psi^k\ra I^k=\{i\in\mathbb{Z}:1\leq i\leq 4\cdot 3^{k-1}\}$ in (\ref{Enum}). We replace each constant $m_{p(\psi)}$ appearing in (\ref{eqn4}) with the variable $x_{p(\psi)}$. Let $\tb{m}_k=(m_1,m_2,\dots,m_{d})\in\Delta^{d-1}$. 

The constants obtained on the right hand-side of the inequalities in the expression (\ref{eqn4}) can be considered as the values of the functions in $\mathcal{G}^k$ at the point $\tb{m}_k$. The first group of functions $\{f_i^{k}\}_{i\in I^k}$ are determined by the relations $(\gamma,s(\gamma),S(\gamma))$ so that $length(\gamma s(\gamma))=0$. The second group $\{g_i^{k,1}\}_{i\in I^k}$ is determined by the relations with $length(\gamma s(\gamma))=1$. Finally, the third group of functions $$\{g_{i,1}^{k,2},\dots,g_{i,a_2}^{k,2}\}\cup\{g_{i,1}^{k,3},\dots,g_{i,a_3}^{k,3}\}\cup\dots\cup\{g_{i,1}^{k,k},\dots,g_{i,a_k}^{k,k}\}$$ are determined by the relations with the condition $2\leq length(\gamma s(\gamma))\leq k$. Hence we obtain $R_k$ many displacement functions so that $e^{2\dgamma}$ is bounded below by $f(\tb{m}_k)$ for at least one of $f\in\mathcal{G}^k$. The formulas of the functions $\{f_i^k\}_{i\in I^k}$ are derived from the fact that they are obtained by the group--theoretical relations $(\gamma, s(\gamma), S(\gamma))$ for $s(\gamma)=\gamma^{-1}$. 
\end{proof}

As an illustration, we list some of the displacement functions for the symmetric decomposition $\Gamma_{\mathcal{D}^2}$.
These displacement functions are produced by using \fullref{thm3.4} for $k=2$,  \fullref{lem1.2} and the group--theoretical relations listed in \fullref{table1} given above:
\begin{small}
\begin{equation}\label{deffs}
\begin{array}{ll}
\dis{f_{1}^{2}(\tb{x})=\frac{1-x_{10}-x_{11}-x_{12}}{x_{10}+x_{11}+x_{12}}\cdot\frac{1-x_{1}}{x_{1}}}, & \dis{f_{7}^{2}(\tb{x})=\frac{1-x_{4}-x_{5}-x_{6}}{x_{4}+x_{5}+x_{6}}\cdot\frac{1-x_{7}}{x_{7}}}\\
\dis{f_{2}^{2}(\tb{x})=\frac{1-x_{7}-x_{8}-x_{9}}{x_{7}+x_{8}+x_{9}}\cdot\frac{1-x_{2}}{x_{2}}}, & \dis{f_{8}^{2}(\tb{x}) =\frac{1-x_{1}-x_{2}-x_{3}}{x_{1}+x_{2}+x_{3}}\cdot\frac{1-x_{8}}{x_{8}}},\\
\dis{f_{3}^{2}(\tb{x})=\frac{1-x_{4}-x_{5}-x_{6}}{x_4+x_5+x_6}\cdot\frac{1-x_{3}}{x_{3}}}, & \dis{f_{9}^{2}(\tb{x})=\frac{1-x_{10}-x_{11}-x_{12}}{x_{10}+x_{11}+x_{12}}\cdot\frac{1-x_{9}}{x_{9}}}, \\
\dis{f_{4}^{2}(\tb{x})=\frac{1-x_{1}-x_{2}-x_{3}}{x_{1}+x_{2}+x_{3}}\cdot\frac{1-x_{4}}{x_{4}}}, & \dis{f_{10}^{2}(\tb{x})=\frac{1-x_{7}-x_{8}-x_{9}}{x_{7}+x_{8}+x_{9}}\cdot\frac{1-x_{10}}{x_{10}}},\\
\dis{f_{5}^{2}(\tb{x})=\frac{1-x_{10}-x_{11}-x_{12}}{x_{10}+x_{11}+x_{12}}\cdot\frac{1-x_{5}}{x_{5}}}, & \dis{f_{11}^{2}(\tb{x}) =\frac{1-x_{4}-x_{5}-x_{6}}{x_{4}+x_{5}+x_{6}}\cdot\frac{1-x_{11}}{x_{11}}},\\
\dis{f_{6}^{2}(\tb{x})=\frac{1-x_{7}-x_{8}-x_{9}}{x_{7}+x_{8}+x_{9}}\cdot\frac{1-x_{6}}{x_{6}}}, & \dis{f_{12}^{2}(\tb{x})=\frac{1-x_{1}-x_{2}-x_{3}}{x_{1}+x_{2}+x_{3}}\cdot\frac{1-x_{12}}{x_{12}}}.\\
\end{array}
\end{equation}
\end{small} 

\noindent Let $\tb{m}=(\nu_{\xi^2}(S_{\infty}),\nu_{\xi\eta^{-1}}(S_{\infty}),\dots,\nu_{\xi^{-2}}(S_{\infty}))\in\Delta^{11}$. For instance, by \fullref{lem1.2} we have the inequalities 
\[
\begin{array}{ccc}
\dxitwo\geq 0.5\log f_1^2(\tb{m}), &  & \dxieta\geq 0.5\log f_3^2(\tb{m}),\\
\dxietanone\geq 0.5\log f_2^2(\tb{m}), &  & \detaxi\geq 0.5\log f_4^2(\tb{m})
\end{array}
\]
obtained by the group--theoretical relations 
($1$), ($2$), ($3$) and ($4$) in  \fullref{table1}. Some other displacement functions for the symmetric decomposition $\Gamma_{\mathcal{D}^2}$ are
\begin{equation}\label{defgs}
\begin{array}{lll}
g_1^{2,1}(\tb{x}) &= &\dis\frac{1-x_4-x_5-x_6-x_7-x_8-x_9-x_{10}-x_{11}-x_{12}}{x_4+x_5+x_6+x_7+x_8+x_9+x_{10}+x_{11}+x_{12}}\cdot\frac{1-x_1}{x_1},
\\
g_2^{2,1}(\tb{x}) & = &\dis\frac{1-x_1-x_2-x_3-x_7-x_8-x_9-x_{10}-x_{11}-x_{12}}{x_1+x_2+x_3+x_7+x_8+x_9+x_{10}+x_{11}+x_{12}}\cdot\frac{1-x_2}{x_2},
\\
g_3^{2,1}(\tb{x}) & = &\dis\frac{1-x_1-x_2-x_3-x_4-x_5-x_6-x_{10}-x_{11}-x_{12}}{x_1+x_2+x_3+x_4+x_5+x_6+x_{10}+x_{11}+x_{12}}\cdot\frac{1-x_3}{x_3},
\\
g_4^{2,1}(\tb{x}) & = &\dis\frac{1-x_1-x_2-x_3-x_4-x_5-x_6-x_{7}-x_{8}-x_{9}}{x_1+x_2+x_3+x_4+x_5+x_6+x_{7}+x_{8}+x_{9}}\cdot\frac{1-x_4}{x_4},
\\
g_5^{2,1}(\tb{x}) &=&\dis\frac{1-x_4-x_5-x_6-x_7-x_8-x_9-x_{10}-x_{11}-x_{12}}{x_4+x_5+x_6+x_7+x_8+x_9+x_{10}+x_{11}+x_{12}}\cdot\frac{1-x_5}{x_5},
\\
g_6^{2,1}(\tb{x}) & = & \dis\frac{1-x_1-x_2-x_3-x_7-x_8-x_9-x_{10}-x_{11}-x_{12}}{x_1+x_2+x_3+x_7+x_8+x_9+x_{10}+x_{11}+x_{12}}\cdot\frac{1-x_6}{x_6},
\end{array}
\end{equation}
obtained by the group--theoretical relations (1), (2), (3), (4), (5) and (6) in \fullref{table2}, respectively. Then these functions imply the inequalities 
\[\begin{array}{ll}
\dxiinv\geq 0.5\log g_1^{2,1}(\tb{m}), & \deta\geq 0.5\log g_4^{2,1}(\tb{m}),\\
\dxiinv\geq 0.5\log g_2^{2,1}(\tb{m}), & \deta\geq 0.5\log g_5^{2,1}(\tb{m}),\\
\dxiinv\geq 0.5\log g_3^{2,1}(\tb{m}), & \deta\geq 0.5\log g_6^{2,1}(\tb{m}).
\end{array}
\] 
By the group--theoretical relations (2), (5), (13) and (16) in  \fullref{table3} we also obtain the following displacement functions for the symmetric decomposition $\Gamma_{\mathcal{D}^2}$ of $\Gamma=\langle\xi,\eta\rangle$:
\begin{small}
\begin{equation}\label{dispothergs}
\begin{array}{ll}
g^{2,2}_{1,1}(\tb{x})=\left(\dis{1\bigg/\sum_{i=1,i\neq 5}^{12}x_i}-1\right)\cdot\dis{\frac{1-x_1}{x_1}}, & g^{2,2}_{1,5}(\tb{x})=\left(\dis{1\bigg/\sum_{i=1,i\neq 1}^{12}x_i}-1\right)\cdot\dis{\frac{1-x_5}{x_5}},\\
g^{2,2}_{1,2}(\tb{x})=\left(\dis{1\bigg/\sum_{i=1,i\neq 6}^{12}x_i}-1\right)\cdot\dis{\frac{1-x_2}{x_2}}, & g^{2,2}_{1,6}(\tb{x})=\left(\dis{1\bigg/\sum_{i=1,i\neq 2}^{12}x_i}-1\right)\cdot\dis{\frac{1-x_6}{x_6}}.
\end{array}
\end{equation}
\end{small}

The functions in (\ref{dispothergs}) provide the lower bounds for the hyperbolic displacements listed below: 
\[\begin{array}{c}
\detaxi\geq 0.5\log g_{1,1}^{2,2}(\tb{m}),\quad \dxieta\geq 0.5\log g_{1,5}^{2,2}(\tb{m}),\\ \detaxi\geq 0.5\log g_{1,2}^{2,2}(\tb{m}),\quad  \dxieta\geq 0.5\log g_{1,6}^{2,2}(\tb{m}).
\end{array}
\]
There are in total $48$ such inequalities for the displacements under the isometries $\psi\in\Psi_r^2\cup\Psi^2$ determined by the symmetric decomposition $\Gamma_{\mathcal{D}^2}$ (see \fullref{lem22}). Notice that the displacement functions $f_4^2$, $f_3^2$, $g_1^{2,1}$, $g_2^{2,1}$, $g_3^{2,1}$, $g_4^{2,1}$, $g_5^{2,1}$, $g_6^{2,1}$, $g_{1,1}^{2,2}$, $g_{1,2}^{2,2}$, $g_{1,5}^{2,2}$ and $g_{1,6}^{2,2}$, which were studied in \cite{Y}, give lower bounds for the hyperbolic displacements under the set of isometries $\Gamma^{\dagger} =\{\xi, \eta, \xi\eta\}\subset\Psi^2_r\cup\Psi^2$ in the symmetric decomposition $\Gamma_{\mathcal{D}^2}$. 

As another example, by the group--theoretical relations in \fullref{Table4}, \fullref{thm3.4} for $k=3$ and \fullref{lem1.2} we obtain the formulas of some of the displacement functions $\{f_i^3\}_{i\in I^3}$ for the symmetric decomposition $\Gamma_{\mathcal{D}^3}$ as
\begin{equation}
\begin{array}{ll}
f^3_{i}(\tb{x})  =  \left(1\bigg/\dis{\sum_{l=28}^{36}x_l}-1\right)\cdot\dis{\frac{1-x_i}{x_i}}, & f^3_{j}(\tb{x})  = \left(1\bigg/\dis{\sum_{l=19}^{27}x_l}-1\right)\cdot\dis{\frac{1-x_j}{x_j}},\\
f^3_{m}(\tb{x})  =  \left(1\bigg/\dis{\sum_{l=10}^{18}x_l}-1\right)\cdot\dis{\frac{1-x_m}{x_m}}, & f^3_{n}(\tb{x}) = \left(1\bigg/\dis{\sum_{l=1}^{9}x_l}-1\right)\cdot\dis{\frac{1-x_n}{x_n}}\\
\end{array}
\end{equation}
for $i\in\{1,5,9,\dots,33\}$, $j\in\{2,6,10,\dots,34\}$, $m\in\{3,7,11,\dots,35\}$ and $n\in\{4,8,12,\dots,36\}$ so that $\dgamma\geq 0.5\log f_i^3(\tb{m})$ for some $i\in I^3$ for every $\gamma\in\Psi^3$, where $\tb{m}=(\nu_{\xi^3}(S_{\infty}),\nu_{\xi^2\eta^{-1}}(S_{\infty}),\dots,\nu_{\xi^{-3}}(S_{\infty}))\in\Delta^{35}$. There are $252$ such displacement functions for the displacements under the isometries $\gamma\in\Psi_r^3\cup\Psi^3$ determined by the symmetric decomposition $\Gamma_{\mathcal{D}^3}$ (see \fullref{lem22}). 
To calculate a lower bound for the maximum of the hyperbolic displacements under the isometries in $\Psi_r^k\cup\Psi^k$, we shall compute the greatest lower bound for the maximum of all of the functions in $\mathcal{G}^k$ over the simplex $\Delta^{d-1}$. In particular, if $G^k$ is the continuous function defined as
\begin{equation}\label{G}
\begin{array}{lllll}
G^k&: & \Delta^{d-1} &\ra&\mathbb{R}\\
  &  & \tb{x} &\mapsto&\max\{f(\tb{x}):f\in\mathcal{G}^k\},
\end{array}
\end{equation}
we aim to calculate $\inf_{\tb{x}\in\Delta^{d-1}}G^k(\tb{x})$. The details of this computation are given in \fullref{S3}.



\section{Infima of the Maximum of the Functions in $\mathcal{G}^k$ on $\Delta^{d-1}$}\label{S3}

Calculations given in this section are for a fixed integer $k\geq 2$. Therefore, we shall drop the superscript $k$, the marker of the symmetric decomposition $\Gamma_{\mathcal{D}^k}$ of $\Gamma=\langle\xi,\eta\rangle$, from the displacement functions $\{f_i^k\}_{i\in I^k}$ whose formulas are listed in \fullref{dispfunc}.

If $\mathcal{F}^k=\{f_i\}_{i\in I^k}$, we will show that $\inf_{\tb{x}\in\Delta^{d-1}}G^k(\tb{x})=\inf_{\tb{x}\in\Delta^{d-1}}F^k(\tb{x})$ for every $k\geq 2$ (see \fullref{minF} and \ref{minG}), where $F^k$ is the continuous function defined as
\begin{equation}\label{F}
\begin{array}{lllll}
F^k &\co & \Delta^{d-1} &\to &\mathbb{R}\\
  &  & \tb{x} &\mapsto &\max\left(f_1(\tb{x}),f_2(\tb{x}),\dots,f_{d}(\tb{x})\right).
\end{array}
\end{equation}
Therefore, it is enough to find $\inf_{\tb{x}\in\Delta^{d-1}}F^k(\tb{x})$. We first prove the following lemma:
\begin{lemma}\label{lemtwo}
If $F^k$ is the function defined in (\ref{F}), then $\alpha_*=\inf_{\tb{x}\in\Delta^{d-1}}F^k(\tb{x})$ is attained in $\Delta^{d-1}$ and contained in the interval $[1,12\cdot 3^{k-1}-3]$ for $k\geq 2$.
\end{lemma}
\begin{proof}
This proof uses analogous arguments given in \cite[Lemma 4.2]{Y}. To save space we give a proof sketch. By the formulas of $f_i$ given in \fullref{dispfunc}, for some $f_i\in\mathcal{F}^k$ we see that $f_i(\tb{x}_n )$ approaches to infinity on any sequence $\{\tb{x}_n\}\subset\Delta^{d-1}$ which limits on $\partial\Delta^{d-1}$. This observation implies that  $\inf_{\tb{x}\in\Delta^{d-1}}F^k(\tb{x})=\min_{\tb{x}\in\Delta^{d-1}}F^k(\tb{x})$.  

We have $f_i(\tb{x})>1$ for every $\tb{x}\in\Delta^{d-1}$ which shows that $  \alpha_*\geq 1$. Consider the point $\ys=(1/d,1/d,\dots,1/d)\in\Delta^{d-1}$, where $d=4\cdot 3^{k-1}$. Then for every $k\geq 2$ we get  $\Sigma_1(\tb{y})=\Sigma_2(\tb{y})=\Sigma_3(\tb{y})=\Sigma_4(\tb{y})=1/4$. Again by the formulas of $f_i$ given in   \fullref{dispfunc}, we have $f_i(\tb{y}^*)=3\cdot (4\cdot 3^{k-1}-1)$ for every $i\in I^k$. As a result we obtain $  \alpha_*\in [1,12\cdot 3^{k-1}-3]$.
\end{proof}

We shall use the notation $\xs$ to denote a point at which the infimum of $F^k$ is attained on $\Delta^{d-1}$. To calculate $  \alpha_*=\min_{\tb{x}\in\Delta^{d-1}}F^k(\tb{x})$, we exploit the convexity properties of the displacement functions in $\mathcal{F}^k$. 

For $j\in\{1,2,3,4\}$ and $i\in I^k$, introduce the functions $f\co\Delta\to (0,1)$, $g\co\Delta\to (0,1)$ and $\Sigma_j^i\co\Delta^{d-1}\to\R$ defined by 
\begin{equation}\label{fg}
\begin{array}{c}
\dis{f(x,y)=\frac{1-x}{x}\cdot\frac{1-y}{y}},\ \ \dis{g(x,y)=\frac{1-x-y}{x+y}\cdot\frac{1-y}{y}},\ \ \dis{\Sigma_j^i(\tb{x})=\sum_{l\in I_j,l\neq i}x_l}
\end{array}
\end{equation}
where $\Delta=\{(x,y)\in\mathbb{R}^2\co x+y<1,\ 0<x, 0<y\}$. Remember that we have the sets
\[
\begin{array}{c}
I_1=\{1,\dots,3^{k-1}\}, I_2=\{3^{k-1}+1,\dots,2\cdot 3^{k-1}\}, I_3=\{2\cdot 3^{k-1}+1,\dots,3\cdot 3^{k-1}\},\\
I_4=\{3\cdot 3^{k-1}+1,\dots, 4\cdot 3^{k-1}\}.
\end{array}
\]
Given a displacement function $f_i(\tb{x})=\sigma(\Sigma_j(\tb{x}))\sigma(x_i)$ in $\mathcal{F}^k$ for $j\in\{1,2,3,4\}$ and $i\in I^k$ in  \fullref{dispfunc}, it can be expressed as 
\begin{equation}
f_i(\tb{x})=\left\{\begin{array}{ll}
                                      f(\Sigma_j(\tb{x}),x_i) & \textnormal{if $i\notin I_j$,}\\
                                      g(\Sigma_j^i(\tb{x}),x_i) & \textnormal{if $i\in I_j$.}
                  \end{array}\right.
\end{equation}
So the convexity of $f_i\in\mathcal{F}^k$ follows from the convexities of $f$ and $g$. We shall use the statement below which gives a sufficient condition to check the convexities of $f$ and $g$:
\begin{thm}\label{thmconvex}
Let $f$ be a twice continuously differentiable real-valued function on an open convex set $C$ in $\mathbb{R}^n$. Then $f$ is a strictly convex function if its Hessian matrix $H_f(\tb{x})=(\partial^2f/\partial x_i\partial x_j(\tb{x}))$ for $i,j=1,\dots,n$ is positive definite for every $\tb{x}\in C$.
\end{thm} 
As this theorem is one of the standard facts from convex analysis, various proofs are readily available in the literature. Therefore no proof will be included here. Interested readers may refer to \cite[Theorem 4.5]{RTRoc} for an analogous statement and its proof. 

In particular  \fullref{thmconvex} implies that a twice continuously differentiable real--valued function $f(x,y)$ is strictly convex on an open convex set $C$ if $f_{xx}(\tb{x})>0$, $f_{yy}(\tb{x})>0$ and $det\ H_f(\tb{x})>0$ for every $\tb{x}\in C$. Then we have the following lemmas:
\begin{lemma}\label{convex2}
Let $C_g=\{(x,y)\in\Delta\co x+2y-xy-y^2<3/4\}$. Then $C_g$ is an open convex set and, $g(x,y)$ is a strictly convex function on $C_g$. 
\end{lemma}
\begin{proof}
Consider the equality $x+2y-xy-y^2=3/4$. For $x=\frac{3/4+y^2-2y}{1-y}$ we have $x''=\frac{1}{2(-1+y)^3}<0$ for every $y\in(0,3/4)$, which implies the first assertion of the lemma. Note that $g$ is twice continuously differentiable on $C_g$. Consider the Hessian matrix $H_g(\tb{x})$ of $g$:
\begin{eqnarray*}
\left[\begin{array}{ll}
g_{xx}(\textbf{x}) & g_{xy}(\textbf{x})\\
g_{yx}(\textbf{x}) & g_{yy}(\textbf{x})
\end{array}\right] & = & \left[
\begin{array}{ll}
\displaystyle{\frac{2(1-y)}{(x+y)^3y}} & \displaystyle{\frac{x+3y-2y^2}{(x+y)^3y^2}}\\
\displaystyle{\frac{x+3y-2y^2}{(x+y)^3y^2}} & \displaystyle{\frac{2x^2(x^2+3xy+3y^2-y^3)}{x^2y^3(x+y)^3}}\end{array}\right]
\end{eqnarray*}
for $\textbf{x}=(x,y)\in\Delta$. It is clear that $g_{xx}(\tb{x})>0$ for every $\tb{x}\in C_g$. We also have $g_{yy}(\tb{x})>0$ for every $\tb{x}\in C_g$ because $$x^2y^3(x+y)^3g_{yy}(\tb{x})=2x^2(x^2+3xy+y^2(3-y))>0.$$ The determinant $(3+4x(-1+y)-8y+4y^2)/(y^4(x+y)^4)$ of $H_{g}(\tb{x})$ is positive for every $(x,y)\in C_{g}$. Hence, $g(x,y)$ is strictly convex on $C_{g}$ by  \fullref{thmconvex}.
\end{proof}
\begin{lemma}\label{convex1}
Let $C_f=\{(x,y)\in\Delta\co 7x+(18-8\sqrt{2})y<3+\sqrt{2}\}$. Then $C_f$ is an open convex set and $f(x,y)$ is a strictly convex function on $C_f$.  
\end{lemma}
\begin{proof}
It is clear to see that $C_f$ is an open convex set and $f$ is twice continuously differentiable on $C_f$. Now consider the Hessian matrix $H_{f}(\textbf{x})$ of $f$:
\begin{eqnarray*}
\left[
\begin{array}{ll}
f_{xx}(\textbf{x}) & f_{xy}(\textbf{x})\\
f_{yx}(\textbf{x}) & f_{yy}(\textbf{x})
\end{array}\right]                & = &  \left[
\begin{array}{ll}
\displaystyle{\frac{2(1-y)}{x^3y}} & \displaystyle{\frac{1}{x^2y^2}}\\
\displaystyle{\frac{1}{x^2y^2}} & \displaystyle{\frac{2(1-x)}{y^3x}}
\end{array}\right]
\end{eqnarray*}
at $\textbf{x}=(x,y)\in\Delta$. Note that $f_{xx}(\tb{x})>0$ and $f_{yy}(\tb{x})>0$ for every $\tb{x}\in C_f$. The determinant $(3+4x(-1+y)-4y)/(x^4y^4)$ of $H_f(\tb{x})$ is positive for every $(x,y)\in\Delta$ if $x+xy+y<3/4$ . The line $7x+(18-8\sqrt{2})y=3+\sqrt{2}$ is tangent to the curve $x+xy+y=3/4$ at the point $P((2-\sqrt{2})/2,\sqrt{2}/4)$. Since for $y=\frac{3/4-x}{1+x}$ we have $y''=\frac{7/4}{(1+x)^3}>0$ for every $x\in(0,3/4)$, the function $f(x,y)$ is strictly convex on $C_f$ by  \fullref{thmconvex}.
\end{proof}
\begin{lemma}\label{convex3}
The functions $f(x,y)$ and $g(x,y)$ are strictly convex functions on the open convex set $C_f\cap C_g$. 
\end{lemma}
\begin{proof}
By the proof of  \fullref{convex1} we know that $f$ has a positive definite Hessian matrix over the set $C=\{(x,y)\in\Delta\co x+xy+y<3/4\}$. Note that $C_f\subset C$. The curves $7x+(18-8\sqrt{2})y=3+\sqrt{2}$, $x+xy+y=3/4$ and $x+2y-xy-y^2=3/4$ intersect in $\Delta$ only at the point $P$ defined in the lemma above. Since we have \[\frac{3+\sqrt{2}-(18-8\sqrt{2})y}{7}<\frac{3/4-y}{1+y}<\frac{3/4+y^2-2y}{1-y}\] for $y\in(\sqrt{2}/4,3/4)$ and  \[\frac{3/4+y^2-2y}{1-y}<\frac{3+\sqrt{2}-(18-8\sqrt{2})y}{7}<\frac{3/4-y}{1+y}\] for $y\in(0,\sqrt{2}/4)$, the conclusion of the lemma follows.
\end{proof}

Let $f_i(\tb{x})=\sigma(\Sigma_j(\tb{x}))\sigma(x_i)$ be a displacement function in $\mathcal{F}^k$ described in  \fullref{dispfunc}. If $i\in I_j$, then define the set
\begin{equation}\label{C1}
C_{f_i}=\{\tb{x}=(x_1,\dots,x_{d})\in\Delta^{d-1}\co \Sigma_j^i(\tb{x})+2x_i-\Sigma_j^i(\tb{x})x_i-(x_i)^2<3/4\},
\end{equation}
if $i\notin I_j$, by abusing the notation, define the set
\begin{equation}\label{C2}
C_{f_i}=\{\tb{x}=(x_1,\dots,x_{d})\in\Delta^{d-1}\co 7\Sigma_j(\tb{x})+(18-8\sqrt{2})x_i<3+\sqrt{2}\}.
\end{equation}
If $C_{f_i}$ for $i\in I^k$ are described as above, then $\dis{\cap_{i=1}^d C_{f_i}}$ is nonempty, where $d=4\cdot 3^{k-1}$. Because, if we consider the point $\tb{y}^*=(1/d,1/d,\dots,1/d)\in\Delta^{d-1}$, then $\Sigma_j(\tb{y}^*)=1/4$ and  
$\Sigma_j^i(\tb{y}^*)=1/4-1/(4\cdot 3^{k-1})$. For $k=2$ and $k\geq 3$, we clearly have
\[
7\Sigma_j(\tb{y}^*)+(18-8\sqrt{2})y_i=\frac{7}{4}+\frac{18-8\sqrt{2}}{4\cdot 3^{k-1}}\leq\frac{7}{4}+\frac{18-8\sqrt{2}}{12}<3+\sqrt{2}.
\]
Thus $\tb{y}^*$ is in $C_{f_i}$ for every $f_i(\tb{x})=\sigma(\Sigma_j(\tb{x})\sigma(x_i)\in\mathcal{F}^k$ such that $i\in I_j$. Similarly for $k=2$ and $k\geq 3$ we have the inequalities
\[
\Sigma_j^i(\tb{y}^*)+2y_i-\Sigma_j^i(\tb{y}^*)y_i-(y_i)^2=\frac{1}{4}+\frac{3}{16\cdot 3^{k-1}}\leq\frac{5}{16}<\frac{3}{4},
\]
which shows that $\tb{y}^*$ is in $C_{f_i}$ for every $f_i(\tb{x})=\sigma(\Sigma_j(\tb{x}))\sigma(x_i)\in\mathcal{F}^k$ such that $i\notin I_j$. 

We shall prove further statements about the elements of the sets $C_{f_i}$. In each statement we consider the following cases
\[
\begin{array}{ccc}
(1)\  k=2, & (2)\  \tnr{$k> 2$ and $k$ is even}, & (3)\ \tnr{$k>2$ and $k$ is odd}.
\end{array}
\]
We will carry out the calculations for $k=2$, if necessary for $k=3$ or $4$, and indicate how to generalise these calculations for the cases in (2) and (3) for easy reading.
For $k\geq 2$ let us define the functions
\begin{equation}\label{mn}
m(k)=\lceil 3^{k-1}/4\rceil,\quad n(k)=\lfloor 3^{k-1}/4\rfloor,\quad\tnr{ and }\quad \alpha(k)=12\cdot 3^{k-1}-3.
\end{equation}
Assume that $k$ is even and $k\geq 2$. We note that there are $m=m(k)$ many elements in $I_1$ which are equivalent to $1$ in modulo $4$. The same is true for the number of elements equivalent to $2$ or $3$. But there are $n=n(k)$ many elements in $I_1$ which are equivalent to $0$ in modulo $4$. In other words we obtain the list $(m,m,m,n)$ for the number of elements in $I_1$ which are equivalent to $1$, $2$, $3$ or $0$, respectively. Together with $I_2$, $I_3$ and $I_4$, 
we have the lists 
\begin{equation}\label{keven}
\begin{tabular}{c|cccl}
            &  1 & 2  & 3  & 0 \\
\hline  
$I_1$  & m & m & m & n\\
$I_2$  & m & m & n  & m\\
$I_3$  & m & n  & m & m\\
$I_4$  & n  & m & m & m. 
\end{tabular}
\end{equation}
Note that the lists for $k=2$ are
$(1,1,1,0)$, $(1,1,0,1)$, $(1,0,1,1)$ and $(0,1,1,1)$. This table will be used in  \fullref{unique1}, \ref{unique2}, \ref{unique5} and \ref{unique6}.

Assume that $k>2$ is odd. In this case there are $m$ many elements in $I_1$ which are equivalent to $1$ in modulo $4$. There are $n$ many elements each in $I_1$ which are equivalent to $2$, $3$ or $0$ in modulo $4$. In other words we obtain the list $(m,n,n,n)$ for the number of elements in $I_1$ which are equivalent to $1$, $2$, $3$ or $0$, respectively. Together with $I_2$, $I_3$ and $I_4$ we have the lists 
\begin{equation}\label{kodd}
\begin{tabular}{c|cccc}
             &  1 & 2  & 3  & 0 \\
\hline  
$I_1$  & m & n & n & n\\
$I_2$  & n & m & n  & n\\
$I_3$  & n & n  & m & n\\
$I_4$  & n  & n & n & m.
\end{tabular}
\end{equation}
This table will be used in \fullref{unique3}, \ref{unique4}, \ref{unique7} and \ref{unique8}. In particular we shall deploy the tables in (\ref{keven}) and (\ref{kodd}) to add the terms in the summations in the lemmas below indexed over some or all of the elements of $I_1$, $I_2$, $I_3$ and $I_4$. Since we only use modulo $4$, we shall indicate $a \bmod 4\equiv b$ with $a\equiv b$ in the rest of this text . Let us define the function  for $k\geq 2$. Then we have the followings:

\begin{lemma}\label{unique1}
Let $\mathcal{F}^k=\{f_i\}$ for ${i\in I^k}$ be the set of displacement functions listed in \fullref{dispfunc} and $F^k$ be as in (\ref{F}). Let $\xs$ be a point in $\Delta^{d-1}$ so that $  \alpha_*=F^k(\xs)$ for $d=4\cdot 3^{k-1}$. Let $f_i\in\mathcal{F}^k$ be of the form $f_i=f(\Sigma_j,x_i)$ for $j\in\{1,2,3,4\}$ and $i\in I^k=I_1\cup I_2\cup I_3\cup I_4$ where $\Sigma_j(\tb{x})$ and $f$ are defined in (\ref{sigma}) and (\ref{fg}), respectively. If $k\geq 2$ is even, $j=1$ and $i\in I_2$ such that $i\equiv 0$,
then  $\xs\in C_{f_i}$, defined in (\ref{C2}).
\end{lemma}
\begin{proof}
Assume on the contrary that $\xs\notin C_{f_i}$. By the definition of  $C_{f_i}$ we obtain that
\begin{equation}\label{ineq}
7\Sigma_1(\tb{x}^*)+(18-8\sqrt{2})x_i^*\geq 3+\sqrt{2}.
\end{equation} 
Let $\Sigma_1^*=\Sigma_1(\xs)$,  $\Sigma_2^*=\Sigma_2(\xs)$,  $\Sigma_3^*=\Sigma_3(\xs)$ and  $\Sigma_4^*=\Sigma_4(\xs)$ defined in (\ref{sigma}), where $\Sigma_1^*+\Sigma_2^*+\Sigma_3^*+\Sigma_4^*=1$  since $\xs\in\Delta^{d-1}$.  
Also let $N=\frac{1}{71} \left(13+7 \sqrt{2}\right)\approx 0.3225$. We consider the cases below: 
\begin{equation}\label{ABC}
\begin{array}{lll}
(A)\ \ \Sigma_1^*\geq N\ \tnr{ and } \ x_i^*\geq N, & (B)\ \ \Sigma_1^*\geq N > x_i^*,  & (C)\ \  x_i^*\geq N > \Sigma_1^*.
\end{array}
\end{equation}
We shall assume without the loss of generality that $k=2$. Assume that ($A$) holds. Then since $\Sigma_2^*>x_i^*$, we have $\Sigma_1^*+\Sigma_2^*> 2N$. This gives the inequality
\begin{equation}\label{sum1}
\Sigma_3^*+\Sigma_4^*< M=1-2N=\frac{1}{71} \left(45-14 \sqrt{2}\right)\approx 0.3549,
\end{equation}
which implies the following cases:
\begin{equation}\label{1-3}
\begin{array}{lll}
(i)\   \Sigma_3^*< M/2,\ \ \Sigma_4^*< M/2, &
(ii)\ \Sigma_3^*< M/2\leq \Sigma_4^*, &
(iii)\ \Sigma_4^*< M/2\leq\Sigma_3^*.
\end{array}
\end{equation}
Assume that ($i$) holds. Since $\Sigma_3^*< M/2$ and $\Sigma_4< M/2$, by the inequalities  $\sigma(M/2)\sigma(x_l^*)<\sigma(\Sigma_r^*)\sigma(x_l^*)\leq\alpha(k)\big|_{k=2}=33$ for $r=3,4$ we find that 
\begin{equation}\label{xis}
x_l^*> X(k)\bigg|_{k=2}=\frac{\sigma(M/2)}{\alpha(k)+\sigma(M/2)}\bigg|_{k=2}=\frac{2-M}{2+(\alpha(k)-1)M}\bigg|_{k=2}\approx 0.1231
\end{equation}
 for every ${l\in I^k}$ so that ${l\equiv 1,2}$. Using the table in (\ref{keven}) for $l\notin I_1$ we calculate that 
\begin{equation}\label{eqnn1}
\begin{multlined}
\Sigma_1^*+x_i^*+\sum_{l\equiv 1,2}x_l >  2N+\sum_{l\equiv 1,2}X(k)\bigg|_{k=2}=2N+(4m(k)+2n(k))X(k)\bigg|_{k=2} \\   \shoveleft[6.5cm]{=  2N+4X(2)
\approx 1.1376>1,}
\end{multlined}
\end{equation}
 a contradiction. The inequalities in (\ref{eqnn1}) hold for every even $k>2$.  Hence case ($i$) does't hold. 

Assume that ($ii$) holds in (\ref{1-3}). By (\ref{xis}) we already know that  $x_l^*> X(k)\big|_{k=2}$ for every $l\in I^k$ such that $l\equiv 2$. For $l\in I_2$ by the table in (\ref{keven}) we derive that  
\begin{equation}\label{eqn32}
\begin{multlined}
\Sigma_1^*+\Sigma_2^*+\Sigma_4^*> L(k)\bigg|_{k=2}=2N+\sum_{l\equiv 2}X(k)\bigg|_{k=2}+\frac{M}{2}\\ \shoveleft[2.5cm]{=2N+m(k)X(k)\bigg|_{k=2}+\frac{M}{2}=2N+X(2)
+\frac{M}{2}
\approx 0.9457,}
\end{multlined}
\end{equation}
which shows that $\Sigma_3^*< R(k)\big|_{k=2}=1-L(k)\big|_{k=2}
\approx 0.0543$. This implies that 
 $x_r^*< Q(k)\big|_{k=2}=\dis(R(k)/3^{k-1})\big|_{k=2}<\dis X(k)\big|_{k=2}$ for some $r\in I_3$ such that $r\not\equiv 2$. We shall examine the cases $r\in I_3$ so that $r\equiv 1$, or $3$, or $0$ in this order.  

Assume that $r\equiv 1$. Using  $\sigma(\Sigma_4^*)\sigma(Q(k))\big|_{k=2}<\sigma(\Sigma_4^*)\sigma(x_r^*)\leq\alpha(k)\big|_{k=2}=33$ we calculate that
\begin{equation*}
\Sigma_4^*> S(k)\bigg|_{k=2}=\frac{\sigma(Q(k))}{\alpha(k)+\sigma(Q(k))}\bigg|_{k=2}
\approx 0.6217,
\end{equation*}
which leads to the contradiction   
\begin{equation}\label{eqnn3}
\Sigma_1^*+\Sigma_2^*+\Sigma_4^*> N+N+S(k)\big|_{k=2}=2N+S(2)\approx 1.2667>1.
\end{equation}
So we conclude that $x_r^*\geq Q(k)|_{k=2}$ for every $r\in I_3$ such that $r\equiv 1$.

Assume that $r\equiv 3$. Since  $\sigma(\Sigma_2^*)\sigma(Q(k))\big|_{k=2}<\sigma(\Sigma_2^*)\sigma(x_r^*)\leq\alpha(k)\big|_{k=2}=33$ we obtain that $\Sigma_2^*> S(k)\big|_{k=2}=S(2)$. Then for $l\in I_3$ using the table in (\ref{keven}) we see that
\begin{equation}\label{eqnn4}
\begin{multlined}
\Sigma_1^*+\Sigma_2^*+\sum_{l\equiv 1}x_l^*+\sum_{l\equiv\ 2}x_l^*+\Sigma_4^*  \\ \shoveleft[1.5cm]{> N+S(k)\bigg|_{k=2}+\sum_{l\equiv\ 1}Q(k)\bigg|_{k=2}+\sum_{l\equiv\ 2}X(k)\bigg|_{k=2}+\frac{M}{2}}\\ \shoveleft[2.5cm]{= N+S(k)\bigg|_{k=2}+m(k)Q(k)\bigg|_{k=2}+n(k)X(k)\bigg|_{k=2}+\frac{M}{2}}\\ \shoveleft[3.5cm]{=N+S(2)+Q(2)+\frac{M}{2}\approx 1.1398>1,}
\end{multlined}
\end{equation} 
a contradiction. So we must have $x_r^*\geq Q(k)\big|_{k=2}$ for every $r\in I_3$ such that $r\equiv 3$.

Assume that $r\equiv 0$. Using $\sigma(\Sigma_1^*)\sigma(Q(k))\big|_{k=2}< \sigma(\Sigma_1^*)\sigma(x_r^*)\leq\alpha(k)\big|_{k=2}=33$ we get $\Sigma_1^*> S(k)\big|_{k=2}=S(2)$. From the table in (\ref{keven}) for $l\in I_2\cup I_3=\{4,5,6,7,8,9\}$ and $t\in I_3$ we calculate 
\begin{equation}\label{eqnn8}
\begin{multlined}
\Sigma_1^*+x_i^*+\sum_{l\equiv 2}x_l^*+\sum_{t\equiv 1,3}x_t^*
+\Sigma_4^* \\ \shoveleft[.5cm]{
> S(k)\bigg|_{k=2}+N
+\sum_{l\equiv\ 2}X(k)\bigg|_{k=2}+\sum_{t\equiv\ 1,3}Q(k)\bigg|_{k=2}+\frac{M}{2}}\\ \shoveleft[1.5cm]{= S(k)\bigg|_{k=2}+N
+(m(k)+n(k))X(k)\bigg|_{k=2}+2m(k)Q(k)\bigg|_{k=2}+\frac{M}{2}} 
\\ \shoveleft[2.5cm]{=S(2)+N+X(2)
+2Q(2)+\frac{M}{2}\approx 1.2810>1,}
\end{multlined}
\end{equation} 
a contradiction. The inequalities in (\ref{eqnn3}), (\ref{eqnn4}) and (\ref{eqnn8})  hold for every even $k>2$. Therefore case ($ii$) doesn't hold.

In case ($iii$) in (\ref{1-3}) we see that the inequalities for $\Sigma_3^*$ and $\Sigma_4^*$ are switched. So the discussion that shows that  case ($ii$) doesn't hold works for case ($iii$) as well by switching the roles of $\Sigma_3^*$ and  $I_3$  with $\Sigma_4^*$ and $I_4$, respectively. We obtain the same expressions on the right-hand side of the inequalities in (\ref{eqn32}), (\ref{eqnn3}), (\ref{eqnn4}) and (\ref{eqnn8}). In particular we repeat the computations given in the order $l\equiv 1,3,0$ for $l\in I_3$ above in the order $l\equiv 2,3,0$ for $l\in I_4$. So case ($iii$) doesn't hold. As a result we conclude that (A) in (\ref{ABC}) is not the case.

We consider the next case $\Sigma_1^*\geq N>x_i^*$ ($B$) in (\ref{ABC}). Then we derive the inequality 
\begin{equation}\label{sum2}
\Sigma_2^*+\Sigma_3^*+\Sigma_4^*\leq M=1-N=\frac{58-7\sqrt{2}}{71}\approx 0.6774,
\end{equation}
which implies the following cases:
\begin{equation}\label{1-7}
\begin{array}{llll}
(i)  & \Sigma_2^*\leq M/3, & \Sigma_3^*\leq M/3, & \Sigma_4^*\leq M/3,\\
(ii) & \Sigma_2^*\leq M/3, & \Sigma_3^*\leq M/3, & \Sigma_4^*\geq M/3,\\
(iii) & \Sigma_2^*\leq M/3, & \Sigma_3^*\geq M/3, & \Sigma_4^*\leq M/3,\\
(iv) & \Sigma_2^*\leq M/3, & \Sigma_3^*\geq M/3, & \Sigma_4^*\geq M/3,\\
(v)  & \Sigma_2^*\geq M/3, & \Sigma_3^*\leq M/3, & \Sigma_4^*\leq M/3,\\
(vi) & \Sigma_2^*\geq M/3, & \Sigma_3^*\leq M/3, & \Sigma_4^*\geq M/3,\\
(vii) & \Sigma_2^*\geq M/3, & \Sigma_3^*\geq M/3, & \Sigma_4^*\leq M/3.
\end{array}
\end{equation}
We examine the cases ($i$)-($vii$). 
Assume that ($i$) holds. Since we have $\Sigma_r^*\leq M/3$ for $r=2,3,4$, using  the inequality $\sigma(M/3)\sigma(x_l^*)\leq\sigma(\Sigma_r^*)\sigma(x_l^*)\leq\alpha(k)\big|_{k=2}=33$ we find that 
\begin{equation}\label{xis2}
x_l^*\geq X(k)\bigg|_{k=2}=\frac{\sigma(M/3)}{\alpha(k)+\sigma(M/3)}\bigg|_{k=2}=\frac{3-M}{3+(\alpha(k)-1)M}\bigg|_{k=2}\approx 0.0941
\end{equation}
 for every ${l\in I^k}$ so that ${l\equiv 3,2,1}$. Since $\Sigma_2^*\leq M/3$ in this case, for ${l\in I_2=\{4,5,6\}}$ by the table in (\ref{keven}) we obtain that 
\begin{equation}\label{eqnn9}
\begin{multlined}
x_i^*<\sum_{{l\equiv 0}}x_l^*\leq Y(k)\bigg|_{k=2}=\frac{M}{3}-\sum_{l\equiv 1,2,3}X(k)\bigg|_{k=2}\\ \shoveleft[2.2cm]{=\frac{M}{3}-(2m(k)+n(k))X(k)\bigg|_{k=2}=\frac{M}{3}-2X(2)\approx 0.0376.}
\end{multlined}
\end{equation}  
By the inequality in (\ref{ineq}) we derive that 
\begin{equation}\label{eqnn10}
\begin{multlined}
\Sigma_1^*\geq L(k)\bigg|_{k=2}=\left(\frac{8 \sqrt{2}-18}{7}\right)Y(k)\bigg|_{k=2}+\frac{\sqrt{2}+3}{7}\approx 0.5587.
\end{multlined}
\end{equation} 
Then using the table in (\ref{keven}) for ${l\in I_2\cup I_3\cup I_4=\{4,5,6,7,8,9,10,11,12\}}$ we obtain a contradiction which is 
\begin{equation}\label{eqnn11}
\begin{multlined}
\Sigma_1^*+\sum_{{l\equiv 1,2,3}}x_l^*\geq L(k)\bigg|_{k=2}+(6m(k)+3n(k))X(k)\bigg|_{k=2}\\ \shoveleft[5.5cm]{=L(2)+6X(2)\approx 1.1593>1.}
\end{multlined}
\end{equation}
The inequalities in (\ref{eqnn11}) holds for every even $k>2$. Hence we conclude that case ($i$) doesn't hold. 
 
Assume that ($ii$) in (\ref{1-7}) holds. Since we have $\Sigma_2^*\leq M/3$ and $\Sigma_3^*\leq M/3$, we find $x_l^*\geq X(k)\big|_{k=2}$ for every ${l\in I^k}$ such that ${l\equiv 2,3}$ by (\ref{xis2}). By the inequality $\Sigma_2^*\leq M/3$, for ${l\in I_2=\{4,5,6\}}$ we obtain from the table in (\ref{keven}) that 
\begin{equation}\label{eqnn12}
\begin{multlined}
x_i^*<\sum_{l\equiv 0,1}x_l^*\leq Y(k)\bigg|_{k=2}=\frac{M}{3}-\sum_{{l\equiv 2,3}}X(k)\bigg|_{k=2}\\ \shoveleft[2.5cm]{=\frac{M}{3}-(m(k)+n(k))X(k)\bigg|_{k=2}=\frac{M}{3}-X(2)\approx 0.1317.}
\end{multlined}
\end{equation} 
By using the inequality in (\ref{ineq}) we obtain 
\begin{equation}\label{eqnn13}
\begin{multlined}
\Sigma_1^*\geq L(k)\bigg|_{k=2}=\left(\frac{8 \sqrt{2}-18}{7}\right)Y(k)\bigg|_{k=2}+\frac{\sqrt{2}+3}{7}\approx 0.5048.
\end{multlined}
\end{equation}  
 We claim that $\Sigma_4^*< 4/13$. Because otherwise for ${l\in I_2\cup I_3=\{4,5,6,7,8,9\}}$ using the table in (\ref{keven})
we derive that 
\begin{equation}\label{eqnn14}
\begin{multlined}
\Sigma_1^*+\sum_{{l\equiv 2,3}}x_l^*+\Sigma_4^*\geq L(k)\bigg|_{k=2}+\sum_{{l\equiv 2,3}}X(k)\bigg|_{k=2}+\frac{4}{13}\\\shoveleft[3.5cm]{=L(k)\bigg|_{k=2}+2(m(k)+n(k))X(k)\bigg|_{k=2}+\frac{4}{13}}\\ \shoveleft[4cm]{=L(2)+2X(2)+\frac{4}{13}\approx 1.0007>1,}
\end{multlined}
\end{equation}
 a contradiction. Using the inequalities $\sigma(x_l^*)\sigma(4/13)<\sigma(x_l^*)\sigma(\Sigma_4^*)\leq\alpha(k)\big|_{k=2}=33$, we find that $x_l^*>(9/(9+4\alpha(k)))\big|_{k=2}$ for ${l\in I^k}$ such that ${l\equiv 1}$. Then using the table in (\ref{keven}) for ${l\in I_2\cup I_3}$ we get 
\begin{equation}\label{eqnn15}
\begin{multlined}
\Sigma_1^*+\sum_{{l\equiv 1,2,3}}x_l^*+\Sigma_4^*> L(k)\bigg|_{k=2}+\sum_{{l\equiv 1}}\frac{9}{9+4\alpha(k)}\bigg|_{k=2}+\sum_{{l\equiv 2,3}}X(k)\bigg|_{k=2}+\frac{M}{3}\\ \shoveleft[2cm]{ =L(k)\bigg|_{k=2}+\frac{9(2m(k))}{9+4\alpha(k)}\bigg|_{k=2}+2(m(k)+n(k))X(k)\bigg|_{k=2}+\frac{M}{3}}\\ \shoveleft[2.5cm]{ =L(2)+\frac{18}{141}+2X(2)+\frac{M}{3}\approx 1.0465>1,}
\end{multlined}
\end{equation}
a contradiction. The inequalities in (\ref{eqnn14}) and (\ref{eqnn15}) hold for every even $k>2$. Therefore, case ($ii$) doesn't hold.

We can repeat the argument given above for case ($ii$) for case ($iii$) in (\ref{1-7})  as well by switching the roles of $\Sigma_3^*$ and $\Sigma_4^*$.  Note that the number of elements in $I_2\cup I_3$ which are equivalent to $2$ or $3$ modulo $4$ is the same as the number of elements in $I_2\cup I_4$ which are equivalent to $1$ or $3$ modulo $4$ by table in (\ref{keven}).  We get the same inequalities in (\ref{eqnn12}), (\ref{eqnn14}) and (\ref{eqnn15}). Hence case ($iii$) doesn't hold.

Assume that case ($iv$) holds in (\ref{1-7}). Since $\Sigma_2^*\leq M/3$, we have $x_l^*\geq X(k)\big|_{k=2}$ for  every $l\equiv 3$ by (\ref{xis2}). We shall examine the following cases: 
\begin{equation}\label{abc}
\begin{array}{ll}
(a)\quad x_i^*\leq (M/3^k)\big|_{k=2}, & (b)\quad (M/3^k)\big|_{k=2}<x_i^*< M/3, 
\end{array}
\end{equation}
Assume that ($a$) holds. Then by the inequality in (\ref{ineq}) we derive the expression below
\begin{equation}\label{eqnn17}
\Sigma_1^*\geq L(k)\bigg|_{k=2}=\left(\frac{3+\sqrt{2}}{7} +\frac{2(4 \sqrt{2}-9) M}{7\cdot 3^k}\right)\bigg|_{k=2}\approx 0.5587.
\end{equation} 
By the table in (\ref{keven}), for ${l\in I_2=\{4,5,6\}}$ we obtain a contradiction which is given as
\begin{equation}\label{eqnn18}
\begin{multlined}
\Sigma_1^*+\Sigma_2^*+\Sigma_3^*+\Sigma_4^*  > L(k)\bigg|_{k=2}+\sum_{{l\equiv 3}}X(k)\bigg|_{k=2}+\frac{2M}{3}\\ \shoveleft[1.5cm]{ = L(k)\bigg|_{k=2}+n(k)X(k)\bigg|_{k=2}+\frac{2M}{3}
 = L(2)+\frac{2M}{3}\approx 1.0103>1.}
\end{multlined}
\end{equation}
The inequalities in (\ref{eqnn18}) holds for every even $k>2$. So ($a$) is not the case. 

Assume that ($b$) holds. Since we have $x_i^*< M/3$, by the inequality in (\ref{ineq}) we obtain 
\begin{equation}\label{eqnn19}
\Sigma_1^*\geq L=\frac{11}{1491}\left(73 \sqrt{2}-47\right)\approx 0.4149
\end{equation}
 We claim that $\Sigma_3^*<10/33$. Because otherwise we calculate for ${l\in I_2=\{4,5,6\}}$ that 
\begin{equation}\label{eqnn20}
\begin{multlined}
\Sigma_1^*+x_i^*+\sum_{{l \equiv 3}}x_l^*+\Sigma_3^*+\Sigma_4^*>L+\frac{M}{3^k}\bigg|_{k=2}+\sum_{{l\equiv 3}}X(k)\bigg|_{k=2}+\frac{10}{33}+\frac{M}{3}\\ \shoveleft[.5cm]{=L+\frac{M}{9}+n(k)X(k)\bigg|_{k=2}+\frac{10}{33}+\frac{M}{3}=L+\frac{4M}{9}+\frac{10}{33}\approx 1.0190>1,}
\end{multlined}
\end{equation}
a contradiction. A similar contradiction arises if we assume $\Sigma_4^*<10/33$ in the inequality above instead of $\Sigma_3^*$. By $\sigma(x_l^*)\sigma(10/33)<\sigma(x_l^*)\sigma(\Sigma_r^*)\leq\alpha(k)\big|_{k=2}=33$ for ${r=3,4}$ we find that $x_l^*>(23/(23+10\alpha(k)))\big|_{k=2}$ for ${l\equiv 1,2}$. Then we compute by the table in (\ref{keven}) for ${l\in I_2=\{4,5,6\}}$ that
\begin{equation}\label{eqnn21}
\begin{multlined}
\Sigma_1^*+x_i^*+\sum_{{l\equiv 1,2,3}}x_l^*+\Sigma_3^*+\Sigma_4^*\\ \shoveleft[1cm]{>L+\frac{M}{3^k}\bigg|_{k=2}+\sum_{{l\equiv 3}}X(k)\bigg|_{k=2}+\sum_{{l\equiv 1,2}}\frac{23}{23+10\alpha(k)}\bigg|_{k=2}+\frac{2M}{3}}\\ 
\shoveleft[2.5cm]{=L+\frac{M}{3^k}\bigg|_{k=2}+n(k)X(k)\bigg|_{k=2}+\frac{23(2m(k))}{23+10\alpha(k)}\bigg|_{k=2}+\frac{2M}{3}} \\ \shoveleft[3.5cm]{=L+\frac{46}{353}+\frac{7M}{9}\approx 1.0721>1,}
\end{multlined}
\end{equation}
a contradiction. The inequalities in (\ref{eqnn20}) and (\ref{eqnn21}) hold for every even $k>2$. Hence ($b$) is not the case either. Hence case ($iv$) doesn't hold. 

Assume that case ($v$) holds in (\ref{1-7}). Since $\Sigma_3^*\leq M/3$ and $\Sigma_4^*\leq M/3$, by using  (\ref{xis2}) above we obtain $x_l^*\geq X(k)\big|_{k=2}$ for ${l\equiv 1,2}$. We shall examine the cases ($a$) and ($b$)  in (\ref{abc}) and, additionally in ($c$), where 
\begin{equation}\label{c}
(c)\quad M/3\leq x_i^*<N.
\end{equation}

If $x_i^*\leq (M/3^k)\big|_{k=2}$ ($a$), by (\ref{ineq}) we obtain $\Sigma_1^*\geq L(k)\big|_{k=2}$, where $L(k)$ is defined in (\ref{eqnn17}). 
We claim that $\Sigma_2^*< 13/50$. Because otherwise using the table in (\ref{keven}) for ${l\in I_3\cup I_4=\{7,8,9,10,11,12\}}$ we get
\begin{equation}\label{eqnn22}
\begin{multlined}
\Sigma_1^*+\Sigma_2^*+\sum_{{l\equiv 1,2}}x_l^*\geq L(k)\bigg|_{k=2}+\frac{13}{50}+\sum_{{l\equiv 1,2}}X(k)\bigg|_{k=2}\\ \shoveleft[.0cm]{=L(k)\bigg|_{k=2}+\frac{13}{50}+2(m(k)+n(k))X(k)\bigg|_{k=2}=L(2)+\frac{13}{50}+2X(2)\approx 1.0069>1,}
\end{multlined}
\end{equation}
a contradiction. By the inequalities $\sigma(x_l^*)\sigma(13/50)<\sigma(\Sigma_2^*)\sigma(x_l^*)\leq\alpha(k)\big|_{k=2}= 33$ for ${l\equiv 3}$, we find $x_l^*>(37/(37+13\alpha(k)))\big|_{k=2}$. For ${l\in I_3\cup I_4}$
this gives a contradiction that is 
\begin{equation}\label{eqnn23}
\begin{multlined}
\Sigma_1^*+\Sigma_2^*+\sum_{{l\equiv 1,2,3}}x_l^*
> L(k)\bigg|_{k=2}+\frac{M}{3}+\sum_{{l\equiv 3}}\frac{37}{37+13\alpha(k)}\bigg|_{k=2}+\sum_{{l\equiv 1,2}}X(k)\bigg|_{k=2}
\\ \shoveleft[2cm]{=L(k)\bigg|_{k=2}+\frac{M}{3}+\frac{37(2(m(k)))}{37+13\alpha(k)}\bigg|_{k=2}+2(m(k)+n(k))X(k)\bigg|_{k=2}}\\ \shoveleft[2.5cm]{=L(2)+\frac{M}{3}+\frac{37}{233}+2X(2)\approx 1.1315>1.}
\end{multlined}
\end{equation}
This rules out the assumption $x_i^*\leq (M/3^k)\big|_{k=2}$ in ($a$). The inequalities in (\ref{eqnn22}) and (\ref{eqnn23}) hold for every even $k>2$.

Assume that $(M/3^k)\big|_{k=2}< x_i^*< M/3$ in ($b$). Since $x_i^*< M/3$, again by (\ref{ineq}) we calculate that  $\Sigma_1^*>\frac{11}{1491}\left(73 \sqrt{2}-47\right)=L$. 
We claim that $\Sigma_2^*< 2/5$. Otherwise by the table in (\ref{keven}) for ${l\in I_3\cup I_4=\{7,8,9,10,11,12\}}$ we would obtain a contradiction
\begin{equation}\label{eqnn24}
\begin{multlined}
\Sigma_1^*+\Sigma_2^*+\sum_{{l\equiv 1,2}}x_l^*> L+\frac{2}{5}+\sum_{{l\equiv 1,2}}X(k)\bigg|_{k=2}\\ \shoveleft[.5cm]{=L+\frac{2}{5}+2(m(k)+n(k))X(k)\bigg|_{k=2}=L+\frac{2}{5}+2X(2)\approx 1.0031>1.}
\end{multlined}
\end{equation}
Then the inequalities $\sigma(x_l^*)\sigma(2/5)<\sigma(\Sigma_2^*)\sigma(x_l^*)\leq \alpha(k)\big|_{k=2}= 33$ for ${l\equiv 3}$ imply that $x_l^*>(3/(3+2\alpha(k)))\big|_{k=2}$ for  ${l\equiv 3}$. We repeat the argument above to improve on these lower bounds as follows: We  claim that $\Sigma_2^*< 16/51$. Otherwise from the table in (\ref{keven}),  for ${l\in I_3\cup I_4}$ we see that
\begin{equation}\label{eqnn25}
\begin{multlined}
\Sigma_1^*+\Sigma_2^*+\sum_{{l\equiv 1,2}}x_l^*+\sum_{{l\equiv 3}}x_l^*>L+ \frac{16}{51}+\sum_{{l\equiv 1,2}}X(k)\bigg|_{k=2}+\sum_{{l\equiv 3}}\frac{3}{3+2\alpha(k)}\bigg|_{k=2}\\ \shoveleft[4cm]{=L+\frac{16}{51}+2(m(k)+n(k))X(k)\bigg|_{k=2}+\frac{3(2m(k))}{3+2\alpha(k)}\bigg|_{k=2}}\\ \shoveleft[4.5cm]{=L+\frac{16}{51}+\frac{2}{23}+2X(2)\approx 1.0038>1,}
\end{multlined}
\end{equation}
a contradiction. By $\sigma(x_l^*)\sigma(16/51)<\sigma(\Sigma_2^*)\sigma(x_l^*)\leq\alpha(k)\big|_{k=2}= 33$, we find that $x_l^*>(35/(35+16\alpha(k)))\big|_{k=2}$ for ${l\equiv 3}$. We claim that $\Sigma_1^*< 15/32$. Because otherwise by the table in (\ref{keven}) for $l\in I_3\cup I_4$ we would obtain 
\begin{equation}\label{eqnn255}
\begin{multlined}
\Sigma_1^*+\Sigma_2^*+\sum_{{l\equiv 1,2,3}}x_l^*>\frac{15}{32}+\frac{M}{3}+\sum_{{l\equiv 1,2}}X(k)\bigg|_{k=2}+\sum_{{l\equiv 3}}\frac{35}{35+16\alpha(k)}\bigg|_{k=2}\\ \shoveleft[3.5cm]{=\frac{15}{32}+\frac{M}{3}+2(m(k)+n(k))X(k)\bigg|_{k=2}+\frac{35(2m(k))}{35+16\alpha(k)}\bigg|_{k=2}}\\ \shoveleft[4cm]{=\frac{15}{32}+\frac{M}{3}+2X(2)+\frac{70}{563}\approx 1.0071>1,}
\end{multlined}
\end{equation}
a contradiction. By the inequalities $\sigma(x_l^*)\sigma(15/32)<\sigma(\Sigma_1^*)\sigma(x_l^*)\leq\alpha(k)\big|_{k=2}=33$ for ${l\equiv 0}$ we find that $x_l^*>(17/17+15\alpha(k)))\big|_{k=2}$. As a result for ${l\in I_3\cup I_4}$ we obtain  
\begin{equation}\label{eqnn28}
\begin{multlined}
\Sigma_1^*+\Sigma_2^*+\sum_{{l\equiv 1,2,3,0}}x_l^*\geq L+\frac{M}{3}+\sum_{{l\equiv 1,2}}X(k)\bigg|_{k=2}\\ \shoveleft[5cm]{+\sum_{{l\equiv 3}}\frac{35}{35+16\alpha(k)}\bigg|_{k=2}+\sum_{{l\equiv 0}}\frac{17}{17+15\alpha(k)}\bigg|_{k=2}}\\ \shoveleft[1.5cm]{= L+\frac{M}{3}+2(m(k)+n(k))X(k)\bigg|_{k=2}+\frac{35(2m(k))}{35+16\alpha(k)}\bigg|_{k=2}+\frac{17(2m(k))}{17+15\alpha(k)}\bigg|_{k=2}}\\ \shoveleft[2.5cm]{=L+\frac{M}{3}+2X(2)+\frac{70}{563}+\frac{17}{256}+\approx 1.0197>1,}
\end{multlined}
\end{equation}
a contradiction, which rules out the assumption $(M/3^k)\big|_{k=2}< x_i^*\leq M/3$. Again all of the inequalities in (\ref{eqnn24}), (\ref{eqnn25}), (\ref{eqnn255}) and (\ref{eqnn28}) hold for every even $k>2$.

Assume that $M/3<x_i^*<N$ (c). Using the table in (\ref{keven}), for ${l\in I_2=\{4,5,6\}}$ we derive that
\begin{equation}\label{eqn67}
\Sigma_2^*\geq x_i^*+\sum_{{l\equiv 1,2}}x_l^*\geq S(k)\bigg|_{k=2}=\frac{M}{3}+2m(k)X(k)\bigg|_{k=2}\approx 0.4140.
\end{equation}
Since $\xs\in\Delta^{11}$ and $\Sigma_1^*\geq N$ by (B), we have 
\begin{equation*}
\Sigma_3^*+\Sigma_4^*\leq L(k)\bigg|_{k=2}=1-N-S(k)\bigg|_{k=2}\approx 0.2634.
\end{equation*} 
Let $Q(k)\big|_{k=2}=(L(k)/2)\big|_{k=2}$. We shall examine the cases below:
\begin{equation}\label{def}
(d)\ \Sigma_3^*\tnr{ and } \Sigma_4^*< Q(k)\bigg|_{k=2},\ (e)\ \Sigma_3^*\leq Q(k)\bigg|_{k=2}<\Sigma_4^*,\  (f)\ \Sigma_4^*\leq Q(k)\bigg|_{k=2}<\Sigma_3^*.
\end{equation}

Assume that ($d$) holds. Using $\sigma(x_l^*)\sigma(Q(k))\big|_{k=2}\leq \sigma(\Sigma_r^*)\sigma(x_l^*)\leq\alpha(k)\big|_{k=2}= 33$ for $r=3,4$, for ${l\in I_3\cup I_4=\{7,8,9,10,11,12\}}$ such that ${l\equiv 1,2}$, we obtain
\begin{equation*}
x_l^*\geq T(k)\bigg|_{k=2}=\frac{\sigma(Q(k))}{\alpha(k)+\sigma(Q(k))}\bigg|_{k=2}\approx 0.1665.
\end{equation*}
As an implication of the inequality above by the table in (\ref{keven}) for ${l\in I_3\cup I_4}$, we get
\begin{equation}\label{eqnn29}
\begin{multlined}
\Sigma_1^*+\Sigma_2^*+\sum_{{l\equiv 1,2}}x_l^*>N+ S(k)\bigg|_{k=2}+\sum_{{l\equiv 1,2}}T(k)\bigg|_{k=2}\\ \shoveleft[0cm]{=N+ S(k)\bigg|_{k=2}+2(m(k)+n(k))T(k)\bigg|_{k=2}=N+S(2)+2T(2)\approx 1.0696>1,}
\end{multlined}
\end{equation}
a contradiction. The inequalities in (\ref{eqnn29})  hold for every even $k>2$. This rules out the assumption in ($d$).

Assume that ($e$) holds in (\ref{def}). Since $\Sigma_3^*\leq Q(k)\big|_{k=2}$, we obtain $x_l^*\geq T(k)\big|_{k=2}$ for ${l\equiv 2}$. We claim that $\Sigma_1^*<12/33$.  Otherwise by the table in (\ref{keven}), for ${l\in I_3}$, we find 
\begin{equation}\label{eqnn31}
\begin{multlined}
\Sigma_1^*+\Sigma_2^*+\sum_{{l\equiv 1,2}}x_l^*+\Sigma_4^*\geq\frac{12}{33}+ S(k)\bigg|_{k=2}+\sum_{{l\equiv 1}}X(k)\bigg|_{k=2}+\sum_{{l\equiv 2}}T(k)\bigg|_{k=2}+Q(k)\bigg|_{k=2}\\ \shoveleft[.5cm]{=\frac{12}{33}+ S(k)\bigg|_{k=2}+m(k)X(k)\bigg|_{k=2}+n(k)T(k)\bigg|_{k=2}+Q(k)\bigg|_{k=2}}\\ \shoveleft[4.5cm]{=\frac{12}{33}+ S(2)+X(2)+Q(2)\approx 1.0035>1,}
\end{multlined}
\end{equation}
a contradiction. Using $\sigma(x_l^*)\sigma(12/33)<\sigma(\Sigma_1^*)\sigma(x_l^*)\leq\alpha(k)\big|_{k=2}= 33$ for ${l\equiv 0}$, we calculate that $x_l^*> (7/(7+4\alpha(k)))\big|_{k=2}$. Then for ${l\in I_3}$ we obtain a contradiction
\begin{equation}\label{eqnn32}
\begin{multlined}
\Sigma_1^*+\Sigma_2^*+\sum_{{l\equiv 1,2,0}}x_l^*+\Sigma_4^*\\ \shoveleft[.1cm]{>N+ S(k)\bigg|_{k=2}+\sum_{{l\equiv 1}}X(k)\bigg|_{k=2}+\sum_{{l\equiv 2}}T(k)\bigg|_{k=2}+\sum_{{l\equiv 0}}\frac{7}{7+4\alpha(k)}\bigg|_{k=2}+Q(k)\bigg|_{k=2}}\\ \shoveleft[.3cm]{=N+S(k)\bigg|_{k=2}+m(k)X(k)\bigg|_{k=2}+n(k)T(k)\bigg|_{k=2}+\frac{7m(k)}{7+4\alpha(k)}\bigg|_{k=2}+Q(k)\bigg|_{k=2}}\\ \shoveleft[4.5cm]{=N+S(2)+X(2)+Q(2)\approx 1.0127>1.}
\end{multlined}
\end{equation}
The inequalities in (\ref{eqnn31}) and (\ref{eqnn32}) hold for every even $k>2$. This shows that ($e$) doesn't hold.

Assume that ($f$) holds in (\ref{def}). In this case we can use the argument above that proves that ($e$) doesn't hold. By interchanging the roles of $\Sigma_3^*$ and $I_3$ with $\Sigma_4^*$ and  $I_4$, respectively, we repeat the computations. We obtain the same inequalities in (\ref{eqnn31}) and (\ref{eqnn32}) which imply that ($f$) doesn't hold.
As a result we rule out the case ($c$). In particular we conclude that case ($v$) in (\ref{1-7}) does not hold.

Assume that case ($vi$) holds in (\ref{1-7}). Since $\Sigma_3^*\leq M/3$ in this case, we know by (\ref{xis2}) that $x_l^*\geq X(k)\big|_{k=2}$ for ${l\equiv 2}$. We examine the cases ($a$), ($b$), and ($c$) in (\ref{abc}) and (\ref{c}). If $x_i^*\leq (M/3^k)\big|_{k=2}$ ($a$), we obtain by (\ref{ineq}) that 
$$\Sigma_1^*\geq L(k)\big|_{k=2}=\frac{1}{4473}\left(761+1229 \sqrt{2}\right)\approx 0.5587,$$
 where $L(k)$ is explicitly given in (\ref{eqnn17}). Then we derive the following contradiction
\begin{equation}\label{eqnn34}
\begin{multlined}
\Sigma_1^*+\Sigma_2^*+\Sigma_4^*\geq L(k)\bigg|_{k=2}+\frac{2M}{3}\approx 1.0103>1.
\end{multlined}
\end{equation}
So ($a$) is not the case. The inequalities in (\ref{eqnn34}) holds for every even $k>2$.

If $(M/3^k)\big|_{k=2}< x_i^*< M/3$ ($b$), we get  $\Sigma_1^*>\frac{11}{1491}\left(73 \sqrt{2}-47\right)=L$ by (\ref{ineq}). We claim that $\Sigma_2^*< 9/25$. Because otherwise we find
\begin{equation}\label{eqnn35}
\begin{multlined}
\Sigma_1^*+\Sigma_2^*+\Sigma_4^*> L+\frac{9}{25}+\frac{M}{3}\approx 1.0007>1,
\end{multlined}
\end{equation}
a contradiction. By the inequality $\sigma(x_l^*)\sigma(9/25)<\sigma(\Sigma_2^*)\sigma(x_l^*)\leq\alpha(k)\big|_{k=2}=33$ for ${l\equiv 3}$, we find that $x_l^*>(16/(16+9\alpha(k)))\big|_{k=2}$. 
Next we claim that $\Sigma_4^*< 31/100$. Otherwise for ${l\in I_3}=\{7,8,9\}$ we would obtain using the table in (\ref{keven}) that 
\begin{equation}\label{eqnn36}
\begin{multlined}
\Sigma_1^*+\Sigma_2^*+\sum_{{l\equiv 2,3}}x_l^*+\Sigma_4^*> L+\frac{M}{3}+\sum_{{l\equiv 2}}X(k)\bigg|_{k=2}+\sum_{{l\equiv 3}}\frac{16}{16+9\alpha(k)}\bigg|_{k=2}+\frac{31}{100}\\ \shoveleft[4cm]{=L+\frac{M}{3}+n(k)X(k)\bigg|_{k=2}+\frac{16m(k)}{16+9\alpha(k)}\bigg|_{k=2}+\frac{31}{100}}\\ \shoveleft[4.5cm]{=L+\frac{M}{3}+\frac{16}{313}+\frac{31}{100}\approx 1.0018>1,}
\end{multlined}
\end{equation}
a contradiction. By the inequality $\sigma(x_l^*)\sigma(31/100)<\sigma(\Sigma_4^*)\sigma(x_l^*)\leq\alpha(k)\big|_{k=2}=33$ for ${l\equiv 1}$, we see that $x_l^*>(69/(69+31\alpha(k)))\big|_{k=2}$. Also we claim that $\Sigma_1^*< 11/25$. Otherwise by the table in (\ref{keven}), for ${l\in I_3}$ we compute that 
\begin{equation}\label{eqnn37}
\begin{multlined}
\Sigma_1^*+\Sigma_2^*+\sum_{{l\equiv 1,2,3}}x_l^*+\Sigma_4^*\\ > \frac{11}{25}+\sum_{{l\equiv 1}}\frac{69}{69+31\alpha(k)}\bigg|_{k=2}+\sum_{{l\equiv 2}}X(k)\bigg|_{k=2}+\sum_{{l\equiv 3}}\frac{16}{16+9\alpha(k)}\bigg|_{k=2}+\frac{2M}{3}\\ \shoveleft[2.5cm]{=\frac{11}{25}+\frac{69m(k)}{69+31\alpha(k)}\bigg|_{k=2}+n(k)X(k)\bigg|_{k=2}+\frac{16m(k)}{16+9\alpha(k)}\bigg|_{k=2}+\frac{2M}{3}}\\ \shoveleft[3.5cm]{= \frac{11}{25}+\frac{69}{1092}+\frac{16}{313}+\frac{2M}{3}\approx1.0059>1,}
\end{multlined}
\end{equation}
a contradiction. By the inequality $\sigma(\Sigma_1^*)\sigma(11/25)<\sigma(\Sigma_1^*)\sigma(x_l^*)\leq\alpha(k)\big|_{k=2}= 33$ this implies that $x_l^*>(14/(14+11\alpha(k)))\big|_{k=2}$ for ${l\equiv 0}$. Finally using the table in (\ref{keven}) for ${l\in I_3}=\{7,8,9\}$ we obtain a contradiction because 
\begin{equation}\label{eqnn38}
\begin{multlined}
\Sigma_1^*+\Sigma_2^*+\Sigma_3^*+\Sigma_4^*
\geq L+\sum_{{l\equiv 1}}\frac{69}{69+31\alpha(k)}\bigg|_{k=2}+\sum_{{l\equiv 2}}X(k)\bigg|_{k=2}
\\ 
\shoveleft[4cm]{+\sum_{{l\equiv 3}}\frac{16}{16+9\alpha(k)}\bigg|_{k=2}+\sum_{{l\equiv 0}} \frac{14}{14+11\alpha(k)}\bigg|_{k=2}+\frac{2M}{3}}
\\ \shoveleft[0cm]{=L+\frac{69m(k)}{69+31\alpha(k)}\bigg|_{k=2}+n(k)X(k)\bigg|_{k=2}+\frac{16m(k)}{16+9\alpha(k)}\bigg|_{k=2}
+\frac{14m(k)}{14+11\alpha(k)}\bigg|_{k=2}+\frac{2M}{3}}\\ 
\shoveleft[3.5cm]{= L+\frac{69}{1092}+\frac{16}{313}+\frac{14}{377}+\frac{2M}{3}\approx1.0180>1.}
\end{multlined}
\end{equation}
This shows that ($b$) doesn't hold. The inequalities in (\ref{eqnn35}), (\ref{eqnn36}), (\ref{eqnn37}) and (\ref{eqnn38}) hold for every even $k>2$. 

Assume that $x_i^*\geq M/3$ ($c$). Then by the table in (\ref{keven}) for ${l\in I_2}=\{4,5,6\}$ we calculate that
\begin{equation}\label{eqn79}
\Sigma_2^*> \frac{M}{3}+\sum_{{l\equiv 2}}X(k)\bigg|_{k=2}=\frac{M}{3}+m(k)X(k)\bigg|_{k=2}\approx 0.3199.
\end{equation} 
We claim that $\Sigma_1^*< 23/50$. Because otherwise for ${l\in I_3}$ we would compute that 
\begin{equation}\label{eqnn39}
\begin{multlined}
\Sigma_1^*+\Sigma_2^*+\Sigma_3^*+\Sigma_4^*
\geq \frac{23}{50}+\frac{M}{3}+m(k)X(k)\bigg|_{k=2}+\sum_{{l\equiv 2}}X(k)\bigg|_{k=2}+\frac{M}{3}\\ \shoveleft[.2cm]{=\frac{23}{50}+\frac{M}{3}+(m(k)+n(k))X(k)\bigg|_{k=2}+\frac{M}{3}=\frac{23}{50}+2X(2)+\frac{2M}{3}\approx 1.0058>1,}
\end{multlined}
\end{equation}
a contradiction. Then we find that $x_l^*>(27/(27+23\alpha(k)))\big|_{k=2}$ by the inequality $\sigma(x_l^*)\sigma(23/50)<\sigma(\Sigma_1^*)\sigma(x_l^*)\leq 33$ for ${l\equiv 0}$. Similarly we claim that  $\Sigma_2^*< 21/50$. Otherwise by the table (\ref{keven}) for ${l\in I_3}=\{7,8,9\}$ we obtain 
\begin{equation}\label{eqnn40}
\begin{multlined}
\Sigma_1^*+\Sigma_2^*+\Sigma_3^*+\Sigma_4^*
> N+\frac{21}{50}+\sum_{{l\equiv 2}}X(k)\bigg|_{k=2}+\sum_{{l\equiv 0}}\frac{27}{27+23\alpha(k)}\bigg|_{k=2}+\frac{M}{3}\\ \shoveleft[2cm]{=N+\frac{21}{50}+n(k)X(k)\bigg|_{k=2}+\frac{27m(k)}{27+23\alpha(k)}\bigg|_{k=2}+\frac{M}{3}}\\
\shoveleft[3cm]{=N+\frac{21}{50}+\frac{27}{786}+\frac{M}{3}\approx 1.0027>1,}
\end{multlined}
\end{equation}
another contradiction. By the inequality $\sigma(x_l^*)\sigma(21/50)<\sigma(\Sigma_2^*)\sigma(x_l^*)\leq 33$ we derive that $x_l^*> (29/(29+21\alpha(k)))\big|_{k=2}$ for ${l\equiv 3}$. Then we claim that $\Sigma_4^*<14/49$. Otherwise for ${l\in I_2\cup I_3}=\{4,5,6,7,8,9\}$ so that ${l\neq i}$ by the table in (\ref{keven}) we would find a contradiction 
\begin{equation}\label{eqnn41}
\begin{multlined}
\Sigma_1^*+x_i^*+\sum_{{l\equiv 2,3,0}}x_l^*+\Sigma_4^*\\ > N+\frac{M}{3}+\sum_{{l\equiv 2}}X(k)\bigg|_{k=2}+\sum_{{l\equiv 3}}\frac{29}{29+21\alpha(k)}\bigg|_{k=2}+\sum_{{l\equiv 0}}\frac{27}{27+23\alpha(k)}\bigg|_{k=2}+\frac{14}{49} \\ \shoveleft[1.5cm]{=N+\frac{M}{3}+(m(k)+n(k))X(k)\bigg|_{k=2}+\frac{29(m(k)+n(k))}{29+21\alpha(k)}\bigg|_{k=2}}\\ \shoveleft[8cm]{+\frac{27(2m(k)-1)}{27+23\alpha(k)}\bigg|_{k=2}+\frac{14}{49}}\\
\shoveleft[3cm]{=N+\frac{M}{3}+X(2)+\frac{29}{722}+\frac{27}{786}+\frac{14}{49}\approx 1.0027>1.}
\end{multlined}
\end{equation}
Now using the inequalities $\sigma(x_l^*)\sigma(14/49)<\sigma(\Sigma_4^*)\sigma(x_l^*)\leq\alpha(k)\big|_{k=2}= 33$ for ${l\equiv 1}$ we see that $x_l^*>(5/(5+2\alpha(k)))\big|_{k=2}$. As a result using the table in (\ref{keven}) for ${l\in I_2\cup I_3}$ so that $l\neq i$ we obtain a contradiction
\begin{equation}\label{eqnn42}
\begin{multlined}
\Sigma_1^*+x_i^*+\sum_{{l\equiv 1,2,3,0}}x_l^*+\Sigma_4^*>N+\frac{M}{3}+\sum_{{l\equiv 2}}X(k)\bigg|_{k=2}\\ \shoveleft[1cm]{+\sum_{{l\equiv 3}}\frac{29}{29+21\alpha(k)}\bigg|_{k=2}+\sum_{{l\equiv 1}}\frac{5}{5+2\alpha(k)}\bigg|_{k=2}+\sum_{{l\equiv 0}}\frac{27}{27+23\alpha(k)}\bigg|_{k=2}+\frac{M}{3}} \\ \shoveleft[2cm]{=N+\frac{M}{3}+(m(k)+n(k))X(k)\bigg|_{k=2}+\frac{29(m(k)+n(k))}{29+21\alpha(k)}\bigg|_{k=2}}\\ \shoveleft[5.5cm]{+\frac{5(2m(k))}{5+2\alpha(k)}\bigg|_{k=2}+\frac{27(2m(k)-1)}{27+23\alpha(k)}\bigg|_{k=2}+\frac{M}{3}}\\
\shoveleft[2.5cm]{=N+\frac{M}{3}+X(2)+\frac{29}{722}+\frac{10}{71}+\frac{27}{786}+\frac{M}{3}\approx 1.0836>1.}
\end{multlined}
\end{equation}
This eliminates the case $x_i^*\geq M/3$ (c). The inequalities in (\ref{eqnn39}),  (\ref{eqnn40}), (\ref{eqnn41}) and (\ref{eqnn42}) hold for every even $k>2$. Hence we conclude that case ($vi$) doesn't hold.

Assume that case ($vii$) holds in (\ref{1-7}). Note that the inequalities for $\Sigma_3^*$ and $\Sigma_4^*$ are switched in this case. Therefore the argument  given above which shows that case ($vi$) doesn't hold can be repeated by replacing the roles of $\Sigma_3^*$ and $I_3$ with $\Sigma_4^*$ and $I_4$.  We obtain the same inequalities in (\ref{eqnn36}),  (\ref{eqnn37}), (\ref{eqnn38}), (\ref{eqn79}), (\ref{eqnn39}),  (\ref{eqnn40}), (\ref{eqnn41}) and (\ref{eqnn42}) which imply that case ($vii$) does't hold. As a result we derive that $\Sigma_1^*\geq N>x_i^*$ (B) in (\ref{ABC}) is not the case either.

Assume that $x_i^*\geq N>\Sigma_1^*$ ($C$) in (\ref{ABC}). Note that  $\Sigma_2^*> x_i^*\geq N>M/3$. We need to consider the following cases
\begin{equation}\label{1-7-2}
\begin{array}{llll}
(i)  & \Sigma_1^*\leq M/3, & \Sigma_3^*\leq M/3, & \Sigma_4^*\leq M/3,\\
(ii) & \Sigma_1^*\leq M/3, & \Sigma_3^*\leq M/3, & \Sigma_4^*\geq M/3,\\
(iii) & \Sigma_1^*\leq M/3, & \Sigma_3^*\geq M/3, & \Sigma_4^*\leq M/3,\\
(iv) & \Sigma_1^*\leq M/3, & \Sigma_3^*\geq M/3, & \Sigma_4^*\geq M/3,\\
(v)  & \Sigma_1^*\geq M/3, & \Sigma_3^*\leq M/3, & \Sigma_4^*\leq M/3,\\
(vi) & \Sigma_1^*\geq M/3, & \Sigma_3^*\leq M/3, & \Sigma_4^*\geq M/3,\\
(vii) & \Sigma_1^*\geq M/3, & \Sigma_3^*\geq M/3, & \Sigma_4^*\leq M/3.
\end{array}
\end{equation}
If ($i$) holds, we see that $x_i^*\geq{\frac{1}{2982}\left(599+470 \sqrt{2}\right)}=L$ by (\ref{ineq}). Since $\Sigma_1^*\leq M/3$, $\Sigma_3^*\leq M/3$ and $\Sigma_4^*\leq M/3$ , we get $x_l^*\geq X(k)\big|_{k=2}$ for ${l\equiv 1,2,0}$ by (\ref{xis2}). Then by the table in (\ref{keven}) for ${l\in I^2}=\{1,\dots,12\}$ we find a contradiction which is
\begin{equation}\label{eqnn43}
x_i^*+\sum_{{l\equiv 1,2,0}}x_l^*\geq L+(9m(k)+3n(k)-1)X(k)\bigg|_{k=2}=L+8X(2)\approx  1.1766>1.
\end{equation}
The inequality in (\ref{eqnn43}) holds for every even $k>2$. Therefore case ($i$) doesn't hold.

If ($ii$) holds, we again have $x_i^*\geq L$. We also have $x_l^*\geq X(k)\big|_{k=2}$ for ${l\equiv 2,0}$ by (\ref{xis2}). Then by the table in (\ref{keven}) for ${l\in I_1\cup I_2\cup I_3}=\{1,2,3,4,5,6,7,8,9\}$ so that ${l\equiv 2,0}$  and ${l\neq i}$ we find that 
\begin{equation}\label{eqnn44}
x_i^*+\sum_{{l\equiv 2,0}}x_l^*+\Sigma_4^*\geq S(k)\bigg|_{k=2}=L+(4m(k)+2n(k)-1)X(k)\bigg|_{k=2}+\frac{M}{3}\approx 0.9319.
\end{equation}
For ${l\in I_1\cup I_2\cup I_3}$  this implies that $\sum_{{l\equiv 1,3}}x_l^*\leq 1-S(k)\big|_{k=2}\approx 0.0681$. Then for some ${r\in I_1\cup I_2\cup I_3}$ so that ${r\equiv 1,3}$ we get
\begin{equation}\label{eqnn45}
x_r^*\leq R(k)\bigg|_{k=2}=\frac{1-S(k)}{5m(k)+n(k)}\bigg|_{k=2}=R(2)=\frac{1-S(2)}{5}\approx 0.0136.
\end{equation} 
If ${r\equiv 1}$ in (\ref{eqnn45}), by the inequality $\sigma(\Sigma_4^*)\sigma(R(k))\big|_{k=2}\leq\sigma(\Sigma_4^*)\sigma(x_r^*)\leq\alpha(k)\big|_{k=2}= 33$ we derive
\begin{equation*}\label{eqnn46}
\Sigma_4^*\geq T(k)\bigg|_{k=2}=\frac{\sigma(R(k))}{\alpha(k)+\sigma(R(k))}\bigg|_{k=2}\approx 0.6870.
\end{equation*}
Then by the table in (\ref{keven}) for ${l\in I_1\cup I_2\cup I_3}$ such that ${l\neq i}$ we obtain a contradiction 
\begin{equation}\label{eqnn47}
\begin{multlined}
x_i^*+\sum_{{l\equiv 2,0}}x_l^*+\Sigma_4^*\geq L+\sum_{{l\equiv 2,0}}X(k)\bigg|_{k=2}+T(k)\bigg|_{k=2}\\ \shoveleft[4.0cm]{
=L+(4m(k)+2n(k)-1)X(k)\bigg|_{k=2}+T(k)\bigg|_{k=2}}\\ \shoveleft[6cm]{=L+3X(2)+T(2)\approx 1.3931>1.}
\end{multlined}
\end{equation}
 So $x_r^*>R(k)\big|_{k=2}$ for ${r\equiv 1}$. 

If ${r\equiv 3}$ in (\ref{eqnn45}),  then we calculate that $\Sigma_2^*\geq T(k)\big|_{k=2}$ by  using the inequality $\sigma(\Sigma_2^*)\sigma(R(k))\big|_{k=2}\leq\sigma(\Sigma_2^*)\sigma(x_r^*)\leq\alpha(k)\big|_{k=2}= 33$. As a result using the table in (\ref{keven}) for ${s\in I_1\cup I_3}$ and ${l\in I_1\cup I_3=\{1,2,3,7,8,9\}}$ we find a contradiction that is
\begin{equation}\label{eqn89}
\begin{multlined}
\sum_{{s\equiv 1}}x_s^*+\Sigma_2^*+\sum_{{l\equiv 2,0}}x_l^*+\Sigma_4^*\geq \sum_{{s\equiv 1}}R(k)\bigg|_{k=2}+T(k)\bigg|_{k=2}+\sum_{{l\equiv 2,0}}X(k)\bigg|_{k=2}+\frac{M}{3}\\ \shoveleft[1.5cm]{=2m(k)R(k)\bigg|_{k=2}+T(k)\bigg|_{k=2}+(2m(k)+2n(k))X(k)\bigg|_{k=2}+\frac{M}{3}}\\\shoveleft[3.5cm]{=\frac{2(1-S(2))}{5}+T(2)+2X(2)+\frac{M}{3}=P\approx 1.1283>1.}
\end{multlined}
\end{equation}
The inequalities in (\ref{eqnn47}) and (\ref{eqn89}) hold for every $k>2$. Hence we conclude that ($ii$) doesn't hold. 

Assume that case ($iii$) holds in (\ref{1-7-2}). We use the same argument given above that shows that case ($ii$) doesn't hold by switching the roles of $\Sigma_4^*$ and $\Sigma_3^*$. We get the same inequalities in (\ref{eqnn44}), (\ref{eqnn47}),   (\ref{eqn89}) and,  (\ref{eqnn45}). Hence case ($iii$) also doesn't hold for every even $k\geq 2$. 

Assume that case ($iv$) holds  in (\ref{1-7-2}). Since $x_i^*\geq L={\frac{1}{2982}\left(599+470 \sqrt{2}\right)}$ by (\ref{ineq}), $\Sigma_3^*\geq M/3$ and $\Sigma_4^*\geq M/3$, we obtain  
$$x_i^*+\Sigma_3^*+\Sigma_4^*\geq L+\frac{2M}{3}=K\approx 0.8754.$$ 
Then we find that $\Sigma_1^*+\Sigma_2^i(\xs)\leq 1-K$, where $\Sigma_2^i$ is defined in (\ref{fg}). For some ${r\in I_1\cup I_2-\{i\}}$ we must have $x_r^*\leq R(k)\big|_{k=2}=((1-K)/(2\cdot 3^{k-1}-1))\big|_{k=2}$. If ${r\equiv 1}$, we see that 
\begin{equation*}\label{eqnn49}
\Sigma_4^*\geq T(k)\bigg|_{k=2}=\frac{\sigma(R(k))}{\alpha(k)+\sigma(R(k))}\bigg|_{k=2}\approx 0.5425
\end{equation*}
by the inequality $\sigma(\Sigma_4^*)\sigma(R(k))\big|_{k=2}\leq\sigma(\Sigma_4^*)\sigma(x_r^*)\leq\alpha(k)\big|_{k=2}= 33$. So we obtain a contradiction because,
\begin{equation}\label{eqnn50}
x_i^*+\Sigma_3^*+\Sigma_4^*\geq L+\frac{M}{3}+T(k)\bigg|_{k=2}=L+\frac{M}{3}+T(2)\approx 1.1921>1.
\end{equation}
If ${r\equiv 2}$, we get $\Sigma_3\geq T(k)\big|_{k=2}$ by  $\sigma(\Sigma_3^*)\sigma(R(k))\big|_{k=2}\leq\sigma(\Sigma_3^*)\sigma(x_l^*)\leq\alpha(k)\big|_{k=2}= 33$. This gives the inequality in (\ref{eqnn50}) again, a contradiction. Thus we have $x_r^*>R(k)\big|_{k=2}$ for ${r\equiv 1,2}$. 

If ${r\equiv 3}$, then by the inequality $\sigma(\Sigma_2^*)\sigma(R(k))\big|_{k=2}\leq\sigma(\Sigma_2^*)\sigma(x_l^*)\leq \alpha(k)\big|_{k=2}= 33$ we derive that $\Sigma_2^*\geq T(k)\big|_{k=2}$. So by the table in (\ref{keven}) for ${r\in I_1=\{1,2,3\}}$ we find 
\begin{equation}\label{eqnn51}
\begin{multlined}
\sum_{{r\equiv 1,2}}x_r^*+\Sigma_2^*+\Sigma_3^*+\Sigma_4^*\geq \sum_{{r\equiv 1,2}}R(k)\bigg|_{k=2}+T(k)\bigg|_{k=2}+\frac{2M}{3}\\ \shoveleft[.0cm]{=2m(k)R(k)\bigg|_{k=2}+T(k)\bigg|_{k=2}+\frac{2M}{3}=\frac{2(1-K)}{5}+T(2)+\frac{2M}{3}\approx1.0440 >1,}
\end{multlined}
\end{equation}
a contradiction. The inequalities in (\ref{eqnn50}) and (\ref{eqnn51}) hold for every even $k>2$. Hence case ($iv$) doesn't hold.

Assume that case ($v$) holds in  (\ref{1-7-2}). Since $\Sigma_3^*\leq M/3$ and $\Sigma_4^*\leq M/3$, by (\ref{xis2}) we have  $x_l^*\geq X(k)\big|_{k=2}$ for ${l\equiv 1,2}$. Using the table in (\ref{keven}), for ${l\in I_2=\{4,5,6\}}$ we derive from (C) that
\begin{equation}\label{Sg2}
\Sigma_2^*> x_i^*+\sum_{{l\equiv 1,2}}x_l^*> S(k)\bigg|_{k=2}=N+2m(k)X(k)\bigg|_{k=2}=N+2X(2)\approx 0.5107.
\end{equation}
Since $\xs\in\Delta^{11}$ and $\Sigma_1^*\geq M/3$ , we have 
\begin{equation*}
\Sigma_3^*+\Sigma_4^*< L(k)\bigg|_{k=2}=1-\frac{M}{3}-S(k)\bigg|_{k=2}=L(2)=\frac{4}{639} \left(58-7 \sqrt{2}\right)\approx 0.2634.
\end{equation*} 
Let $Q(k)\big|_{k=2}=L(k)\big|_{k=2}$. We shall examine the  cases below in the rest of the argument:
\begin{equation}\label{def2}
(d)\ \Sigma_3^*\tnr{ and } \Sigma_4^* <\dis Q(k)\bigg|_{k=2},\ (e)\ \Sigma_3^*<\dis Q(k)\bigg|_{k=2}\leq\Sigma_4^*,\  (f)\ \Sigma_4^*<\dis Q(k)\bigg|_{k=2}\leq\Sigma_3^*.
\end{equation}
Assume that ($d$) holds. Using $\sigma(x_l^*)\sigma(Q(k))\big|_{k=2}<\sigma(\Sigma_3^*)\sigma(x_l^*)\leq\alpha(k)\big|_{k=2}= 33$ for ${l\equiv 2}$ and $\sigma(x_l^*)\sigma(Q(k))\big|_{k=2}<\sigma(\Sigma_4^*)\sigma(x_l^*)\leq\alpha(k)\big|_{k=2}= 33$ for ${l\equiv 1}$, we obtain
\begin{equation*}
x_l^*> T(k)\bigg|_{k=2}=\frac{\sigma(Q(k))}{\alpha(k)+\sigma(Q(k))}\bigg|_{k=2}\approx 0.1665
\end{equation*}
for ${l\in I_3\cup I_4=\{7,8,9,10,11,12\}}$.  We claim that $\Sigma_1^*<4/25$. Because otherwise, by the table in (\ref{keven}) for ${l\in I_3\cup I_4}$, we get a contradiction 
\begin{equation}\label{eqnn52}
\begin{multlined}
\Sigma_1^*+\Sigma_2^*+\sum_{{l\equiv 1,2}}x_l^*>\frac{4}{25}+ S(k)\bigg|_{k=2}+\sum_{{l\equiv 1,2}}T(k)\bigg|_{k=2}\\ \shoveleft[1.7cm]{=\frac{4}{25}+ S(k)\bigg|_{k=2}+2(m(k)+n(k))T(k)\bigg|_{k=2}}\\ \shoveleft[3.7cm]{=\frac{4}{25}+ S(2)+2T(2)\approx 1.0037>1.}
\end{multlined}
\end{equation}
By the inequalities $\sigma(x_l^*)\sigma(4/25)<\sigma(\Sigma_1^*)\sigma(x_l^*)\leq\alpha(k)\big|_{k=2}= 33$ for ${l\equiv 0}$, we calculate that $x_l^*>  (21/(21+4\alpha(k)))\big|_{k=2}$. For ${l\in I_3\cup I_4=\{7,8,9,10,11,12\}}$ this implies
\begin{equation}\label{eqnn53}
\begin{multlined}
\Sigma_1^*+\Sigma_2^*+\sum_{{l\equiv 0,1,2}}x_l^*>\frac{M}{3}+ S(k)\bigg|_{k=2}+\sum_{{l\equiv 0}}\frac{21}{21+4\alpha(k)}\bigg|_{k=2}+\sum_{{l\equiv 1,2}}T(k)\bigg|_{k=2}\\ \shoveleft[2.2cm]{=\frac{M}{3}+ S(k)\bigg|_{k=2}+\frac{21(2m(k))}{21+4\alpha(k)}\bigg|_{k=2}+2(m(k)+n(k))T(k)\bigg|_{k=2}} \\ \shoveleft[3.2cm]{=\frac{M}{3}+ S(2)+\frac{42}{153}+2T(2)\approx 1.3441>1,}
\end{multlined}
\end{equation}
a contradiction. The inequalities in (\ref{eqnn52}) and (\ref{eqnn53}) hold for every even $k>2$. This rules out the assumption in ($d$).

Assume that ($e$) holds in (\ref{def2}). Since $\Sigma_3^*< Q(k)\big|_{k=2}$, we obtain $x_l^*> T(k)\big|_{k=2}$ for ${l\equiv 2}$. We claim that $\Sigma_1^*< 3/11$.  Otherwise by the table in (\ref{keven}), for ${l\in I_3}$, we find 
\begin{equation}\label{eqnn54}
\begin{multlined}
\Sigma_1^*+\Sigma_2^*+\sum_{{l\equiv 1,2}}x_l^*+\Sigma_4^*>\frac{3}{11}+ S(k)\bigg|_{k=2}+\sum_{{l\equiv 1}}X(k)\bigg|_{k=2}+\sum_{{l\equiv 2}}T(k)\bigg|_{k=2}+Q(k)\bigg|_{k=2}\\ \shoveleft[1.5cm]{=\frac{3}{11}+ S(k)\bigg|_{k=2}+m(k)X(k)\bigg|_{k=2}+n(k)T(k)\bigg|_{k=2}+Q(k)\bigg|_{k=2}}\\ \shoveleft[3.5cm]{=\frac{3}{11}+ S(2)+X(2)+Q(2)\approx 1.0093>1,}
\end{multlined}
\end{equation}
a contradiction. Using $\sigma(x_l^*)\sigma(3/11)<\sigma(\Sigma_1^*)\sigma(x_l^*)\leq\alpha(k)\big|_{k=2}= 33$ for ${l\equiv 0}$, we calculate that $x_l^*>  (8/(8+3\alpha(k)))\big|_{k=2}$. Then by the table in (\ref{keven}) for ${l\in I_3}$ we obtain
\begin{equation}\label{eqnn56}
\begin{multlined}
\Sigma_1^*+\Sigma_2^*+\sum_{{l\equiv 1,2,0}}x_l^*+\Sigma_4^*\\ \shoveleft[.1cm]{>\frac{M}{3}+S(k)\bigg|_{k=2}+\sum_{{l\equiv 1}}X(k)\bigg|_{k=2}+\sum_{{l\equiv 2}}T(k)\bigg|_{k=2}+\sum_{{l\equiv 0}}\frac{8}{8+3\alpha(k)}\bigg|_{k=2}+Q(k)\bigg|_{k=2}}\\ \shoveleft[.1cm]{=\frac{M}{3}+ S(k)\bigg|_{k=2}+m(k)X(k)\bigg|_{k=2}+n(k)T(k)\bigg|_{k=2}+\frac{8m(k)}{8+3\alpha(k)}\bigg|_{k=2}+Q(k)\bigg|_{k=2}}\\\shoveleft[2.5cm]{=\frac{M}{3}+ S(2)+X(2)+\frac{8}{107}+Q(2)\approx 1.0372>1,}
\end{multlined}
\end{equation}
a contradiction. The inequalities in (\ref{eqnn54}) and (\ref{eqnn56}) hold for every even $k>2$. This shows that ($e$) doesn't hold.

Assume that ($f$) holds in (\ref{def2}). We can use the argument above that proves that ($e$) doesn't hold to show that ($f$) also doesn't hold by interchanging the roles of $\Sigma_3^*$ and $I_3$ with $\Sigma_4^*$ and $I_4$, respectively. We get the same inequalities  in (\ref{eqnn54}) and (\ref{eqnn56}). As a result we conclude that case ($v$) in (\ref{1-7-2}) does not hold.

Assume that case ($vi$) holds in (\ref{1-7-2}). Since $\Sigma_3^*\leq M/3$, we have $x_l^*\geq X(k)\big|_{k=2}$ for every ${l\equiv 2}$ by (\ref{xis2}). Using the inequalities in (\ref{Sg2}) and (\ref{keven}) for ${l\in I_2}$ and the assumption of (C) we find that 
\begin{equation}\label{eqn101}
\Sigma_2^*>  x_i^*+\sum_{{l\equiv 2}}x_l^*>S(k)\bigg|_{k=2}=N+m(k)X(k)\bigg|_{k=2}=N+ X(2)\approx 0.4166.
\end{equation}
Since $\Sigma_1^*\geq M/3$ and $\Sigma_4^*\geq M/3$ we see that 
$$\Sigma_3^*\leq L(k)\bigg|_{k=2}=1-\frac{2M}{3}-S(k)\bigg|_{k=2}=1-\frac{2M}{3}-S(2)\approx 0.1317.$$ 
We must have $x_r^*\leq R(k)\big|_{k=2}= (L(k)/3^{k-1})\big|_{k=2}$ for some ${r\in I_3}$. Since we have $R(k)\big|_{k=2}<X(k)\big|_{k=2}$ for every even $k\geq 2$, we deduce that ${r\not\equiv 2}$. 

Assume that ${r\equiv 0}$. By
$\sigma(\Sigma_1^*)\sigma(R(k))\big|_{k=2}\leq\sigma(\Sigma_1^*)\sigma(x_r^*)\leq\alpha(k)\big|_{k=2}=33$, we obtain 
\begin{equation}\label{eqnn60}
\Sigma_1^*\geq T(k)\bigg|_{k=2}=\frac{\sigma(R(k))}{\alpha(k)+\sigma(R(k))}\bigg|_{k=2}\approx 0.3975.
\end{equation}
We claim that $\Sigma_1^*< 9/25$. Otherwise by the table in (\ref{keven}) for ${l\in I_3=\{7,8,9\}}$ we would find a contradiction which is  
\begin{equation}\label{eqnn57}
\begin{multlined}
\Sigma_1^*+\Sigma_2^*+\sum_{{l\equiv 2}}x_l^*+\Sigma_4^*> \frac{9}{25}+S(k)\bigg|_{k=2}+\sum_{{l\equiv 2}}X(k)\bigg|_{k=2}+\frac{M}{3}\\ \shoveleft[.1cm]{=\frac{9}{25}+S(k)\bigg|_{k=2}+n(k)X(k)\bigg|_{k=2}+\frac{M}{3}=\frac{9}{25}+S(2)+\frac{M}{3}\approx 1.0025>1.}
\end{multlined}
\end{equation}
Then we find that $x_l^*>(16/(16+9\alpha(k)))\big|_{k=2}$ for all ${l\equiv 0}$ by using the inequality 
$\sigma(9/25)\sigma(x_l^*)<\sigma(\Sigma_1^*)\sigma(x_l^*)\leq\alpha(k)\big|_{k=2}=33$. Then for ${l\in I_3}$ we obtain 
\begin{equation}\label{eqnn58}
\begin{multlined}
\Sigma_1^*+\Sigma_2^*+\sum_{{l\equiv 0,2}}x_l^*+\Sigma_4^*> T(k)\bigg|_{k=2}+S(k)\bigg|_{k=2}\\\shoveleft[2cm]{+\sum_{{l\equiv 0}}\frac{16}{16+9\alpha(k)}\bigg|_{k=2}+\sum_{{l\equiv 2}}X(k)\bigg|_{k=2}+\frac{M}{3}}\\ \shoveleft[2cm]{=T(k)\bigg|_{k=2}+S(k)\bigg|_{k=2}+\frac{16m(k)}{16+9\alpha(k)}\bigg|_{k=2}+n(k)X(k)\bigg|_{k=2}+\frac{M}{3}}\\ \shoveleft[4cm]{=T(2)+S(2)+\frac{16}{313}+\frac{M}{3}\approx 1.0911>1,}\end{multlined}
\end{equation}
a contradiction, which shows that $x_r^*>R(k)\big|_{k=2}$ for all ${r\equiv 0}$. 

Assume that ${r\equiv 1}$. By $\sigma(\Sigma_4^*)\sigma(R(k))\big|_{k=2}\leq\sigma(\Sigma_4^*)\sigma(x_r^*)\leq\alpha(k)\big|_{k=2}=33$ we get $\Sigma_4^*\geq T(k)\big|_{k=2}$, defined in (\ref{eqnn60}). Then we derive the same inequalities in (\ref{eqnn58}) since $\Sigma_1^*\geq M/3$ and $\Sigma_4^*\geq T(k)\big|_{k=2}$ switched roles. So we must have $x_r^*>R(k)\big|_{k=2}$ for all ${r\equiv 1}$. 

Assume that ${r\equiv 3}$. By $\sigma(\Sigma_2^*)\sigma(R(k))\big|_{k=2}\leq\sigma(\Sigma_2^*)\sigma(x_r^*)\leq\alpha(k)\big|_{k=2}=33$ we get $\Sigma_2^*\geq T(k)\big|_{k=2}$, defined in (\ref{eqnn60}). But $S(k)>T(k)$ for every even $k\geq 2$. So we shall use $S(k)$ for the calculations. We claim that $\Sigma_2^*<23/50$. Otherwise using the table in (\ref{keven}) for ${l\in I_3}$ we derive a contradiction that is 
\begin{equation}\label{eqnn61}
\begin{multlined}
\Sigma_1^*+\Sigma_2^*+\sum_{{l\equiv 1,2,0}}x_l^*+\Sigma_4^*> \frac{M}{3}+\frac{23}{50}+\sum_{{l\equiv 1}}R(k)\bigg|_{k=2}\\ 
\shoveleft[2.8cm]{+\sum_{{l\equiv 2}}X(k)\bigg|_{k=2}+\sum_{{l\equiv 0}}\frac{16}{16+9\alpha(k)}\bigg|_{k=2}+\frac{M}{3}}\\ \shoveleft[1cm]{=\frac{M}{3}+\frac{23}{50}+m(k)R(k)\bigg|_{k=2}+n(k)X(k)\bigg|_{k=2}+\frac{16m(k)}{16+9\alpha(k)}\bigg|_{k=2}+\frac{M}{3}}\\ \shoveleft[4cm]{=\frac{M}{3}+\frac{23}{50}+R(2)+\frac{16}{313}+\frac{M}{3}\approx 1.0067>1.}\end{multlined}
\end{equation}
Then we compute that $x_l^*>(27/(27+23\alpha(k)))\big|_{k=2}$ for all ${l\equiv 3}$ by using the inequality 
$\sigma(23/50)\sigma(x_l^*)\leq\sigma(\Sigma_1^*)\sigma(x_l^*)\leq\alpha(k)\big|_{k=2}=33$. We claim that $\Sigma_4^*<11/48$. Otherwise by the table in (\ref{keven}) for ${l\in I_3}$ we obtain 
\begin{equation}\label{eqnn62}
\begin{multlined}
\Sigma_1^*+\Sigma_2^*+\sum_{{l\equiv 0,1,2,3}}x_l^*+\Sigma_4^*> \frac{M}{3}+S(k)\bigg|_{k=2}+\sum_{{l\equiv 0}}\frac{16}{16+9\alpha(k)}\bigg|_{k=2}\\ \shoveleft[2cm]{+\sum_{{l\equiv 3}}\frac{27}{27+23\alpha(k)}\bigg|_{k=2}+\sum_{{l\equiv 1}}R(k)\bigg|_{k=2}+\sum_{{l\equiv 2}}X(k)\bigg|_{k=2}}+\frac{11}{48}\\ \shoveleft[2cm]{=\frac{M}{3}+S(k)\bigg|_{k=2}+\frac{16m(k)}{16+9\alpha(k)}\bigg|_{k=2}+\frac{27m(k)}{27+23\alpha(k)}\bigg|_{k=2}}\\ \shoveleft[6cm]{+m(k)R(k)\bigg|_{k=2}+n(k)X(k)\bigg|_{k=2}+\frac{11}{48}}\\ \shoveleft[3cm]{=\frac{M}{3}+S(2)+\frac{16}{313}+\frac{27}{786}+R(2)+\frac{11}{48}\approx 1.0010>1,}
\end{multlined}
\end{equation}
a contradiction. So we find that $x_l^*>(37/(37+11\alpha(k)))\big|_{k=2}$ for all ${l\equiv 1}$ by the inequality 
$\sigma(1/4)\sigma(x_l^*)\leq\sigma(\Sigma_2^*)\sigma(x_l^*)\leq\alpha(k)\big|_{k=2}=33$. Then using the table in (\ref{keven}) for ${l\in I_3}$ we compute that
\begin{equation}\label{eqnn63}
\begin{multlined}
\Sigma_1^*+\Sigma_2^*+\sum_{{l\equiv 0,1,2,3}}x_l^*+\Sigma_4^*> \frac{M}{3}+S(k)\bigg|_{k=2}+\sum_{{l\equiv 0}}\frac{16}{16+9\alpha(k)}\bigg|_{k=2}\\ \shoveleft[1cm]{+\sum_{{l\equiv 3}}\frac{27}{27+23\alpha(k)}\bigg|_{k=2}+\sum_{{l\equiv 1}}\frac{37}{37+11\alpha(k)}\bigg|_{k=2}+\sum_{{l\equiv 2}}X(k)\bigg|_{k=2}}+\frac{M}{3}\\ \shoveleft[1.5cm]{=\frac{M}{3}+S(k)\bigg|_{k=2}+\frac{16m(k)}{16+9\alpha(k)}\bigg|_{k=2}+\frac{27m(k)}{27+23\alpha(k)}\bigg|_{k=2}}\\ \shoveleft[5cm]{+\frac{37m(k)}{37+11\alpha(k)}\bigg|_{k=2}+n(k)X(k)\bigg|_{k=2}+\frac{M}{3}}\\ \shoveleft[3cm]{=\frac{M}{3}+S(2)+\frac{16}{313}+\frac{27}{786}+\frac{37}{400}+\frac{M}{3}\approx 1.0486>1,}
\end{multlined}
\end{equation}
a contradiction. The inequalities in (\ref{eqnn57}), (\ref{eqnn58}), (\ref{eqnn61}), (\ref{eqnn62}) and (\ref{eqnn63}) hold for every even $k>2$. Hence case ($vi$) doesn't hold. 

Assume that case ($vii$) in (\ref{1-7-2}) holds. We use the argument given above to prove that  case ($vi$) doesn't hold to show that  case ($vii$) also doesn't hold by switching the roles of $\Sigma_3^*$ and $I_3$ with $\Sigma_4^*$ and $I_4$, respectively.  We find the inequalities in (\ref{eqn101}), (\ref{eqnn57}), (\ref{eqnn58}), (\ref{eqnn61}), (\ref{eqnn62}) and (\ref{eqnn63}). As a result we find that $x_i^*\geq N>\Sigma_1^*$ ($C$) in (\ref{ABC}) is not the case. Finally the conclusion of the lemma follows.  
\end{proof}

The proof of \fullref{unique1}  is symmetric in the sense that it can be repeated for any other displacement function $f_i$ in $\mathcal{F}^k$ for the choices of indices $j\in\{1,2,3,4\}$ and $i\in I^k=\{1,2,\dots,4\cdot 3^{k-1}\}$  satisfying the hypothesis of the lemma for any $k\geq 2$. Rearrangement of the relevant index sets is required. In fact we have the following statements:
\begin{lemma}\label{unique2}
Under the hypothesis of  \fullref{unique1}, if $k\geq 2$ is even, then we have $\xs\in C_{f_i}$, defined in (\ref{C2}), for each of the following cases:
\begin{equation}\label{cases}
\begin{array}{lllll}
j=1, & i\equiv 0, &  i\in I_2, &  i\in I_3, &  i\in I_4,\\
j=2, & i\equiv 3, &  i\in I_1, &  i\in I_3, &  i\in I_4,\\
j=3, & i\equiv 2, &  i\in I_1, &  i\in I_2, &  i\in I_4,\\
j=4, & i\equiv 1, &  i\in I_1, &  i\in I_2, &  i\in I_3.
\end{array}
\end{equation}
\end{lemma}
\begin{proof}
We reorganise the inequalities in (\ref{ABC}), (\ref{sum1}), (\ref{1-3}), (\ref{sum2}), (\ref{1-7}) and (\ref{1-7-2}) according to each $j$ and $i$ listed in the lemma. Then we follow the computations carried out in the proof of \fullref{unique1} for the chosen $j$ and $i$.  By using the table in (\ref{keven}) we perform analogous computations given  in the proof of  \fullref{unique1} and get the same inequalities in the proof.  This  implies the conclusion of the lemma. 
\end{proof}

\begin{lemma}\label{unique3}
Under the hypothesis of \fullref{unique1}, if $k>2$ is odd, $j=1$ and $i\in I_2$ so that $i\equiv 0$, then  $\xs\in C_{f_i}$, defined in (\ref{C2}).
\end{lemma}
\begin{proof}
Since $j=1$, $i\equiv 0$ and $i\in I_2$, we use the same steps given in the proof of  \fullref{unique1} with the same organisations listed in  (\ref{ABC}), (\ref{sum1}), (\ref{1-3}), (\ref{sum2}), (\ref{1-7}) and (\ref{1-7-2}). 

Because $k>2$ is odd, there are changes to be made in the counts of certain summations. These changes are listed in detail in  \fullref{table5} below. Without changing the orders of  the sums appearing in each of the inequalities and computations, from left to right we replace the terms given under the column '$k\geq 2$, even' with the terms given under the column '$k>2$, odd' for the indicated equations in the proof of  \fullref{unique1}. Let $m=m(k)$ and $n=n(k)$ defined in (\ref{mn}). 

\begin{table}[h!]
\begin{center}
\scalebox{0.62}{
\begin{tabular}{|r|c|c|c|r|c|c|c|}
  \hline
              &               & $k\geq 2$, even        &$k>2$, odd  & & & $k\geq 2$, even & $k>2$, odd  \\
  \hline
   A.i & (\ref{eqnn1}) & $4m+2n$ & $5n+m$  & c&              (\ref{eqnn39}) & $m+n$ & $m+n$  \\
  \hline
    ii & (\ref{eqn32}) &  $m$         & $m$         &   vii &                                      &  $m+n$     &  $m+n$   \\
  \hline
    iii    &                    &  $m$         &  $n$          &   c  &          (\ref{eqnn40}) & $n,\ m$ & $n,\ n$                      \\
  \hline
  ii & (\ref{eqnn4}) & $m$, $n$  & $n$, $n$ &    vii &                                       &  $n,\ m$     &  $n,\ m$  \\
   \hline
   iii &                      & $m$, $n$  & $n$, $n$ &  c &                (\ref{eqnn41}) & $m+n,\ m+n,\ 2m-1$  & $m+n,\ m+n,\ 2n-1$\\
  \hline
  ii & (\ref{eqnn8})  & $m+n,\ 2m$  &   $m+n,\ m+n$ & vii   &                        &  $m+n,\ m+n,\ 2m-1$     & $2n,\ 2n,\ n+m-1$   \\ 
    \hline
 iii     &                      & $m+n,\ 2m$  & $2n,\ 2n$  &  c &        (\ref{eqnn42}) & $m+n,\ m+n,\ 2m,\ 2m-1$ & $m+n,\ m+n,\ 2n,\ 2n-1$    \\ 
  \hline
 B.i  & (\ref{eqnn9})  & $2m+n$               & $m+2n$  &   vii &                           & $m+n,\ m+n,\ 2m,\ 2m-1$  & $2n,\ 2n,\ m+n,\ n+m-1$\\ 
  \hline
 i  & (\ref{eqnn11})  & $6m+3n$  & $2m+7n$ &  C.i   &     (\ref{eqnn43}) &  $9m+3n-1$ & $9n+3m-1$      \\
\hline
ii    & (\ref{eqnn12})  & $m+n$                  & $m+n$ & ii & (\ref{eqnn44}) &  $4m+2n-1$ & $5n+m-1$      \\ 
 \hline
  iii     &                        & $m+n$                  & $2n$     & iii &                           & $4m+2n-1$  & $4n+2m-1$   \\ 
\hline
  ii & (\ref{eqnn14})    & $2m+2n$  & $2m+2n$        &  ii  &  (\ref{eqnn45}) & $5m+n$ & $4n+2m$\\ 
 \hline
  iii   &                        &    $2m+2n$          &  $4n$ &  iii    &                              & $5m+n$ & $5n+m$                           \\ 
  \hline
 ii  & (\ref{eqnn15})  & $2m$, $2m+2n$ &  $2n$, $2m+2n$ & ii  &     (\ref{eqnn47}) & $4m+2n-1$ & $5n+m-1$            \\ 
   \hline
iii       &                      & $2m$, $2m+2n$ &  $m+n$, $4n$  &  iii   &                                 &  $4m+2n-1$ & $4n+2m-1$                   \\
\hline
iv.a &  (\ref{eqnn18})    & $n$ & $n$      &            ii        &                        (\ref{eqn89})                   &  $2m,\ 2m+2n$ & $m+n,\ 4n$ \\ 
   \hline
iv.b  &  (\ref{eqnn20}) & $n$                                 & $n$ &    iii       &                                              & $2m,\ 2m+2n$ & $2n,\ 2m+2n$ \\ 
   \hline
 b &  (\ref{eqnn21}) & $n$, $2m$          & $n$, $m+n$  &     iv    &    (\ref{eqnn51})  & $2m$ & $m+n$ \\ 
  \hline
v.a & (\ref{eqnn22})  & $2m+2n$                               & $4n$    &   v    &       (\ref{Sg2}) & $2m$ & $n+m$ \\ 
 \hline
  a  &  (\ref{eqnn23})  & $2m,\ 2m+2n$                                & $m+n,\ 4n$  & v.d & (\ref{eqnn52}) & $2m+2n$ & $4n$  \\ 
 \hline
v.b &  (\ref{eqnn24})  & $2m+2n$                                & $4n$ & d & (\ref{eqnn53}) & $2m,\ 2m+2n$ & $m+n,\ 4n$ \\ 
\hline
  b &  (\ref{eqnn25})  & $2m+2n,\ 2m$                                & $4n,\ m+n$  & e & (\ref{eqnn54}) & $m,\ n$ & $n,\ n$\\ 
  \hline
 b &  (\ref{eqnn255})  & $2m+2n,\ 2m$                             & $4n,\ m+n$   & f & & $m,\ n$ & $n,\ n$  \\ 
  \hline
 b &   (\ref{eqnn28})  & $2m+2n,\ 2m,\ 2m$           & $4n,\ m+n,\ m+n$  &  e &   (\ref{eqnn56})  & $m,\ n,\ m$ & $n,\ n,\ n$  \\ 
  \hline  
  v.c &  (\ref{eqn67})  & $2m$                               & $m+n$      & f  &    & $m,\ n,\ m$ & $n,\ n,\ m$ \\ 
  \hline 
 c.d   &  (\ref{eqnn29})  & $2m+2n$                   & $4n$    & vi &  (\ref{eqn101}) &  $m$ &  $m$  \\ 
  \hline
  e &  (\ref{eqnn31}) & $m,\ n$ & $n,\ n$  &    vii    &                                             & $m$ & $n$ \\ 
  \hline
   f &                     & $m,\ n$ & $n,\ n$  &  vi   &     (\ref{eqnn57}) & $n$ & $n$                                                                   \\ 
\hline
   e &   (\ref{eqnn32}) & $m,\ n,\ m$ & $n,\ n,\ n$     & vii    &      & $n$ & $n$                                                        \\
  \hline
     f  &                  & $m,\ n,\ m$ & $n,\ n,\ m$     & vi &   (\ref{eqnn58}) & $m,\ n$ &  $n,\ n$   \\
 \hline    
 vi.b &   (\ref{eqnn36}) & $n,\ m$ & $n,\ m$   & vii &   & $m,\ n$ & $m,\ n$   \\ 
  \hline
    vii    &                       & $n,\ m$ & $n,\ n$  & vi &  (\ref{eqnn61}) & $m,\ n,\ m$ & $n,\ n,\ n$ \\ 
  \hline
  b  &  (\ref{eqnn37}) & $m,\ n,\ m$ & $n,\ n,\ m$ &vii &  & $m,\ n,\ m$ & $n,\ n,\ m$ \\ 
\hline
 vii &    & $m,\ n,\ m$ & $n,\ n,\ n$     & vi &  (\ref{eqnn62}) & $m,\ m,\ m,\ n$ &  $n,\ m,\ n,\ n$ \\
 \hline
  b  &   (\ref{eqnn38}) &  $m,\ n,\ m,\ m$ & $n,\ n,\ m,\ n$ & vii &   & $m,\ m,\ m,\ n$ & $m,\ n,\ n,\ n$    \\ 
 \hline
  vii  &   &    $m,\ n,\ m,\ m$  & $n,\ n,\ n,\ m$  & vi  & (\ref{eqnn63}) & $m,\ m,\ m,\ n$ & $n,\ m,\ n,\ n$  \\ 
 \hline
vi.c &   (\ref{eqn79}) & $m$  &  $m$ & vii &  & $m,\ m,\ m,\ n$ & $m,\ n,\ n,\ n$  \\ 
 \hline
 vii &    & $m$  &  $m$ &  &  &  & \\  
\hline
 \end{tabular}}
\caption{List of changes to be made in the proof of \fullref{unique1}.} \label{table5}
\end{center}
\end{table}

After the changes are made, all of the inequalities listed in the table are still satisfied giving the necessary lower bounds for contradictions. The entries in the table without hyperlinks are for the equations in the line one above with a hyperlink for which the computations are not explicitly carried out in the proof of \fullref{unique1}, eg  case ($iii$) in (\ref{1-3}). Hence the conclusion of the lemma holds.  
\end{proof}

\begin{lemma}\label{unique4}
Under the hypothesis of  \fullref{unique1}, if $k>2$ is odd, then we have $\xs\in C_{f_i}$, defined in (\ref{C2}), for each of the cases in (\ref{cases}).
\end{lemma}
\begin{proof}
Given a pair of $j$ and $i$ listed in the lemma, we reorganise the inequalities in (\ref{ABC}), (\ref{sum1}), (\ref{1-3}), (\ref{sum2}), (\ref{1-7}) and (\ref{1-7-2}) accordingly.  By using the terms listed under the columns  '$k>2$, odd' in  \fullref{table5} for the indicated equations, we repeat the arguments presented in the proof of \fullref{unique1} for the chosen $j$ and $i$. 
\end{proof}

We shall continue proving statements about the elements of the sets $C_{f_i}$ for $f_i\in\mathcal{F}^k$. Note that there is no displacement function $f_i\in\mathcal{F}^2$ in the form $f_i=g(\Sigma_j^i,x_i)$ if $k=2$. Therefore in the following statements we shall give the explicit computations for $k=3$. We have the lemmas below:

\begin{lemma}\label{unique5}
Let $\mathcal{F}^k=\{f_i\}$ for ${i\in I^k}$ be the set of displacement functions listed in \fullref{dispfunc} and $F^k$ be as in (\ref{F}). Let $\xs$ be a point in $\Delta^{d-1}$ so that $  \alpha_*=F^k(\xs)$ for $d=4\cdot 3^{k-1}$. Let $f_i\in\mathcal{F}^k$ be of the form $f_i=g(\Sigma_j^i,x_i)$ for $j\in\{1,2,3,4\}$ and $i\in I^k=I_1\cup I_2\cup I_3\cup I_4$, where $\Sigma_j^i(\tb{x})$ and $g$ are defined in  (\ref{fg}), respectively. If $k>2$ is {odd}, $j=1$ and $i\in I_1$ such that $i\equiv 0$,
then  $\xs\in C_{f_i}$, defined in (\ref{C1}).
\end{lemma}
\begin{proof}
Assume on the contrary that $\xs\notin C_{f_i}$. Then by the definition of $C_{f_i}$ we have 
\begin{equation}\label{U5E1}
\Sigma_1^i(\xs)+(2-\Sigma_1^i(\xs))x_i^*-(x_i^*)^2\geq 3/4.
\end{equation}
Let $\Sigma_1^*=\Sigma_1(\xs)$,  $\Sigma_2^*=\Sigma_2(\xs)$,  $\Sigma_3^*=\Sigma_3(\xs)$ and  $\Sigma_4^*=\Sigma_4(\xs)$ defined in (\ref{sigma}), where $\Sigma_1^*+\Sigma_2^*+\Sigma_3^*+\Sigma_4^*=1$  since $\xs\in\Delta^{d-1}$.  We have $\Sigma_1^i(\xs)+x_i^*=\Sigma_1^*$. Also let $N=\frac{1}{4} \left(3-\sqrt{3}\right)\approx 0.3170$. Remember that $\sigma(x)=1/x-1$. We consider the cases: 
\begin{equation}\label{ABC2}
\begin{array}{c}
(A)\ \Sigma_1^i(\xs)\geq N,\ x_i^*\geq N,\ (B)\ \Sigma_1^i(\xs)\geq N > x_i^*,\ (C)\ x_i^*\geq N > \Sigma_1^i(\xs).
\end{array}
\end{equation}
Assume without loss of generality that $k=3$. Assume that (A) holds.  We derive that $\Sigma_1^*\geq 2N$. Then we have the inequality 
\begin{equation}\label{sum3}
\Sigma_2^*+\Sigma_3^*+\Sigma_4^*\leq M=1-2N=\frac{1}{2} \left(\sqrt{3}-1\right)\approx 0.3660,
\end{equation}
which implies the following cases:
\begin{equation}\label{1-7-3}
\begin{array}{llll}
(i)  & \Sigma_2^*\leq M/3, & \Sigma_3^*\leq M/3, & \Sigma_4^*\leq M/3,\\
(ii) & \Sigma_2^*\leq M/3, & \Sigma_3^*\leq M/3, & \Sigma_4^*\geq M/3,\\
(iii) & \Sigma_2^*\leq M/3, & \Sigma_3^*\geq M/3, & \Sigma_4^*\leq M/3,\\
(iv) & \Sigma_2^*\geq M/3, & \Sigma_3^*\leq M/3, & \Sigma_4^*\leq M/3, \\
(v) & \Sigma_2^*\leq M/3, & \Sigma_3^*\geq M/3, & \Sigma_4^*\geq M/3,\\
(vi) & \Sigma_2^*\geq M/3, & \Sigma_3^*\leq M/3, & \Sigma_4^*\geq M/3,\\
(vii) & \Sigma_2^*\geq M/3, & \Sigma_3^*\geq M/3, & \Sigma_4^*\leq M/3.
\end{array}
\end{equation}
Assume that ($i$) holds. By using $\sigma(M/3)\sigma(x_l^*)\leq\sigma(\Sigma_r^*)\sigma(x_l^*)\leq\alpha(k)\big|_{k=3}=105$ for $r=2,3,4$ we find that 
\begin{equation}\label{xis3}
x_l^*\geq X(k)\bigg|_{k=3}=\frac{\sigma(M/3)}{\alpha(k)+\sigma(M/3)}\bigg|_{k=3}=\frac{3-M}{(\alpha(k)-1) M+3}\bigg|_{k=3}\approx 0.0641
\end{equation}
for every ${l\in I_2\cup I_3\cup I_4=\{10,11,\dots,36\}}$ so that ${l\equiv 1,2,3}$. Then using the table in (\ref{kodd}) for ${l\in I_2\cup I_3\cup I_4}$ we see that 
\begin{equation}\label{2eqnn1}
\begin{multlined}
\sum_{{l\equiv 1,2,3}}x_l^*\geq \sum_{{l\equiv 1,2,3}}X(k)\bigg|_{k=3}=(2m(k)+7n(k))X(k)\bigg|_{k=3}=20X(3)\approx 1.2828>1,
\end{multlined}
\end{equation}
a contradiction. The inequality in (\ref{2eqnn1}) holds for every odd $k>3$. Therefore case ($i$) doesn't hold. 

Assume that ($ii$) holds in (\ref{1-7-3}). By $\sigma(M/3)\sigma(x_l^*)\leq\sigma(\Sigma_r^*)\sigma(x_l^*)\leq\alpha(k)\big|_{k=3}=105$ for $r=2,3$ we obtain $x_l^*\geq X(k)\big|_{k=3}$ for every ${l\in I_2\cup I_3=\{10,11,\dots,27\}}$ so that ${l\equiv 2,3}$ by (\ref{xis3}). Then using the table in (\ref{kodd}) for ${l\in I_2\cup I_3}$ we see that 
\begin{equation}\label{2eqnn2}
\begin{multlined}
\Sigma_1^*+\sum_{{l\equiv 2,3}}x_l^*+\Sigma_4^*\geq 2N+\sum_{{l\equiv 2,3}}X(k)\bigg|_{k=3}+\frac{M}{3}\\ \shoveleft{=2N+2(m(k)+n(k))X(k)\bigg|_{k=3}+\frac{M}{3}=2N+10X(3)+\frac{M}{3}\approx 1.3974>1,}
\end{multlined}
\end{equation}
a contradiction. The inequality in (\ref{2eqnn2}) holds for every odd $k>3$. Therefore case ($ii$) doesn't hold. 

Assume that ($iii$) holds in (\ref{1-7-3}). We can repeat the argument given above for this case as well. We need to switch the role of ${I_2\cup I_3}$ with ${I_2\cup I_4=\{10,\dots,18,28,\dots,36\}}$ because, $\Sigma_2^*\leq M/3$ and $\Sigma_4^*\leq M/3$.  By using the table in (\ref{kodd}) for ${l\in I_2\cup I_4}$ we get 
\begin{equation}\label{2eqnn3}
\begin{multlined}
\Sigma_1^*+\sum_{{l\equiv 3,1}}x_l^*+\Sigma_3^*\geq 2N+\sum_{{l\equiv 3,1}}X(k)\bigg|_{k=3}+\frac{M}{3}\\ \shoveleft[1.0cm]{=2N+4n(k)X(k)\bigg|_{k=3}+\frac{M}{3}=2N+8X(3)+\frac{M}{3}\approx 1.2691>1,}
\end{multlined}
\end{equation}
a contradiction. The inequality above holds for every odd $k>3$. So case ($iii$) doesn't hold.

Assume that ($iv$) holds in (\ref{1-7-3}). We use the same argument used in case ($iii$) by switching the role of $I_2\cup I_4$ with $I_3\cup I_4$. Then we get the same inequality in (\ref{2eqnn3}) which hold for every odd $k\geq 3$. This is because by the table in (\ref{kodd}) the number $4n(k)$ of elements in ${I_2\cup I_4}$ equivalent to $1$ or $3$ is the same as the number of elements in ${I_3\cup I_4}$ equivalent to $2$ or $1$. So case ($iv$) doesn't hold. 

Assume that ($v$) holds in (\ref{1-7-3}). Since $\Sigma_2^*\leq M/3$ in this case, we calculate that $x_l^*\geq X(k)\big|_{k=3}$ for every ${l\in I_2=\{10,11,\dots,18\}}$ so that ${l\equiv 3}$ by (\ref{xis3}). Then for ${l\in I_2}$ we find a contradiction which is
\begin{equation}\label{2eqnn5}
\begin{multlined}
\Sigma_1^*+\sum_{{l\equiv 3}}x_l^*+\Sigma_3^*+\Sigma_4^*> 2N+\sum_{{l\equiv 3}}X(k)\bigg|_{k=3}+\frac{2M}{3}\\ \shoveleft[1.0cm]{=2N+n(k)X(k)\bigg|_{k=2}+\frac{2M}{3}=2N+2X(3)+\frac{2M}{3}\approx 1.0063>1.}
\end{multlined}
\end{equation}
Since the inequality in (\ref{2eqnn5}) holds for every odd $k>3$, case ($v$) doesn't hold.  

The argument given above for case ($v$) also shows that cases ($vi$) and ($vii$) don't hold. Because we can repeat the computations for case ($iv$) by switching the role of $I_2$  with $I_3$ for case ($vi$). For case ($vii$), we switch the role of $I_2$ with $I_4$.  By the table in (\ref{kodd}) we obtain the same inequalities in  (\ref{2eqnn5}).  As a result we conclude that $\Sigma_1^i(\xs)\geq N$ and $x_i^*\geq N$ (A) in (\ref{ABC2}) is not the case.

Assume that (B) holds in (\ref{ABC2}).  We know that $\Sigma_1^i(\xs)\geq N$. Then we have the inequality 
\begin{equation}\label{sum5}
x_i^*+\Sigma_2^*+\Sigma_3^*+\Sigma_4^*\leq M=1-N=\frac{1}{4} \left(1+\sqrt{3}\right)\approx 0.6830.
\end{equation}
Note that if two of the terms $\Sigma_2^*$, $\Sigma_3^*$ or $\Sigma_4^*$  are less than or equal to $M/4$ simultaneously, then the third one cannot be less than or equal to $M/4$. Because by using the inequality $\sigma(M/4)\sigma(x_l^*)\leq\sigma(\Sigma_r^*)\sigma(x_l^*)\leq\alpha(k)\big|_{k=3}=105$ for $r=2,3,4$ we find that 
\begin{equation}\label{xis4}
x_l^*\geq X(k)\bigg|_{k=3}=\frac{\sigma(M/4)}{\alpha(k)+\sigma(M/4)}\bigg|_{k=3}=\frac{4-M}{(\alpha(k)-1) M+4}\bigg|_{k=3}\approx 0.0442
\end{equation}
for every ${l\in I_2\cup I_3\cup I_4=\{10,11,\dots,36\}}$ so that ${l\equiv 1,2,3}$. Then using the table in (\ref{kodd}) for ${l\in I_4}$, ${l\in I_3}$ and ${l\in I_2}$, respectively, in the each of following inequalities we see that 
\begin{equation}\label{2eqnn6_1}
\sum_{{l\equiv 2,3}}x_l^*\geq 2n(k)X(k)\bigg|_{k=3}\approx 0.1768>\frac{M}{4},
\end{equation}
\begin{equation}\label{2eqnn6_2}
\sum_{{l\equiv 1,3}}x_l^*\geq (n(k)+m(k))X(k)\bigg|_{k=3}\approx 0.2210>\frac{M}{4},
\end{equation}
\begin{equation}\label{2eqnn6_3}
\sum_{{l\equiv 1,2}}x_l^*\geq (n(k)+m(k))X(k)\bigg|_{k=3}\approx 0.2210>\frac{M}{4}.
\end{equation}
The inequalities in (\ref{2eqnn6_1}), (\ref{2eqnn6_2}) and (\ref{2eqnn6_3}) hold for every odd $k>3$.
This  implies the following first $6$ of $13$ cases:
\begin{equation}\label{1-7-4}
\begin{array}{lllll}
(i) & x_i^*\geq M/4& \Sigma_2^*\leq M/4, & \Sigma_3^*\leq M/4, & \Sigma_4^*\geq M/4,\\
(ii) & x_i^*\geq M/4& \Sigma_2^*\leq M/4, & \Sigma_3^*\geq M/4, & \Sigma_4^*\leq M/4,\\
(iii) & x_i^*\geq M/4&  \Sigma_2^*\geq M/4, & \Sigma_3^*\leq M/4, & \Sigma_4^*\leq M/4,\\
(iv) & x_i^*\geq M/4 & \Sigma_2^*\leq M/4, & \Sigma_3^*\geq M/4, & \Sigma_4^*\geq M/4,\\
(v) & x_i^*\geq M/4& \Sigma_2^*\geq M/4, & \Sigma_3^*\leq M/4, & \Sigma_4^*\geq M/4,\\
(vi)& x_i^*\geq M/4 & \Sigma_2^*\geq M/4, & \Sigma_3^*\geq M/4, & \Sigma_4^*\leq M/4.
\end{array}
\end{equation}

Assume that ($i$) holds in (\ref{1-7-4}). Since $\Sigma_2^*\leq M/4$ and $\Sigma_3^*\leq M/4$, we obtain that  $x_l^*\geq X(k)\big|_{k=3}$ for all ${l\in I^k}$ such that ${l\equiv 3,2}$ by (\ref{xis4}).  Then we compute for ${l\in I_2\cup I_3=\{10,11,\dots,27\}}$ that
\begin{equation}\label{2eqnn7}
\begin{multlined}
\Sigma_1^i(\xs)+x_i^*+\sum_{{l\equiv 2,3}}x_l^*+\Sigma_4^*\geq N+\frac{M}{2}+ \sum_{{l\equiv 2,3}}X(k)\bigg|_{k=3}\\ \shoveleft[0cm]{=N+\frac{M}{2}+2(m(k)+n(k))X(k)\bigg|_{k=3}=N+\frac{M}{2}+10X(3)\approx 1.1006>1,}
\end{multlined}
\end{equation}
a contradiction. The inequality in (\ref{2eqnn7}) holds for every odd $k>3$. So case ($i$) doesn't hold. 

Assume that ($ii$) holds in (\ref{1-7-4}). Since $\Sigma_2^*\leq M/4$ and $\Sigma_4^*\leq M/4$, we know that  $x_l^*\geq X(k)\big|_{k=3}$ for all ${l\in I^k}$ such that ${l\equiv 3,1}$ by (\ref{xis4}).  Then using the table in (\ref{kodd}) for ${l\in I_2\cup I_4=\{10,\dots,18,28,\dots,36\}}$ we calculate that
\begin{equation}\label{2eqnn8}
\begin{multlined}
\Sigma_1^i(\xs)+x_i^*+\sum_{{l\equiv 3,1}}x_l^*+\Sigma_3^*\geq N+\frac{M}{2}+ \sum_{{l\equiv 3,1}}X(k)\bigg|_{k=3}\\ \shoveleft[0cm]{=N+\frac{M}{2}+4n(k)X(k)\bigg|_{k=3}=N+\frac{M}{2}+8X(3)\approx 1.0121>1,}
\end{multlined}
\end{equation}
a contradiction.  The inequality in (\ref{2eqnn8}) holds for every odd $k>3$. Therefore, case ($ii$) doesn't hold.

For case ($iii$) in (\ref{1-7-4}) we get the same inequality in (\ref{2eqnn8}) by replacing the index set ${I_2\cup I_4}$ with the index set ${I_3\cup I_4=\{19,\dots,36\}}$. Since the inequalities in (\ref{2eqnn8}) hold for every odd $k>3$, case ($iii$) doesn't hold. 

Assume that case ($iv$) holds. By the inequality $\Sigma_2^*\leq M/4$ in this case, we obtain $x_l^*\geq X(k)\big|_{k=3}$ for every ${l\equiv 3}$ by (\ref{xis4}). We claim that $\Sigma_3^*<13/50$. Because otherwise by the table in (\ref{kodd}) for ${l\in I_2}$ we would get 
\begin{equation}\label{2eqnn9}
\begin{multlined}
\Sigma_1^i(\xs)+x_i^*+\sum_{{l\equiv 3}}x_l^*+\Sigma_3^*+\Sigma_4^*\geq N+\frac{M}{4}+ \sum_{{l\equiv 3}}X(k)\bigg|_{k=3}+\frac{13}{50}+\frac{M}{4}\\ \shoveleft[0cm]{=N+\frac{M}{2}+n(k)X(k)\bigg|_{k=3}+\frac{13}{50}=N+\frac{M}{2}+2X(3)+\frac{13}{50}\approx 1.0069>1,}
\end{multlined}
\end{equation}
a contradiction.  A similar contradiction arises if we assume that $\Sigma_4^*\geq 13/50$ by the same inequality in (\ref{2eqnn9}). Then by  
$\sigma(13/50)\sigma(x_l^*)<\sigma(\Sigma_r^*)\sigma(x_l^*)\leq\alpha(k)\big|_{k=3}=105$ for $r=3,4$ we obtain $x_l^*> (37/(37+13\alpha(k)))\big|_{k=3}$ for every ${l\in I_2=\{10,11,\dots,18\}}$ so that ${l\equiv 2,1}$.  By the table in (\ref{kodd}) for ${l\in I_2}$ we calculate that 
\begin{equation}\label{2eqnn10}
\begin{multlined}
\Sigma_1^*+\sum_{{l\equiv 3,2,1}}x_l^*+\Sigma_3^*+\Sigma_4^*\geq N+\frac{3M}{4}+ \sum_{{l\equiv 3}}X(k)\bigg|_{k=3}+\sum_{{l\equiv 2,1}}\frac{37}{37+13\alpha(k)}\bigg|_{k=3}\\ \shoveleft[4cm]{=N+\frac{3M}{4}+n(k)X(k)\bigg|_{k=3}+\frac{37(n(k)+m(k))}{37+13\alpha(k)}\bigg|_{k=3}}\\ \shoveleft[4.5cm]{=N+\frac{3M}{4}+2X(3)+\frac{185}{1402}\approx 1.0496>1,}
\end{multlined}
\end{equation}
a contradiction, where $\Sigma_1^i(\xs)+x_i^*=\Sigma_1^*$. The inequality in (\ref{2eqnn10}) holds for every odd $k>3$. Hence case ($iv$) doesn't hold. 

For the case ($v$) we can use the argument given above for case ($iv$) by switching the role of ${I_2}$ with ${I_3}$. We obtain the same inequalities in (\ref{2eqnn9}) and (\ref{2eqnn10}). Therefore case ($v$) also doesn't hold.  

For case ($vi$) we again follow the same computations given above for case ($iv$) by switching the role of ${I_2}$ with ${I_4=\{28,\dots,36\}}$. By using the table in (\ref{kodd}) we find the same inequality in (\ref{2eqnn9}). But we need to change $n$ and $n+m$ in (\ref{2eqnn10}) with $n$ and $2n$, respectively. Resulting sum will still be greater than $1$ for every odd $k\geq 3$. As a result case ($vi$) doesn't hold either. So we ruled out the first $6$ cases in (\ref{1-7-4}) out of $13$ possible cases. 

Under the assumption of (B) in (\ref{ABC2}) we have the following $7$ additional cases:
\begin{equation}\label{1-7-5}
\begin{array}{lllll}
(vii) & x_i^*\leq M/4& \Sigma_2^*\leq M/4, & \Sigma_3^*\leq M/4, & \Sigma_4^*\geq M/4,\\
(viii) & x_i^*\leq M/4& \Sigma_2^*\leq M/4, & \Sigma_3^*\geq M/4, & \Sigma_4^*\leq M/4,\\
(ix) & x_i^*\leq M/4&  \Sigma_2^*\geq M/4, & \Sigma_3^*\leq M/4, & \Sigma_4^*\leq M/4,\\
(x) & x_i^*\leq M/4 & \Sigma_2^*\leq M/4, & \Sigma_3^*\geq M/4, & \Sigma_4^*\geq M/4,\\
(xi) & x_i^*\leq M/4& \Sigma_2^*\geq M/4, & \Sigma_3^*\leq M/4, & \Sigma_4^*\geq M/4,\\
(xii)& x_i^*\leq M/4 & \Sigma_2^*\geq M/4, & \Sigma_3^*\geq M/4, & \Sigma_4^*\leq M/4,\\
(xiii) & x_i^*\leq M/4 & \Sigma_2^*\geq M/4, & \Sigma_3^*\geq M/4, & \Sigma_4^*\geq M/4.
\end{array}
\end{equation} 
Before we proceed to examine the cases in this group, we derive the following inequality from (\ref{U5E1}). Since $x_i^*\leq M/4$ and $\Sigma_1^i(\xs)-(x_i^*)^2<\Sigma_1^i(\xs)$, we obtain 
\begin{equation}\label{Sig1i}
\Sigma_1^i(\xs)\geq L=\frac{3-2M}{4-M}\approx 0.4926.
\end{equation}

Assume that case ($vii$) holds. Since $\Sigma_2^*\leq M/4$ and $\Sigma_3^*\leq M/4$, we know that  $x_l^*\geq X(k)\big|_{k=3}$ for every ${l\equiv 2,3}$ by (\ref{xis4}). We claim that $\Sigma_4^*< 4/25$. Assume otherwise. Then by the table in (\ref{kodd}) for ${l\in I_2\cup I_3=\{10,11,\dots, 27\}}$ we see that
\begin{equation}\label{2eqnn11}
\begin{multlined}
\Sigma_1^i(\xs)+\sum_{{l\equiv 3,2}}x_l^*+\Sigma_4^*\geq L+ \sum_{{l\equiv 3,2}}X(k)\bigg|_{k=3}+\frac{4}{25}\\ \shoveleft[0cm]{=L+2(n(k)+m(k))X(k)\bigg|_{k=3}+\frac{4}{25}=L+10X(3)+\frac{4}{25}\approx 1.0947>1,}
\end{multlined}
\end{equation}
a contradiction. By  
$\sigma(4/25)\sigma(x_l^*)<\sigma(\Sigma_4^*)\sigma(x_l^*)\leq\alpha(k)\big|_{k=3}=105$ we obtain that $x_l^*> (21/(21+4\alpha(k)))\big|_{k=3}$ for every ${l\in I_2\cup I_3\cup I_4=\{10,11,\dots, 36\}}$ so that ${l\equiv 1}$. Then for ${l\in I_2\cup I_3\cup I_4}$  we calculate that 
\begin{equation}\label{2eqnn12}
\begin{multlined}
\Sigma_1^i(\xs)+\sum_{{l\equiv 3,2,1}}x_l^*\geq L+ \sum_{{l\equiv 3,2}}X(k)\bigg|_{k=3}+\sum_{{l\equiv 1}}\frac{21}{21+4\alpha(k)}\bigg|_{k=3}\\ \shoveleft[2.5cm]{=L+(2m(k)+4n(k))X(k)\bigg|_{k=3}+\frac{21(3n(k))}{21+4\alpha(k)}\bigg|_{k=3}}\\ \shoveleft[3.5cm]{=L+14X(3)+\frac{126}{441}\approx 1.3972>1,}
\end{multlined}
\end{equation}
a contradiction. The inequalities in (\ref{2eqnn11}) and (\ref{2eqnn12}) hold for every odd $k>3$. Hence case ($vii$) doesn't hold. 

For cases ($viii$) and ($ix$) we can repeat the computations given above for case ($vii$) by switching the roles of $I_2$ and $I_3$ with $I_2$ and $I_4$ respectively for case ($viii$) and, with $I_3$ and $I_4=\{28,\dots,36\}$ respectively for case ($ix$). For both of the cases we obtain the same inequality in (\ref{2eqnn11}) showing that $\Sigma_3^*<4/25$ and $\Sigma_4^*<4/25$. In the inequality in (\ref{2eqnn12}) we need to replace $2m+4n$ and $3n$ with $m+5n$ and $m+2n$ respectively using the table in (\ref{kodd}). Resulting inequalities hold for every odd $k\geq 3$. So both of these cases also don't hold. 

Assume that case ($x$) holds in (\ref{1-7-5}). Since $\Sigma_2^*\leq M/4$, we get $x_l^*\geq X(k)\big|_{k=3}$ for all ${l\equiv 3}$ by (\ref{xis4}). We claim that $\Sigma_3^*< 1/4$. Assume the contrary. Then for ${l\in I_2}$ by the table in (\ref{kodd}) we compute that 
\begin{equation}\label{2eqnn13}
\begin{multlined}
\Sigma_1^i(\xs)+\sum_{{l\equiv 3}}x_l^*+\Sigma_3^*+\Sigma_4^*\geq L+ \sum_{{l\equiv 3}}X(k)\bigg|_{k=3}+\frac{1}{4}+\frac{M}{4}\\ \shoveleft[0cm]{=L+n(k)X(k)\bigg|_{k=3}+\frac{1}{4}+\frac{M}{4}=L+2X(3)+\frac{1}{4}+\frac{M}{4}\approx 1.0018>1,}
\end{multlined}
\end{equation}
a contradiction. Using a similar argument above we can also show that $\Sigma_4^*<1/4$. So by $\sigma(1/4)\sigma(x_l^*)<\sigma(\Sigma_r^*)\sigma(x_l^*)\leq\alpha(k)\big|_{k=3}=105$ for $r=3,4$, we derive that  $x_l^*>(3/(3+\alpha(k)))\big|_{k=3}$ for every ${l\equiv 2,1}$. 
Then for ${l\in I_2=\{10,\dots,18\}}$ by the table in (\ref{kodd}) we calculate that 
\begin{equation}\label{2eqnn16}
\begin{multlined}
\Sigma_1^i(\xs)+\sum_{{l\equiv 1,2,3}}x_l^*+\Sigma_3^*+\Sigma_4^*\geq L+ \sum_{{l\equiv 3}}X(k)\bigg|_{k=3}+\sum_{{l\equiv 1,2}}x_l^*+\frac{2M}{4}\\ \shoveleft[2cm]{=L+n(k)X(k)\bigg|_{k=3}+\frac{3(m(k)+n(k))}{3+\alpha(k)}\bigg|_{k=3}+\frac{2M}{4}}\\ \shoveleft[2.5cm]{=L+2X(3)+\frac{15}{108}+\frac{M}{2}\approx 1.0614>1,}
\end{multlined}
\end{equation}
a contradiction. The inequalities (\ref{2eqnn13}) and (\ref{2eqnn16}) hold for every odd $k>3$. Hence case ($x$) doesn't hold. 

Assume that case ($xi$) holds in (\ref{1-7-5}). By switching the role of $I_2$ in case ($x$) with $I_3=\{19,\dots,27\}$ we repeat the same argument given for case ($x$) to show that case ($xi$) doesn't hold as well. Using the table in (\ref{kodd}) we obtain the same inequalities in (\ref{2eqnn13}) and (\ref{2eqnn16}) which show that this case also doesn't hold. 

For case ($xii$) in (\ref{1-7-5})  we again repeat an analog of the argument given above for case ($x$). We need to replace $n$ and $n+m$ in (\ref{2eqnn16}) with $n$ and $2n$, respectively. Then the resulting inequality holds for every odd $k\geq 3$.  Hence case ($xii$) doesn't hold.

It is clear that case ($xiii$) in (\ref{1-7-5}) doesn't hold. Because we derive the following inequality otherwise
\begin{equation*}
\Sigma_1^i(\xs)+\Sigma_2^*+\Sigma_3^*+\Sigma_4^*\geq L+\frac{3M}{4}\approx 1.0049>1,
\end{equation*}
a contradiction. As a conclusion $\Sigma_1^i(\xs)\geq N> x_i^*$ (B) in (\ref{ABC2}) is not the case. 

Assume that (C) holds in (\ref{ABC2}). Since we have $x_i^*\geq N>\Sigma_1^i(\xs)$, we derive that
\begin{equation}\label{sum4}
\Sigma_1^i(\xs)+\Sigma_2^*+\Sigma_3^*+\Sigma_4^*\leq M=1-N=\frac{1}{4} \left(1+\sqrt{3}\right)\approx 0.6830.
\end{equation}
Assume that $\Sigma_1^i(\xs)\geq M/4$. The arguments we presented above to show that cases ($i$)-($vi$) in (\ref{1-7-4}) don't hold can be repeated by switching the roles of $x_i^*$ and $\Sigma_1^i(\xs)$. Therefore cases with the assumptions listed in ($i$)-($vi$) for $\Sigma_2^*$, $\Sigma_3^*$ and $\Sigma_4^*$ don't hold. 

If $\Sigma_1^i(\xs)\leq M/4$, then any two of the terms $\Sigma_2^*$, $\Sigma_3^*$ and $\Sigma_4^*$ cannot be less than or equal to $M/4$ simultaneously by the inequalities in (\ref{2eqnn6_1}),  (\ref{2eqnn6_2}) and (\ref{2eqnn6_3}). Therefore it is enough to consider the following cases:
\begin{equation}\label{1-7-7}
\begin{array}{lllll}
(i) & \Sigma_1^i(\xs)\leq M/4 & \Sigma_2^*\leq M/4, & \Sigma_3^*\geq M/4, & \Sigma_4^*\geq M/4,\\
(ii) & \Sigma_1^i(\xs)\leq M/4& \Sigma_2^*\geq M/4, & \Sigma_3^*\leq M/4, & \Sigma_4^*\geq M/4,\\
(iii)& \Sigma_1^i(\xs)\leq M/4 & \Sigma_2^*\geq M/4, & \Sigma_3^*\geq M/4, & \Sigma_4^*\leq M/4,\\
(iv) & \Sigma_1^i(\xs)\leq M/4 & \Sigma_2^*\geq M/4, & \Sigma_3^*\geq M/4, & \Sigma_4^*\geq M/4.
\end{array}
\end{equation} 
Before we proceed to studying these cases, we derive the following lower bond using the inequality in (\ref{U5E1}). Since $\Sigma_1^i(\xs)\leq M/4$ and $(2-\Sigma_1^i(\xs))x_i^*<2x_i^*$, we find that 
\begin{equation}\label{xi}
x_i^*\geq L=\frac{1}{4} \left(4-\sqrt{5+\sqrt{3}}\right)\approx 0.3513.
\end{equation}
Assume that case ($i$) holds. We already know by (\ref{xis4}) that $x_l^*\geq X(k)\big|_{k=3}$ for every ${l\equiv 3}$ because $\Sigma_2^*\leq M/4$. We claim that $\Sigma_3^*<31/100$. Because otherwise using the table in (\ref{kodd}) for ${l\in I_1\cup I_2}$ we would have 
\begin{equation}\label{2eqnn17}
\begin{multlined}
x_i^*+\sum_{{l\equiv 3}}x_l^*+\Sigma_3^*+\Sigma_4^*\geq L+\sum_{{l\equiv 3}}X(k)\bigg|_{k=3}+\frac{31}{100}+\frac{M}{4}\\
\shoveleft[.1cm]{=L+2n(k)X(k)\bigg|_{k=3}+\frac{31}{100}+\frac{M}{4}=L+4X(3)+\frac{31}{100}+\frac{M}{4}\approx 1.0089>1,}
\end{multlined}\end{equation}
a contradiction. By replacing the roles of $\Sigma_3^*$ and $\Sigma_4^*$ in the inequality above we also see that $\Sigma_4^*<31/100$. By $\sigma(31/100)\sigma(x_l^*)<\sigma(\Sigma_r^*)\sigma(x_l^*)\leq\alpha(k)\big|_{k=3}=105$ for ${r=3,4}$, we get $x_l^*>(69/(69+\alpha(k)))\big|_{k=3}$ for every ${l\equiv 1,2}$.
For ${l\in I_1\cup I_2}$ this implies that
\begin{equation}\label{2eqnn20}
\begin{multlined}
x_i^*+\sum_{{l\equiv 1,2,3}}x_l^*+\Sigma_3^*+\Sigma_4^*\geq L+\sum_{{l\equiv 1,2}}\frac{69}{69+31\alpha(k)}\bigg|_{k=3}+\sum_{{l\equiv 3}}X(k)\bigg|_{k=3}+\frac{2M}{4}\\
\shoveleft[4cm]{=L+\frac{69(2m(k)+2n(k))}{69+31\alpha(k)}\bigg|_{k=3}+2n(k)X(k)\bigg|_{k=3}+\frac{M}{2}}\\
\shoveleft[6cm]{=L+\frac{115}{554}+4X(3)+\frac{M}{2}\approx 1.0773>1,}
\end{multlined}
\end{equation} 
a contradiction. The inequalities in (\ref{2eqnn17}) and (\ref{2eqnn20}) hold for every odd $k>3$. Hence case ($i$) doesn't hold. 

The argument above used to show that case ($i$) doesn't hold can be repeated to examine cases ($ii$) and ($iii$) also. We need to replace the index set ${I_1\cup I_2}$ with ${I_1\cup I_3}$ for case ($ii$) and replace it with ${I_1\cup I_4}$ for case ($iii$). For case ($ii$) the inequalities in (\ref{2eqnn17}) and (\ref{2eqnn20}) stay the same. For case ($iii$) we need to interchange $2n$ with $m+n$ in (\ref{2eqnn17}) and, $2m+2n$ and $2n$ with $4n$ and $m+n$, respectively, in (\ref{2eqnn20}). After these changes the resulting inequalities still hold for every odd $k>3$. Therefore, these cases don't hold. 

Assume that ($iv$) holds in (\ref{1-7-7}). We claim that $\Sigma_2^*<31/100$. Assume otherwise. Then we compute that 
\begin{equation}\label{2eqnn18}
\begin{multlined}
x_i^*+\Sigma_2^*+\Sigma_3^*+\Sigma_4^*\geq L+\frac{31}{100}+\frac{2M}{4}\approx 1.0028>1,
\end{multlined}\end{equation}
a contradiction. By replacing the role of $\Sigma_2^*$ with $\Sigma_3^*$ and then with $\Sigma_4^*$ in the inequality above we also see that $\Sigma_3^*<31/100$ and $\Sigma_4^*<31/100$. By using the inequalities $\sigma(31/100)\sigma(x_l^*)<\sigma(\Sigma_r^*)\sigma(x_l^*)\leq\alpha(k)\big|_{k=3}=105$ for ${r=2,3,4}$, we calculate that $x_l^*>(69/(69+\alpha(k)))\big|_{k=3}$ for every ${l\equiv 1,2,3}$.
By the table in (\ref{kodd}) for ${l\in I_1}$ we find
\begin{equation}\label{2eqnn21}
\begin{multlined}
x_i^*+\sum_{{l\equiv 1,2,3}}x_l^*+\Sigma_2^*+\Sigma_3^*+\Sigma_4^*\geq L+\sum_{{l\equiv 1,2,3}}\frac{69}{69+31\alpha(k)}\bigg|_{k=3}+\frac{3M}{4}\\
\shoveleft[4.5cm]{=L+\frac{69(m(k)+2n(k))}{69+31\alpha(k)}\bigg|_{k=3}+\frac{3M}{4}}\\
\shoveleft[5cm]{=L+\frac{161}{1108}+\frac{3M}{4}\approx 1.0089>1,}
\end{multlined}
\end{equation} 
a contradiction. The inequalities in (\ref{2eqnn18}) and (\ref{2eqnn21}) hold for every odd $k>3$. Hence case ($iv$) doesn't hold. This shows that $x_i^*\geq N>\Sigma_1^i(\xs)$ (C) in (\ref{ABC2}) is not the case either. Finally the conclusion of the lemma follows. 
\end{proof}

Similar to \fullref{unique1} the proof of  \fullref{unique5} is symmetric in the sense that it can be reiterated to prove analogous results for the displacement  functions $f_i$ in $\mathcal{F}^k$ for the choices of $i\in I^k=\{1,\dots,4\cdot 3^{k-1}\}$ and $j\in\{1,2,3,4\}$ satisfying the hypothesis of  \fullref{unique5}. In particular we prove the following:   
\begin{lemma}\label{unique6}
Under the hypothesis of  \fullref{unique5}, if $k>3$ is odd, then we have $\xs\in C_{f_i}$, defined in (\ref{C1}), for each of the cases 
\begin{equation}\label{cases1}
\begin{array}{llllll}
j=1, & i\in I_1, &  i\equiv 1, &  i\equiv 2, &  i\equiv 3, & i\equiv 0,\\
j=2, & i\in I_2, &  i\equiv 1, &  i\equiv 2, &  i\equiv 3, & i\equiv 0,\\
j=3, & i\in I_3, &  i\equiv 1, &  i\equiv 2, &  i\equiv 3, & i\equiv 0,\\
j=4, & i\in I_4, &  i\equiv 1, &  i\equiv 2, &  i\equiv 3, & i\equiv 0.
\end{array}
\end{equation}
\end{lemma}
\begin{proof}
We reorganise the inequalities in (\ref{ABC2}), (\ref{sum3}), (\ref{1-7-3}), (\ref{sum5}), (\ref{1-7-4}), (\ref{sum4}) and (\ref{1-7-7}) according to each $j$ and $i$ listed in the lemma. Then we follow the computations carried out in the proof of \fullref{unique5} for the chosen $j$ and $i$.  Using the table in (\ref{kodd}) we carry out the analogs of the  computations given in the proof of  \fullref{unique5} which implies the conclusion. 
\end{proof}
\begin{lemma}\label{unique7}
Under the hypothesis of \fullref{unique5}, if $k>2$ is even, $j=1$ and $i\in I_1$ so that $i\equiv 0$,  then we have $\xs\in C_{f_i}$ defined in (\ref{C1}).
\end{lemma}
\begin{proof}
Because we have  $j=1$, $i\equiv 0$ and $i\in I_1$, we give the same arguments given in the proof of \fullref{unique5} with the same organisations listed in  (\ref{ABC2}), (\ref{sum3}), (\ref{1-7-3}), (\ref{sum5}), (\ref{1-7-4}), (\ref{sum4}) and (\ref{1-7-7}). Because $k>2$ is even, there are changes to be made in the terms of some of the summations. These changes are listed in the table below:

\begin{table}[ht!]
\begin{center}
\scalebox{0.8}{
\begin{tabular}{|r|c|c|c|r|c|c|c|c|}
  \hline
          &   & $k>2$, odd &  $k>2$, even  & & & $k>2$, odd & $k>2$, even  \\
  \hline
    A.i &  (\ref{2eqnn1}) & $2m+7n$ & $6m+3n$ & vii &   (\ref{2eqnn11})  & $2m+2n$  &   $2m+2n$   \\ 
  \hline 
     ii & (\ref{2eqnn2}) & $2m+2n$ &  $2m+2n$ &   viii &                          &  $4n$        & $2m+2n$   \\ 
  \hline
    iii & (\ref{2eqnn3}) & $4n$ & $2m+2n$   & 	ix &		 & $4n$ & $2n+2m$ \\ 
  \hline
    iv  &  & $4n$                        & $2m+2n$ & vii & (\ref{2eqnn12}) &  $2m+4n$, $3n$     &  $4m+2n$, $2m+n$ \\ 
  \hline
    v  & (\ref{2eqnn5}) & $n$                               & $n$ & viii & 	       & $5n+m$, $m+2n$ &$4m+2n$, $2m+n$ \\ 
  \hline
    vi  &                      & $n$                        & $n$ & ix &	 & $5n+m$, $m+2n$   & $4m+2n$, $2m+n$  \\
  \hline
    vii  &             & $n$  & $n$   &  x &(\ref{2eqnn13})      & $n$           & $n$\\ 
  \hline
         & (\ref{2eqnn6_1})  & $2n$     & $2m$   & xi &  & $n$ & $n$ \\ 
   \hline
          & (\ref{2eqnn6_2})  & $m+n$     & $2m$   & xii &  & $n$ & $n$ \\ 
   \hline
          & (\ref{2eqnn6_3})  & $m+n$     & $2m$  & x & (\ref{2eqnn16})  & $n,\ m+n$            & $n,\ 2m$   \\ 
   \hline
   B.i & (\ref{2eqnn7})  & $2m+2n$     & $2m+2n$   &  xi   &                        & $n,\ m+n$            & $n,\ 2m$   \\ 
  \hline
   ii & (\ref{2eqnn8})  & $4n$                               & $2m+2n$   &  xii  &                      & $n,\ 2n$            & $n,\ 2m$  \\ 
  \hline
   iii   &                           & $4n$                                & $2m+2n$   & C.i & (\ref{2eqnn17}) & $2n$ & $m+n$ \\ 
  \hline
   iv & (\ref{2eqnn9})  & $n$                               & $n$   &  ii & & $2n$  & $m+n$ \\ 
  \hline
   v &    &   $n$         &  $n$  & iii &  & $m+n$ & $m+n$  \\ 
  \hline
   vi &      &  $n$      &  $n$   & i & (\ref{2eqnn20}) & $2m+2n,\ 2n$ & $4m,\ n+m$ \\ 
  \hline
    iv& (\ref{2eqnn10})  & $n,\ n+m$    & $n,\ 2m$   & ii  &  &  $2m+2n,\ 2n$ & $4m,\ n+m$ \\ 
  \hline
      v &                      & $n,\ n+m$   & $n,\ 2m$  & iii & & $4n,\ m+n$ & $4m,\ m+n$ \\ 
  \hline
    vi &	   & $n,\ 2n$ & $n,\ 2m$   & iv &  (\ref{2eqnn21}) & $m+2n$ & $3m$ \\ 
   \hline
 \end{tabular}}
\caption{List of changes to be made in the proof of  \fullref{unique5}.} \label{table6}
\end{center}
\end{table}

In each of the inequalities and computations, from left to right, we replace the terms given under the column '$k>2$, odd' with the terms given under the column '$k>2$, even' for indicated equations, where $m=m(k)$ and $n=n(k)$ are defined in (\ref{mn}). All of the resulting inequalities are still satisfied which proves the lemma. 
\end{proof}

\begin{lemma}\label{unique8}
Under the hypothesis of  \fullref{unique5}, if $k>2$ is even, then we have $\xs\in C_{f_i}$, defined in (\ref{C2}), for each of the cases in (\ref{cases1}).
\end{lemma}
\begin{proof}
Given a pair of $j$ and $i$ listed in the lemma, we reorganise the inequalities in  (\ref{ABC2}), (\ref{sum3}), (\ref{1-7-3}), (\ref{sum3}), (\ref{1-7-4}), (\ref{sum4}) and (\ref{1-7-7}) accordingly.  By using the terms listed under the columns  '$k>2$, even' in  \fullref{table6} for the indicated equations, we repeat the arguments presented in the proof of  \fullref{unique1} for the chosen $j$ and $i$. We get the conclusion of the lemma. 
\end{proof}

\begin{prop}\label{uniquef}
Let $\mathcal{F}^k=\{f_i\}$ for $i\in I^k=\{1,\dots,4\cdot 3^{k-1}\}$ be the set of displacement functions listed in  \fullref{dispfunc} and $F^k$ be as in (\ref{F}) for $k\geq 2$ and $d=4\cdot 3^{k-1}$.  Let $\xs$ be a point in $\Delta^{d-1}$ so that $  \alpha_*=F^k(\xs)$. Then $\xs\in\cap_{i=1}^{d} C_{f_i}$. 
\end{prop}
\begin{proof}
The statement follows from  \fullref{unique1}, \ref{unique2}, \ref{unique3}, \ref{unique4}, \ref{unique5}, \ref{unique6}, \ref{unique7} and \ref{unique8}.
\end{proof}

At this point we review three more facts from convex analysis that we shall need. Proofs of these statements are relatively elementary. Therefore they are omitted. Interested readers may again refer to \cite[Theorem 2.1, Theorem 5.5]{RTRoc} and \cite[Proposition 5.4.1]{KL}:  
\begin{thm}\label{convex4}
If $\{C_i\}$ for $i\in I$ is a collection of finitely many nonempty convex sets in $\mathbb{R}^d$ with $C=\displaystyle\cap_{i\in I} C_i\neq\emptyset$, then $C$ is also convex.
\end{thm}
\begin{thm}\label{convex5}
If $\{f_i\}$ for $i\in I$ is a finite set of strictly convex functions defined on a convex set $C\subset\mathbb{R}^d$, then $\max_{\tb{x}\in C}\{f_i(\tb{x})\co i\in I\}$ is also strictly convex on $C$. 
\end{thm}
\begin{prop}\label{convex6}
Let $F$ be a convex function on an open convex set $C\subset\mathbb{R}^d$. If $\xs$ is a local minimum of $F$, then it is a global minimum of $F$, and the set $\{\ys\in C\co F(\ys)=F(\xs)\}$ is convex. If $F$ is strictly convex and $\xs$ is a global minimum then the set $\{\ys\in C\co F(\ys)=F(\xs)\}$ consists of $\xs$ alone. 
\end{prop}

An implication of the statements above for the set of displacement functions $\mathcal{F}^k$ is the uniqueness of the point, whose existence is guaranteed by \fullref{lemtwo}, at which $F^k$ takes its minimum value. In other words we prove the following statement:

\begin{lemma}\label{unique9}
Let $\{f_i\}$ for ${i\in I^k}$ be the set of displacement functions listed in  \fullref{dispfunc} and $F^k$ be as in (\ref{F}).  If $\xs$ and $\ys$ are two points in $\Delta^{d-1}$ so that $  \alpha_*=F^k(\xs)=F^k(\ys)$, then $\xs=\ys$. 
\end{lemma}
\begin{proof}
Let $C_{f_i}$ for $i\in I^k$ be the subsets of $\Delta^{d-1}$ as described in (\ref{C1}) and (\ref{C2}). By \fullref{convex2} and \ref{convex1} they are open convex subsets of $\Delta^{d-1}$. Then $\cap_{i\in I^k}C_{f_i}$ is also open and convex by  \fullref{convex4}. Since the displacement functions in $\mathcal{F}^k=\{f_i\}$ for $i\in I^k$ are either of the form $f(\Sigma_j(\tb{x}),x_i)$ or of the form $g(\Sigma_j^i(\tb{x}),x_i)$,  each $f_i$ is a strictly convex function on the open convex set $C_{f_i}$ by \fullref{convex2} and \ref{convex1}. Then  \fullref{convex3} implies that every $f_i$ for $i\in I^k$ is convex on $\cap_{i\in I^k}C_{f_i}$.

Let $F=F^k$ and $C=\cap_{i\in I^k}C_{f_i}$. The conclusion of the lemma follows from  \fullref{convex5}, \fullref{convex6} and  \fullref{uniquef}.
\end{proof}

The uniqueness of $\xs$ given by  \fullref{unique9} reduces the amount of computations necessary to calculate the infimum of the maximum of the functions in $\mathcal{G}^k$ for the decomposition $\Gamma_{\mathcal{D}^k}$ considerably when compared to the number computations given in \cite{Y} to calculate the infimum of the maximum of the functions in $\mathcal{G}^{\dagger}$ for the decomposition $\Gamma_{\mathcal{D}^{\dagger}}$ (see \cite[Section 4.3]{Y}). We prove the statements below:

\begin{thm}\label{minF}
Let $F^k\co\Delta^{d-1}\to\R$ be defined by $\tb{x}\mapsto\max\{f_i(\tb{x})\co i\in I^k\}$, where $\{f_i\}$ for $i\in I^k$ is the set of functions listed in \fullref{dispfunc} and $d=4\cdot 3^{k-1}$. Then we have  $\inf_{\tb{x}\in\Delta^{d-1}}F^k(\tb{x})=12\cdot 3^{k-1}-3$ for $k\geq 2$. 
\end{thm}
\begin{proof}
Let $\xs=(x_1^*,x_2^*,\dots,x_d^*)\in\Delta^{d-1}$ be a point at which $F^k$ takes its minimum value $  \alpha_*$.  Assume that $k=2$. Consider the cycles $\tau_1=(1\ 12)(2\ 10)(3\ 11)(4\ 5)(8\ 9)$, $\tau_2=(1\ 9)(2\ 8)(3\ 7)(4\ 6)(10\ 12)$ and $\tau_3=(1\ 5)(2\ 6)(3\ 4)(7\ 8)(11\ 12)$ in the symmetric group $S_{12}$. Note that $\tau_1(I_1)=I_4$, $\tau_1(I_l)=I_l$ for $l=2,3$, $\tau_2(I_1)=I_3$, $\tau_2(I_l)=I_l$ for $l=2,4$ and, $\tau_3(I_1)=I_2$, $\tau_3(I_l)=I_l$ for $l=3,4$. 

Let $T_l\co\Delta^{11}\to\Delta^{11}$ be the transformation with the formula $x_i\mapsto x_{\tau_l(i)}$ for $l=1,2,3$. Clearly we have $T_l(\Delta^{11})=\Delta^{11}$ for any $l$. Let $H_l\co\Delta^{11}\to\mathbb{R}$ be the function so that $H_l(\tb{x})=\max\{(f_i\circ T_l)(\tb{x})\co i=1,2,\dots, 12\}$. Since $f_i(T_l(\tb{x}))=f_{\tau_l(i)}(\tb{x})$ for every $i=1,2,\dots 12$ for every $\tb{x}\in\Delta^{11}$ for every $l$ (see the formulas in (\ref{deffs})), we derive that $F^2(\tb{x})=H_l(\tb{x})$ for every $\tb{x}\in\Delta^{11}$ for every $l$. We know by   \fullref{uniquef} that $\xs$ is unique, ie $T_l^{-1}(\xs)=\xs$ for $l=1,2,3$. 

For $l=1$, we find that $x_1^*=x_{12}^*$, $x_2^*=x_{10}^*$, $x_3^*=x_{11}^*$, $x_4^*=x_5^*$, $x_8^*=x_9^*$. For $l=2,3$ we have $x_1^*=x_5^*=x_9^*$,  $x_2^*=x_3^*=x_4^*=x_6^*=x_7^*=x_8^*$, $x_{10}^*=x_{11}^*=x_{12}^*$ which implies that $x_i^*=x_j^*=1/12$ for every $i,j\in I^2=\{1,2,\dots,12\}$. Then we compute that $F^2(\xs)=  \alpha_*=33$. This proves the conclusion of the theorem for $k=2$. 

In the rest of the proof two cases will be considered: $k>2$ is even or  $k$ is odd. In each case maps analogous to $T_l$ and $H_l$ used above are required. Since their definitions will be similar to $T_l$ and $H_l$ with appropriate dimension changes, we shall not state their formulas explicitly to save space. By abusing the notation for both $\tau_l$ and $T_l$, for a fixed index  $T_l$ will be used to denote all transformations defined by $\tau_l$. Since we use the equivalence in modulo $4$  only, we shall express $a \bmod 4\equiv b$ with $a\equiv b$.

Assume that $k$ is even and $k>2$. Remember that there are $m=\lceil 3^{k-1}/4\rceil$ many elements in $I_1$ which are equivalent to $1$ in modulo $4$. The same is true for the number of elements equivalent to $2$ or $3$. But there are $n=\lfloor 3^{k-1}/4\rfloor$ many elements in $I_1$ which are equivalent to $0$ in modulo $4$. For $I_2$, $I_3$ and $I_4$ we have the table in (\ref{keven}).

Let $S_d$ denote the symmetric group. For the group of first four sets we assume that $i\in I_1$ and $j\in I_4$. For $A_5$ we assume that $i,j\in I_2$ and, for $A_6$ we assume $i,j\in I_3$. Define the following sets of transpositions in $S_d$:
\[ 
\begin{array}{c}
A_1  =  \{(i,j)\co i\equiv 1,\ j\equiv 0\},\ 
A_2  =  \{(i,j)\co i\equiv 2,\ j\equiv 2\},\\
A_3  =  \{(i,j)\co i\equiv 3,\ j\equiv 3\},\ 
A_4  =  \{(i,j)\co i\equiv 0,\ j\equiv 1\},\\
A_5  =  \{(i,j)\co i\equiv 0,\ j\equiv 1,\ i\neq j\},\ 
A_6  =  \{(i,j)\co i\equiv 0,\ j\equiv 1,\ i\neq j\}.
\end{array}
\]
Let $\mathcal{A}_1$ be the set of cycles so that each cycle is formed by the multiplications of $m$ transpositions in $A_1$ whose first entries are in increasing order. Define $\mathcal{A}_2$, $\mathcal{A}_3$,  $\mathcal{A}_5$ and $\mathcal{A}_6$ in the same way. Similarly, let $\mathcal{A}_4$ be the set of cycles formed by the multiplication of $n$ transpositions in $A_4$ whose first entries are in increasing order. Also let
\[
\begin{array}{lll}
\mathcal{A}_7  &= & \{(i_1i_2\cdots i_m)\co i_1,i_2,\dots,i_m\equiv 2,\ i_1,i_2,\dots,i_m\in I_2\},\\
\mathcal{A}_8  &=  &\{(i_1i_2\cdots i_n)\co i_1,i_2,\dots,i_n\equiv 2,\ i_1,i_2,\dots,i_n\in I_3\},\\
\mathcal{A}_9  &=  &\{(i_1i_2\cdots i_n)\co i_1,i_2,\dots,i_n\equiv 3,\ i_1,i_2,\dots,i_n\in I_2\},\\
\mathcal{A}_{10} &= & \{(i_1i_2\cdots i_m)\co i_1,i_2,\dots,i_m\equiv 3,\ i_1,i_2,\dots,i_m\in I_3\}.
\end{array}
\]
Choose one cycle from each set $\mathcal{A}_1$, $\mathcal{A}_2$,\dots, $\mathcal{A}_{10}$. Consider the multiplication of all  of these $10$ disjoint cycles. Let $\Theta_1$ be the set of all cycles obtained this way. For any element of $\Theta_1$, denote it by $\tau_1$, we have $\tau_1(I_1)=I_4$, $\tau_1(I_2)=I_2$ and $\tau_1(I_3)=I_3$. 

Let $\Theta_2$ be the set of cycles formed by the same process given above using the following sets of transpositions and cycles $(i_1,i_2,\dots,i_m)$ and $(i_1,i_2,\dots,i_n)$ in $S_d$. Assume for the first four sets that $i\in I_1$ and $j\in I_3$. The entries for the cycles $(i_1,i_2,\dots,i_m)$ and $(i_1,i_2,\dots,i_n)$ are given by the  group of last four sets:
\[
\begin{array}{c}
\{(i,j)\co i \equiv 1,\ j \equiv 1\},\  
\{(i,j)\co i \equiv 2,\ j \equiv 0,\},\
\{(i,j)\co i \equiv 3,\ j \equiv 3\},\\ 
\{(i,j)\co i \equiv 0,\ j \equiv 2\},\ 
\{(i,j)\co i \equiv 0,\ j \equiv 2,\ i,\ j\in I_2,\ i\neq j\},\\ 
\{(i,j)\co i \equiv 0,\ j \equiv 2,\ i,\ j\in I_4,\ i\neq j\},\\
\{i_l \equiv 1,\ i_l\in I_2,\ l=1,\dots,m\},\
\{i_l \equiv 1,\ i_l\in I_4,\ l=1,\dots,n\},\\ 
\{i_l \equiv 3,\ i_l\in I_2,\ l=1,\dots,n\},\ 
\{i_l \equiv 3,\ i_l\in I_4,\ l=1,\dots,m\}.
\end{array}
\]
For any element of $\Theta_2$, denote it by $\tau_2$, we see that $\tau_2(I_1)=I_3$, $\tau_2(I_2)=I_2$ and $\tau_2(I_4)=I_4$. 

Finally let $\Theta_3$ be the set of cycles obtained by the same method used above for $\Theta_1$ and $\Theta_2$. This time we use the transpositions and cycles $(i_1,i_2,\dots,i_m)$ and $(i_1,i_2,\dots,i_n)$  below. Assume for the group of first four sets that $i\in I_1$ and $j\in I_2$.  For the cycles $(i_1,i_2,\dots,i_m)$ and $(i_1,i_2,\dots,i_n)$ entries are given by the group of last four sets below: 
\[
\begin{array}{c}
\{(i,j)\co i\equiv 1,\ j \equiv 1\},\  
 \{(i,j)\co i \equiv 2,\ j \equiv 2\},\ 
\{(i,j)\co i \equiv 3,\ j \equiv 0\},\\
\{(i,j)\co i \equiv 0,\ j \equiv 3\},\ 
\{(i,j)\co i \equiv 0,\ j \equiv 3,\ i,\ j\in I_3,\ i\neq j\},\\
\{(i,j)\co i \equiv 0,\ j \equiv 3,\ i,\ j\in I_4,\ i\neq j\},\\
\{i_l  \equiv 1,\ i_l\in I_3,\ l=1,\dots,m\},\ 
\{i_l \equiv 1,\ i_l\in I_4,\ l=1,\dots,n\},\\
\{i_l \equiv 2,\ i_l\in I_3,\ l=1,\dots,n\},\ 
\{i_l \equiv 2,\ i_l\in I_4,\ l=1,\dots,m\}.
\end{array}
\]
For any element of $\Theta_3$, denote it by $\tau_3$, we observe that $\tau_3(I_1)=I_2$, $\tau_3(I_3)=I_3$ and $\tau_3(I_4)=I_4$.

By  \fullref{uniquef}, we have $T_l^{-1}(\xs)=\xs$ for  every $\tau_1\in\Theta_1$, $\tau_2\in\Theta_2$ and $\tau_3\in\Theta_3$. Therefore for $i\in I_1$ for the first four sets, we conclude that $x_i^*=x_j^*$ for each of the following cases separately
 \begin{equation}\label{first}
\left\{ \begin{array}{l}
i \equiv 1,\\ j \equiv 0\ (j\in I_4),\\ j \equiv 1\ (j\in I_2\cup I_3)\end{array}\right\},
\left\{\begin{array}{l}
i \equiv 2,\\ j \equiv 0\ (j\in I_3),\\ j \equiv 2\ (j\in I_2\cup I_4)\end{array}\right\}, \left\{\begin{array}{l}
i \equiv 0,\\ j \equiv 1\ (j\in I_4),\\ j \equiv 2\ (j\in I_3),\\ j \equiv 3\ (j\in I_2)\end{array}\right\},
\end{equation}
\begin{equation}
\left\{\begin{array}{l}
i \equiv 3,\\ j \equiv 0\ (j\in I_2),\\ j \equiv 3\ (j\in I_3\cup I_4)\end{array}\right\},
\left\{\begin{array}{l}
i \equiv 0,\ j \equiv 1\tnr{ or } 2 \ (i,\ j\in I_2,\ i\neq j),\\
i \equiv 0,\ j \equiv 1\tnr{ or } 3 \ (i,\ j\in I_3,\ i\neq j),\\
i \equiv 0,\ j \equiv 2\tnr{ or } 3 \ (i,\ j\in I_4,\ i\neq j)
\end{array}\right\}.
\end{equation}
Similarly we have the equalities of entries $x_i^*=x_j^*$ for each of the cases listed below:
\begin{equation}
\left\{\begin{array}{l}
i \equiv 0,\ j \equiv 1\tnr{ or } 2 \ (i,\ j\in I_2,\ i\neq j),\\
i \equiv 0,\ j \equiv 1\tnr{ or } 3 \ (i,\ j\in I_3,\ i\neq j),\\
i \equiv 0,\ j \equiv 2\tnr{ or } 3 \ (i,\ j\in I_4,\ i\neq j)
\end{array}\right\},
\end{equation}
\begin{equation}\label{last}
\left\{\begin{array}{l}
i,\ j \equiv 1,\ i,\ j \equiv 2,\ \ i,\ j \equiv 3\ (i,\ j\in I_2,\ i\neq j),\\
i,\ j \equiv 1,\ i,\ j \equiv 2,\ \ i,\ j \equiv 3\ (i,\ j\in I_3,\ i\neq j),\\
i,\ j \equiv 1,\ i,\ j \equiv 2,\ \ i,\ j \equiv 3\ (i,\ j\in I_4,\ i\neq j)
\end{array}\right\}.
\end{equation}
We combine the equalities $x_i^*=x_j^*$ for the indices given in (\ref{first})-(\ref{last}). We find that 
\[
\begin{array}{ll}
               x_1^*=x_j^* & j \equiv 0\ (j\in I_2\cup I_3\cup I_4),\\
               x_1^*=x_j^* & j \equiv 1\ (j\in I_1\cup I_2\cup I_3),\\
               x_1^*=x_j^* & j \equiv 2\ (j\in I_1\cup I_2\cup I_4),\\
               x_1^*=x_j^* & j \equiv 3\ (j\in I_1\cup I_3\cup I_4),
\end{array}
\begin{array}{ll}
                x_4^*=x_j^* & j \equiv 0\ (j\in I_1),\\
                x_4^*=x_j^* & j \equiv 1\ (j\in I_2),\\
                x_4^*=x_j^* & j \equiv 2\ (j\in I_3),\\
                x_4^*=x_j^* & j \equiv 3\ (j\in I_4).
\end{array}
\]
As a result there are two possible values $\alpha_1$ and $\alpha_2$ for $  \alpha_*=\inf_{\tb{x}\in\Delta^{d-1}}F^k(\tb{x})$, where
\[
\alpha_1=\frac{1-nx_4^*-3mx_1^*}{nx_4^*+3mx_1^*}\cdot\frac{1-x_1^*}{x_1^*}\ \tnr{ and }\  \alpha_2=\frac{1-nx_4^*-3mx_1^*}{nx_4^*+3mx_1^*}\cdot\frac{1-x_4^*}{x_4^*}. 
\]
If $\alpha_1=  \alpha_*>\alpha_2$, we get $x_1^*<x_4^*$. Since $\xs\in\Delta^{d-1}$, we have $nx_4^*+3mx_1^*=1/4$, which implies that $1/x_1^*-1>4(n+3m)-1$. Then we see that $$\alpha_1>12(n+3m)-3\geq 12\cdot 3^{k-1}-3,$$ where $n=\lfloor 3^{k-1}/4\rfloor$ and $m=\lceil 3^{k-1}/4\rceil$. This is a contradiction by \fullref{lemtwo}. By symmetry the inequality $\alpha_1<\alpha_2$ also gives a contradiction. So we derive that $\alpha_1=\alpha_2$ or $x_1^*=x_4^*$ which shows that $x_i=x_j=1/d$ for every $i,j\in I^k$ and $d=4\cdot 3^{k-1}$. Hence the conclusion of the theorem follows in this case.

Assume that $k>2$ is odd. In this case there are $m=\lceil 3^{k-1}/4\rceil$ many elements in $I_1$ which are equivalent to $1$ in modulo $4$. There are $n=\lfloor 3^{k-1}/4\rfloor$ many elements each in $I_1$ which are equivalent to $2$, $3$ or $0$ in modulo $4$. In other words we obtain the list $(m,n,n,n)$ for the number of elements  which are equivalent to $1$, $2$, $3$ or $0$, respectively. In $I_2$, $I_3$ and $I_4$ we have the lists in the table (\ref{kodd}).

We shall use the same sets $\mathcal{A}_1$, $\mathcal{A}_2$, \dots, $\mathcal{A}_{10}$ of cycles defined above for the even $k$ case, by switching the roles of $m$ and $n$ if necessary, to construct $\Theta_1$ the set of cycles formed by the multiplication of cycles chosen one from each set $\mathcal{A}_1$, $\mathcal{A}_2$, \dots, $\mathcal{A}_{10}$. So for any $\tau_1\in\Theta_1$ we have $\tau_1(I_1)=I_4$, $\tau_1(I_2)=I_2$ and $\tau_1(I_3)=I_3$. 

Define $\Theta_2$ by using the transpositions and cycles $(i_1,i_2,\dots,i_m)$ and $(i_1,i_2,\dots,i_n)$ listed below. Assume for the group of first four sets that $i\in I_2$ and $j\in I_3$.  For the cycles $(i_1,i_2,\dots,i_m)$ and $(i_1,i_2,\dots,i_n)$ entries are given by the group of second four sets: 
\[
\begin{array}{c}
\{(i,j)\co i \equiv 1,\ j \equiv 1\},\   
\{(i,j)\co i \equiv 2,\ j \equiv 3\},\
\{(i,j)\co i \equiv 3,\ j \equiv 2\},\\  
\{(i,j)\co i \equiv 0,\ j \equiv 0\},\ 
\{(i,j)\co i \equiv 2,\ j \equiv 3,\ i,\ j\in I_1,\ i\neq j\},\\ 
\{(i,j)\co i \equiv 2,\ j \equiv 3,\ i,\ j\in I_4,\ i\neq j\},\\
\{i_l \equiv 1,\ i_l\in I_1,\ l=1,\dots,m\},\ 
\{i_l \equiv 1,\ i_l\in I_4,\ l=1,\dots,n\},\\ 
\{i_l \equiv 0,\ i_l\in I_1,\ l=1,\dots,n\},\
\{i_l \equiv 0,\ i_l\in I_4,\ l=1,\dots,m\}.
\end{array}
\]
Then for any element of $\tau_2\in\Theta_2$ we see that $\tau_2(I_2)=I_3$, $\tau_2(I_1)=I_1$ and $\tau_2(I_4)=I_4$. 

For $\Theta_3$ we shall use the sets of transpositions described below. For the group of first four sets we assume that $i\in I_1$ and $j\in I_3$. For the group of second four sets we assume that $i\in I_2$ and $j\in I_4$. Let 
\[
\begin{array}{c}
B_1=\{(i,j)\co i \equiv 1,\ j \equiv 3\},
B_2=\{(i,j)\co i \equiv 2,\ j \equiv 0\},\\
B_3=\{(i,j)\co i \equiv 3,\ j \equiv 1\},
B_4=\{(i,j)\co i \equiv 0,\ j \equiv 2\},\\
B_5=\{(i,j)\co i \equiv 1,\ j \equiv 3\},
B_6=\{(i,j)\co i \equiv 2,\ j \equiv 0\},\\
B_7=\{(i,j)\co i \equiv 3,\ j \equiv 1\},
B_8=\{(i,j)\co i \equiv 0,\ j \equiv 2\}.
\end{array}
\]
Let $\mathcal{B}_1$ and $\mathcal{B}_6$ be the sets of cycles so that each cycle in each set is formed by the multiplications of $m$ transpositions in $B_1$ and $B_6$, respectively, whose first entries are in increasing order. Similarly, let $\mathcal{B}_2$, $\mathcal{B}_3$, $\mathcal{B}_4$, $\mathcal{B}_5$, $\mathcal{B}_7$ and, $\mathcal{B}_8$  be the set of cycles formed by the multiplication of $n$ transpositions in $B_2$, $B_3$, $B_4$, $B_5$, $B_7$ and $B_8$, respectively, whose first entries are in increasing order. Choose one cycle from each set $\mathcal{B}_1$,\dots, $\mathcal{B}_8$. Consider the multiplication of all of these $8$ cycles. Let $\Theta_3$ be the set of all these disjoint cycles. Then for any element of $\tau_3\in\Theta_3$ we have $\tau_3(I_1)=I_3$ and $\tau_3(I_2)=I_4$. 

By analogous definitions for $T_l$ and $H_l$ with appropriate dimensions, we derive by  \fullref{uniquef} that  $T_l^{-1}(\xs)=\xs$ for  every $\tau_1\in\Theta_1$, $\tau_2\in\Theta_2$ and $\tau_3\in\Theta_3$. For $i\in I_1$ this implies the equalities $x_i^*=x_j^*$ for the following indices:
 \begin{equation}\label{third}
 \left\{\begin{array}{l}
i \equiv 1,\\ j \equiv 2\ (j\in I_2),\\ j \equiv 3\ (j\in I_3),\\ j \equiv 0\ (j\in I_4).\end{array}\right\},
\left\{\begin{array}{l}
i \equiv 2,3,\\ j \equiv 0,1,\ (j\in I_2)\\ j \equiv 0,1,\ (j\in I_3)\\ j \equiv 2,3,\ (j\in I_4).
\end{array}\right\},
\left\{\begin{array}{l}
i \equiv 0,\\ j \equiv 3,\ (j\in I_2)\\ j \equiv 2,\ (j\in I_3)\\ j \equiv 1,\ (j\in I_4).\end{array}\right\}.
\end{equation}
If we combine all of the equalities $x_i^*=x_j^*$ in (\ref{third}), we find that 
\[
\begin{array}{ll}
               x_1^*=x_j^*  & j \equiv 1\ (j\in I_1),\\
               x_1^*=x_j^*  & j \equiv 2\ (j\in I_2),\\
               x_1^*=x_j^*  & j \equiv 3\ (j\in I_3),\\
               x_1^*=x_j^*  & j \equiv 0\ (j\in I_4),
\end{array}
\begin{array}{ll}
               x_4^*=x_j^* & j \equiv 1\ (j\in I_4),\\
                x_4^*=x_j^* & j \equiv 2\ (j\in I_3),\\
                x_4^*=x_j^* & j \equiv 3\ (j\in I_2),\\
                x_4^*=x_j^* & j \equiv 0\ (j\in I_1),
\end{array}
\begin{array}{ll}
                x_2^*=x_j^*  & j \equiv 1\ (j\in I_2\cup I_3),\\
               x_2^*=x_j^*  & j \equiv 2\ (j\in I_1\cup I_4),\\
               x_2^*=x_j^*  & j \equiv 3\ (j\in I_1\cup I_4),\\
               x_2^*=x_j^*  & j \equiv 0\ (j\in I_2\cup I_3).
\end{array}
\]
This means that there are three possible values $\alpha_1$, $\alpha_2$ and $\alpha_3$ for $  \alpha_*$ at $\xs$, where 
\[
\alpha_1=\frac{1-mx_1^*-2nx_2^*-nx_4^*}{mx_1^*+2nx_2^*+nx_4^*}\cdot\frac{1-x_1^*}{x_1^*},\ \alpha_2=\frac{1-mx_1^*-2nx_2^*-nx_4^*}{mx_1^*+2nx_2^*+nx_4^*}\cdot\frac{1-x_2^*}{x_2^*},
\]
\[
\alpha_3=\frac{1-mx_1^*-2nx_2^*-nx_4^*}{mx_1^*+2nx_2^*+nx_4^*}\cdot\frac{1-x_4^*}{x_4^*}.
\]
Assume that $\alpha_1=  \alpha_*>\alpha_2\geq\alpha_3$. Then we conclude that $x_1^*<x_2^*\leq x_4^*$. Since $\xs\in\Delta^{d-1}$, we have the equality $mx_1^*+2nx_2^*+nx_4^*=1/4$, which implies that $1/x_1^*-1>4(m+3n)-1$. Then we find that 
$
\alpha_1>12(m+3n)-3\geq 12\cdot 3^{k-1}-3.
$
This is a contradiction by \fullref{lemtwo}. Because of symmetry we obtain a contradiction in any case unless $\alpha_1=\alpha_2=\alpha_3$, which implies that $x_1^*=x_2^*=x_4^*$.  In other words, we get $x_i^*=x_j^*=1/d$ for every $i,j\in I^k$ and $d=4\cdot 3^{k-1}$. An elementary computation verifies the conclusion of the theorem in this case as well.
\end{proof}

\begin{thm}\label{minG}
Let $G^k\co\Delta^{d-1}\to\mathbb{R}$ be defined by $\tb{x}\mapsto\max\{f(\tb{x})\co f\in\mathcal{G}^k\}$, where $\mathcal{G}^k$ is the set of functions in  \fullref{dispfunc}. Then $\inf_{\tb{x}\in\Delta^{d-1}}G^k(\tb{x})=12\cdot 3^{k-1}-3$.
\end{thm}
\begin{proof}
The displacement functions $g_i^{k,1}(\tb{x})$ for $i\in I^k$ are produced by the group--theoretical relations $(\gamma,s(\gamma),S(\gamma))$ of $\Gamma=\langle\xi,\eta\rangle$ with $length(\gamma s(\gamma))=1$ (see  \fullref{lem22} and  \fullref{dispfunc}). Therefore, $S(\gamma)$ contains $3\cdot 3^{k-1}$ many isometries. Since $g_i^{k,1}(\tb{x})=\sigma\left(\sum_{\psi\in S(\gamma)}x_{p(\psi)}\right)\sigma(x_i)$, where $p$ is the mapping defined in  (\ref{Enum}), we calculate that $g_i^{k,1}(\xs)=(4\cdot 3^{k-1}-1)/3<  \alpha_*$ for every $i\in I^k$. 

The functions in the union $\{g_{i,1}^{k,2},\dots,g_{i,a_2}^{k,2}\}\cup\{g_{i,1}^{k,3},\dots,g_{i,a_3}^{k,3}\}\cup\dots\cup\{g_{i,1}^{k,k},\dots,g_{i,a_k}^{k,k}\}$ are produced by the relations $(\gamma,s(\gamma),S(\gamma))$ so that $2\leq length(\gamma s(\gamma))=m\leq k$. For each group of functions in the union above $S(\gamma)$ contains $4\cdot 3^{k-1}-3^{k-m}$ many isometries, respectively. This implies that the sums in the formulas of these functions contain  $4\cdot 3^{k-1}-3^{k-m}$ many summands. Then we see that $G^k(\xs)=F^k(\xs)$ because, by direct calculations we have $g_{i,1}^{k,m}(\xs)=\cdots=g_{i,a_m}^{k,m}(\xs)<  \alpha_*$ for every $m=2,\dots,k$.
Since $\mathcal{F}^k\subset\mathcal{G}^k$, we have $G^k(\tb{x})\geq F^k(\tb{x})$ for every $\tb{x}\in\Delta^{d-1}$. Hence, the conclusion of the theorem follows.
\end{proof}



\section{Proof of The Main Theorem}\label{sec4}

{Finally we present a proof of the main result of this paper. Although the proof goes along the same lines as the proof of \cite[Theorem 5.1]{Y}, we include the details for the sake of completeness. } 
\begin{thm}\label{thm4.1}
Let $\xi$ and $\eta$ be two non--commuting isometries of $\hyp$. Suppose that $\xi$ and $\eta$ generate a torsion--free discrete group $\Gamma$ which is not
co--compact and contains no parabolic. Let $\Gamma_{k}$ and $\alpha_{k}$ denote the set of  isometries of length at most $k\geq2$ in $\Gamma=\langle\xi,\eta\rangle$ and the real number $12\cdot 3^{k-1}-3$, respectively. Then for any $z_0\in\hyp$ we have
$$
e^{\left(\dis{2\max\nolimits_{\gamma\in\Gamma_{k}}\left\{\dgamma\right\}}\right)}\geq\alpha_{k}.
$$
\end{thm}
\begin{proof}
We have the following two cases: (\ref{I}) $\Gamma=\langle\xi,\eta\rangle$ is geometrically infinite, or   (\ref{II}) $\Gamma=\langle\xi,\eta\rangle$ is geometrically finite. Assume that the prior is the case. 

We know by \cite[Proposition 9.2]{CSParadox} that $\Gamma=\langle\xi,\eta\rangle$ is a free
group on the generators $\xi$ and $\eta$. Then it can be decomposed as in (\ref{symm.d.}). Let $\Gamma_{\mathcal{D}^k}$ be the symmetric decomposition of $\Gamma=\langle\xi,\eta\rangle$ so that $\mathcal{D}^k=(\Psi^k,\Psi^k_r)$, where $\Gamma_k=\Psi^k\cup\Psi^k_r$. Since $\Gamma=\langle\xi,\eta\rangle$ is geometrically infinite,  \fullref{dispfunc} and  \fullref{minG} imply the conclusion of the theorem in this case:
$$
\max_{\gamma\in\Gamma_{k}}\left\{\dgamma\right\} \geq  \frac{1}{2}\log G^{k}(\tb{m})
                         \geq  \frac{1}{2}\log\left(\inf_{\tb{x}\in\Delta^{d-1}} G^{k}(\tb{x})\right)=\frac{1}{2}\log\alpha_k.
$$
Above $\tb{m}=\left(m_{p(\psi)}\right)_{\psi\in\Psi^{k}}\in\Delta^{d-1}$, where $p$ and $m_{p(\psi)}$ are the bijection and the total measures defined in (\ref{Enum}) and  \fullref{dispfunc}, respectively. The function $G^{k}$ is defined in (\ref{G}).

Assume that $\Gamma=\langle\xi,\eta\rangle$ is geometrically finite. Let $\mathfrak{X}$ denote the character variety $PSL_2(\C)\times PSL_2(\C)\simeq\tnr{Isom}^+(\hyp)\times\tnr{Isom}^+(\hyp)$.  Let  $\mathfrak{GF}$ be the open subset of $\mathfrak{X}$, consisting of $(\xi,\eta)$ such that $\langle\xi,\eta\rangle$ is free, geometrically finite and without any parabolic.  Then $(\xi,\eta)$ is in $\mathfrak{GF}$. We define the function $f_{z_0}^k\co\mathfrak{X}\to\R$ such that
\begin{displaymath}
f_{z_0}^k(\xi,\eta)=\max_{\psi\in\Psi^k}\{\dpsi\}
\end{displaymath}
for a fixed $z_0\in\hyp$. The function $f_{z_0}^k$ is continuous and proper. Therefore, it takes a minimum value at some point $(\xi_0,\eta_0)\in\overline{\mathfrak{GF}}$. We claim that $(\xi_0,\eta_0)$ is in $\overline{\mathfrak{GF}}-\mathfrak{GF}$.

Assume on the contrary that $(\xi_0,\eta_0)\in\mathfrak{GF}$. Since $\Gamma=\langle\xi,\eta\rangle$ is torsion-free, each isometry $\gamma\in\Gamma_k$ has infinite order. This implies that $\gamma\cdot z\neq z$ for every $z\in\hyp$. In particular, we get $\gamma\cdot z_0\neq z_0$ for any $\gamma\in\Gamma_k$. Therefore, there
exists hyperbolic geodesic segments joining $z_0$ to $\gamma\cdot z_0$ for every $\gamma\in\Psi^k_r$. Note that, since we have
$\tr{dist}(z_0,\gamma_1\gamma_2\cdot z_0)=\tr{dist}(\gamma_1^{-1}\cdot z_0,\gamma_2\cdot z_0)$ and $\dgamma=\dgammainv$, all of the hyperbolic displacements under the isometries in $\Gamma_k$ are realised by the geodesic line segments 
joining the points $\{z_0\}\cup\{\gamma\cdot z_0\co\gamma\in\Psi^k_r\}$. 

Let us enumerate the elements of $\Psi_r^k$ for some index set $I$ in $\mathbb{N}$. Let $P_0=z_0$ and $P_i=\gamma_i\cdot z_0$ for every $i\in I$.  Let $\Delta_{ij}=\triangle{P_iP_0P_j}$ denote the geodesic triangle with vertices $P_i$, $P_0$ and $P_j$. The value $f_{z_0}^{k}(\xi_0,\eta_0)$ is the unique longest side length of $\Delta_{ij}$ for some $i,j\in I$. We shall denote these geodesic triangles with $\widetilde{\Delta}_{ij}$ and their vertices by $\widetilde{P}_i$, $P_0$ and $\widetilde{P}_j$. There are two cases to consider: (1) all of $\widetilde{\Delta}_{ij}$ are acute or (2) there exists at least one $\widetilde{\Delta}_{ij}$ which is not acute.

Assume that the latter (2) is the case (In the rest of the argument we shall use figures from $k=2$ case for illustrations). 
Fix one of the non-acute geodesic triangles $\widetilde{\Delta}_{ij}$ and denote it by $\Delta$. If $\widetilde{P}_i$ lies in $\gamma$, we let $P_i^{(l)}$ be a sequence of points in the interior of $\gamma$ so that $P_i^{(l)}\to\widetilde{P}_i$. Let $\gamma$ denote the longest edge of $\Delta$. By the hyperbolic law of sines, $\gamma$ is opposite to the non--acute angle. Let $P_j^{(l)}=\widetilde{P}_j$ and $P_0^{(l)}=P_0$ for every $l\in\mathbb{N}$. Otherwise, we let $P_j^{(l)}$ be a sequence of points in the interior of $\gamma$ so that $P_j^{(l)}\to\widetilde{P}_j$ and define $P_i^{(l)}=\widetilde{P}_i$ and $P_0^{(l)}=P_0$ for every $l\in\mathbb{N}$. 
\begin{figure}[h!]
\centering 
\includegraphics[scale=.65]{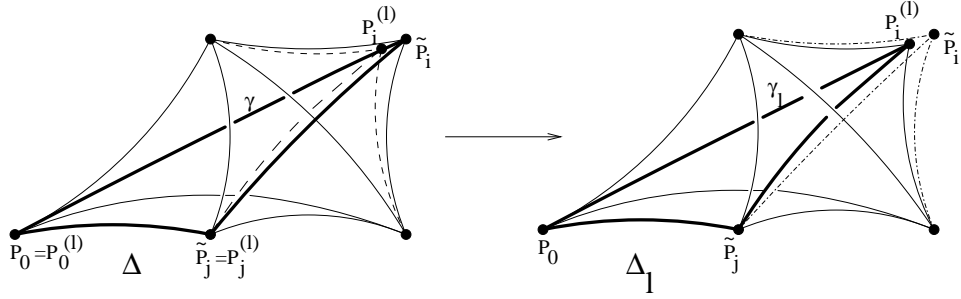}
\caption{Moving along $\gamma$ in case (2).}
\label{fig1}
\end{figure}

Let $\Delta_l$ be the geodesic triangle contained in $\Delta$ with vertices $P_0^{(l)}$, $P_i^{(l)}$ and $P_j^{(l)}$. By the construction, the unique longest side $\gamma_l$ of $\Delta_l$ is contained in $\gamma$ for all but finitely many $l$. Let $\{\xi_l\}$ be a sequence of isometries such that $\xi_l\to\xi_0$ and $\xi_l^{-1}\cdot z_0=P_i^{(l)}$. Similarly, let $\{\eta_l\}$ be a sequence of isometries such that $\eta_l\to\eta_0$ and $\eta_l\cdot z_0=P_j^{(l)}$. Then we have $(\xi_l,\eta_l)\in\mathfrak{GF}$ for all but finitely many $l$ and $f_{z_0}^{k}(\xi_l,\eta_l)=length(\gamma_l)<f_{z_0}^{k}(\xi_0,\eta_0)$, a contradiction.

Assume that all of $\Delta_{ij}$ are acute (1). Fix one of $\Delta_{ij}$ and call it $\Delta$. Then the perpendicular arc $\gamma_i$ from $\widetilde{P}_i$ to the geodesic containing $P_0$ and $\widetilde{P}_j$ meets it in the interior of the edge of $\Delta$ opposite to $\widetilde{P}_i$. Let $P_i^{(l)}$ be a sequence of points in the interior of $\gamma_i$ so that $P_i^{(l)}\to\widetilde{P}_i$. For each $l$, we see that
\begin{figure}[h!]
\centering
\includegraphics[scale=.65]{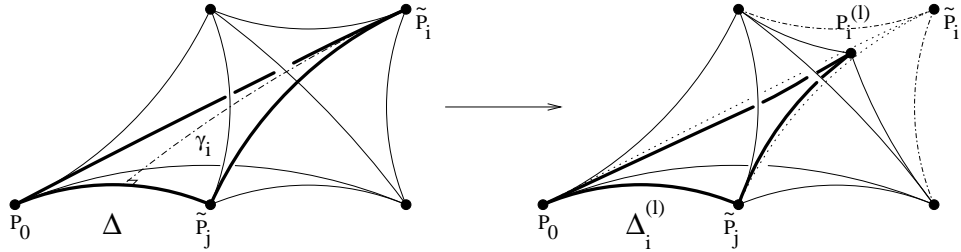}
\caption{Moving along $\gamma_i$ in case (1).}
\label{fig2}
\end{figure}
$d(P_i^{(l)},P_0)<d(\widetilde{P}_i,P_0)$ by applying the hyperbolic law of cosines to the right triangle containing $P_i^{(l)}$, $P_0$ and a sub--arc of $\gamma_i$. Similarly, we have $d(P_i^{(l)},\widetilde{P}_j)<d(\widetilde{P}_i,\widetilde{P}_j)$.

The geodesic triangle $\Delta_i^{(l)}$ with vertices $P_0$, $P_i^{(l)}$ and $\widetilde{P}_j$ is itself acute. This is because its angles at $P_0$ and $\widetilde{P}_j$ are less than those of $\Delta$. Also the angle of $\Delta$ at $\widetilde{P}_i$ is the limit of the angles at $P_i^{(l)}$. This implies that the perpendicular arc $\gamma_i^{(l)}$ from $\widetilde{P}_j$ to the geodesic containing $P_0$ and $P_i^{(l)}$ meets this geodesic inside of $\Delta_i^{(l)}$. 
\begin{figure}
\centering
\includegraphics[scale=.65]{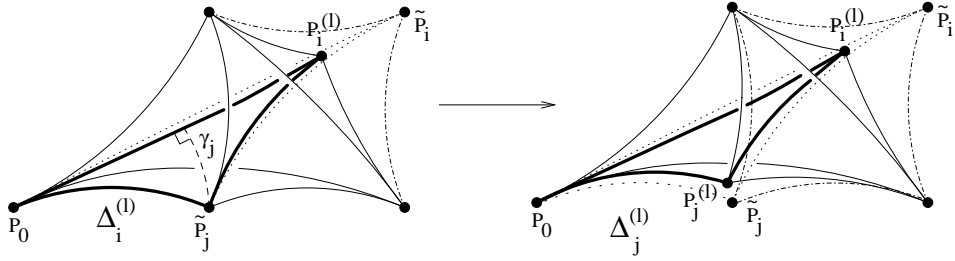}
\caption{Moving along $\gamma_j$ in case (1).}
\label{fig3}
\end{figure}
Let $P_j^{(l)}$ be the point on $\gamma_j^{(l)}$ at distance $1/l$ from $\widetilde{P}_j$. We find that $d(P_j^{(l)},P_0)<d(\widetilde{P}_j,P_0)$ and $d(P_j^{(l)},P_i^{(l)})<d(\widetilde{P}_j,P_i^{(l)})<d(\widetilde{P}_j,\widetilde{P}_i)$ by the hyperbolic law of cosines. As a result we obtain a triangle with vertices at $P_0$, $P_i^{(l)}$ and $P_j^{(l)}$ so that all edge lengths are less than those of $\Delta$. Let $\{\xi_l\}$ and $\{\eta_l\}$ be the sequences such that $\xi_l^{-1}\cdot z_0=P_i^{(l)}$ and $\eta_l\cdot z_0=P_j^{(l)}$. Then we have $f_{z_0}^{k}(\xi_l,\eta_l)<f_{z_0}^{k}(\xi_0,\eta_0)$ for all but finitely many $l$, a contradiction. So the claim is proved.

We know that the set of $(\xi,\eta)$ such that $\langle\xi,\eta\rangle$ is free, geometricly infinite and without parabolics is dense in $\overline{\mathfrak{GF}}-\mathfrak{GF}$. We also know that  every $(\xi,\eta)\in\mathfrak{X}$ with  $\langle\xi,\eta\rangle$ is free and without parabolic is in $\overline{\mathfrak{GF}}$ \cite[Propositions 9.3 and 8.2]{CSParadox}. This reduces geometrically finite case to geometrically infinite case. Finally, the conclusion of the theorem follows from the fact that $(\xi_0,\eta_0)\in\mathfrak{GF}$.
\end{proof} 

Notice that all of the arguments used in this paper to prove  \fullref{thm4.1} can be carried out in a more general setting; in particular in the case $\Gamma=\langle\xi_1,\dots,\xi_n\rangle$ is a purely loxodromic, finitely generated free Kleinian group for $n\geq 2$. In fact we can propose  
\begin{conjecture}
If $\Gamma_{k,n}$ is the set of all isometries of length at most $k\geq 2$ in $\Gamma$, then 
 $\max_{\gamma\in\Gamma_{k,n}}\{\textnormal{dist}(z_0,\gamma\cdot z_0)\}\geq 0.5\log((2n-1)(2n(2n-1)^{k-1}-1))$  for any $z_0\in\hyp$.
\end{conjecture}
We conclude this paper with a proof sketch for this conjecture. Details of the arguments outlined below will be left to the future studies. 

We consider the cases in (\ref{I}) and (\ref{II}). In the case $\Gamma=\langle\xi,\dots,\xi_n\rangle$ is geometrically infinite, we use symmetric decomposition $\Gamma_{\mathcal{D}^{k,n}}$ of $\Gamma$, where $\mathcal{D}^{k,n}=(\Psi^{k,n},\Psi^{k,n}_r)$ is defined in (\fullref{fgsymd}). Above $\Psi^{k,n}$ is the set of words of length $k$ and $\Psi^{k,n}_r$ is the set of words of length less than $k$.  

Let $d=2n\cdot (2n-1)^{k-1}$ and $R_{k,n}=k+(2n-2)\sum_{l=1}^{k-1}\sum_{s=0}^{\min\{l,k-l\}}(2n-1)^{s-1}$. It is possible to prove an analog of \fullref{lem22} stating that there are $d\cdot R_{k,n}$ many group--theoretical relations for the decomposition $\Gamma_{\mathcal{D}^{k,n}}$. Using these group--theoretical relations an analog of  \fullref{thm3.4} can be stated. This gives the decomposition of the area measure $A_{z_0}$ corresponding to the symmetric decomposition $\Gamma_{\mathcal{D}^{k,n}}$ of $\Gamma$. Then using  \fullref{lem1.2} we prove an analog of  \fullref{dispfunc} which provides a set $\mathcal{G}^{k,n}$  of  $d\cdot R_{k,n}$ many displacement functions so that only a set $\mathcal{F}^{k,n}$ of $d$ many of which are significant to compute the infimum of the maximum of the functions in $\mathcal{G}^{k,n}$ on the simplex $\Delta^{d-1}$. 

As in  \fullref{minF} and \ref{minG} the lower bounds proposed in the conjecture are a consequence of the uniqueness of the point $\xs\in\Delta^{d-1}$ at which the infimum of the maximum of the displacement functions in $\mathcal{F}^{k,n}$ is attained. The uniqueness of $\xs$ is implied by a statement  similar to \fullref{uniquef} stating that there exists a strictly convex set $C$ in $\Delta^{d-1}$ containing $\xs$ such that each displacement function in $\mathcal{F}^{k,n}$ is strictly convex on $C$. Since the infimum of the maximum of the functions in  $\mathcal{F}^{k,n}$ is itself convex on $C$, the uniqueness of $\xs$ follows from some standard facts in convex analysis. Using all of the bijections of $\Delta^{d-1}$ fixing the set $\mathcal{F}^{n,k}$ we derive that all of the coordinates of $\xs$ are equal. Then a simple computation gives the lower bounds in the conjecture completing the proof in the case (\ref{I}).

In the case (\ref{II}) $\Gamma=\langle\xi_1,\dots,\xi_n\rangle$ is geometrically finite, the assertion of the conjecture can be proved along the same lines as in the proof of  \fullref{thm4.1}.

%
%
%
%

\end{document}